\documentclass[american,11pt]{amsbook}

\usepackage[english]{babel}
\usepackage{a4wide}
\usepackage[utf8]{inputenc}
\usepackage{graphicx}
\usepackage{caption}
\usepackage{subcaption}
\usepackage{color}
\usepackage{amssymb}

\usepackage{comment}
\excludecomment{removed}

\usepackage{ifthen}

\DeclareMathOperator{\coker}{coker}

\DeclareMathOperator{\diverg}{div}
\DeclareMathOperator{\End}{End}
\DeclareMathOperator{\ev}{ev}
\DeclareMathOperator{\GL}{GL}

\DeclareMathOperator{\Hom}{Hom}

\DeclareMathOperator{\ImaginaryPart}{Im}
\DeclareMathOperator{\ind}{ind}
\DeclareMathOperator{\rank}{rank}
\DeclareMathOperator{\RealPart}{Re}
\DeclareMathOperator{\sign}{sign}
\DeclareMathOperator{\sing}{Sing}

\newcommand{\conSum}{{\mathbin{\#}}}
\newcommand{\restricted}[2]{{\left.{#1}\right|_{#2}}}
\newcommand{\lcan}{{\lambda_{\mathrm{can}}}}

\newcommand{\p}{\partial}
\newcommand{\lie}[1]{{\mathcal{L}_{#1}}}
\newcommand{\abs}[1]{{\left\lvert #1\right\rvert}}

\newcommand{\0}{{\mathbf{0}}}

\renewcommand{\epsilon}{\varepsilon}

\usepackage{xspace}
\newcommand{\BLOB}{\textsf{bLob}\xspace}
\makeatother
\newcommand{\LOB}{\textsf{Lob}\xspace}

\renewcommand{\AA}{{\mathbb{A}}}
\newcommand{\bB}{{\mathcal{B}}}
\newcommand{\CC}{{\mathbb{C}}}
\newcommand{\CP}{{\CC P}}
\newcommand{\DD}{{\mathbb D}}
\newcommand{\eE}{{\mathcal{E}}}
\newcommand{\HH}{{\mathbb{H}}}
\newcommand{\NN}{{\mathbb{N}}}
\newcommand{\RR}{{\mathbb{R}}}
\newcommand{\QQ}{{\mathbb{Q}}}
\newcommand{\SSS}{{\mathbb{S}}}
\newcommand{\TT}{{\mathbb{T}}}
\newcommand{\ZZ}{{\mathbb{Z}}}
\newcommand{\fF}{{\mathcal F}}

\newcommand{\bfP}{{\mathbf{p}}}
\newcommand{\bfQ}{{\mathbf{q}}}

\newcommand{\bfx}{{\mathbf{x}}}
\newcommand{\bfy}{{\mathbf{y}}}

\newcommand{\defin}[1]{\textbf{#1}}
\newcommand{\mM}{{\mathcal M}}

\newcounter{CobStep}
\theoremstyle{plain}

\newcounter{maintheorem}

\newtheorem{main_theorem}[maintheorem]{Theorem}
\newtheorem*{non_fillability_theorem}{Theorem~\ref{thm:
non-fillability}}
\newtheorem*{homology_LOB_theorem}{Theorem~\ref{thm: homology LOB}}

\newtheorem{theorem}{Theorem}[section]

\newtheorem{lemma}[theorem]{Lemma}

\newtheorem{corollary}[theorem]{Corollary}

\newtheorem{proposition}[theorem]{Proposition}

\theoremstyle{remark}
\newtheorem{remark}[theorem]{Remark}
\newtheorem{example}[theorem]{Example}

\theoremstyle{definition}
\newtheorem*{definition}{Definition}

\numberwithin{equation}{section}

\makeindex

\title{Higher dimensional contact topology via holomorphic disks}

\begin{document}

\maketitle

\section{Introduction}

In '85 Gromov published his article on pseudo-holomorphic curves
\cite{Gromov_HolCurves} that made symplectic topology as we know it
today only possible.
Using these techniques, Gromov presented in his initial paper many
spectacular results, and soon many other people started using these
methods to settle questions that before had been out of reach
\cite{Eliashberg_filling, McDuffRuled, McDuff_contactType,
  HoferWeinstein, Eliashberg3Torus, AbreuSympGroup} and many others;
for more recent results in this vein we refer to
\cite{WendlGirouxTorsion, OanceaViterbo}.
While the references above rely on studying the topology of the moduli
space itself, Gromov's $J$-holomorphic methods have also been used to
develop powerful algebraic theories like Floer Homology, Gromov-Witten
Theory, Symplectic Field Theory, Fukaya Theory etc. that basically
rely on counting rigid holomorphic curves (that means holomorphic
curves that are isolated).
Note though that we will completely ignore such algebraic techniques
in these notes.
Gromov's approach for studying a symplectic manifold~$(W,\omega)$
consists in choosing an \emph{auxiliary} almost complex structure~$J$
on $W$ that is compatible with $\omega$ in a certain way.
This auxiliary structure allows us to study so called $J$-holomorphic
curves, that means, equivalence classes of maps
\begin{equation*}
  u\colon (\Sigma, j) \to (W,J)
\end{equation*}
from a Riemann surface $(\Sigma, j)$ to $W$ whose differential at
every point $x\in \Sigma$ is a $(j,J)$-complex map
\begin{equation*}
  Du_x \colon T_x\Sigma \to T_{u(x)} W \;.
\end{equation*}
Conceivable generalizations of such a theory based on studying
$J$-holomorphic surfaces or even higher dimensional $J$-complex
manifolds only work for \emph{integrable} complex structures;
otherwise generically such submanifolds do not exist.
A different approach has been developed by Donaldson
\cite{DonaldsonSympSubmanifolds, DonaldsonLefschetzPencil}, and
consists in studying approximately holomorphic sections in a line
bundle over $W$.
This theory yields many important results, but has a very different
flavor than the one discussed here by Gromov.
The $J$-holomorphic curves are relatively rare and usually come in
finite dimensional families.
Technical problems aside, one tries to understand the symplectic
manifold~$(W, \omega)$ by studying how these curves move through $W$.
Let us illustrate this strategy with the well-known example of
$\CP^n$.
We know that there is exactly one complex line through any two points
of $\CP^n$.
We fix a point $z_0\in \CP^n$, and study the space of all holomorphic
lines going through $z_0$.
It follows directly that $\CP^n \setminus \{z_0\}$ is foliated by
these holomorphic lines, and every line with $z_0$ removed is a disk.
Using that the lines are parametrized by the corresponding complex
line in $T_{z_0}\CP^n$ that is tangent to them, we see that the space
of holomorphic lines is diffeomorphic to $\CP^{n-1}$, and that
$\CP^n\setminus \{z_0\}$ will be a disk bundle over $\CP^{n-1}$.
In this example, we have used an ambient manifold that we understand
rather well, $\CP^n$, to compute the topology of the space of complex
lines.
So far, it might seem unclear how one could obtain information about
the topology of the space of complex lines in an ambient space that we
do not understand equally well, to then extract in a second step
missing information about the ambient manifold.
The common strategy is to assume that the almost complex manifold we
want to study already contains a family of holomorphic curves.
We then observe how this family evolves, hoping that it will
eventually ``fill up'' the entire symplectic manifold (or produce
other interesting effects).
To briefly sketch the type of arguments used in general, consider now
a symplectic manifold~$W$ with a compatible almost complex structure,
and suppose that it contains an open subset $U$ diffeomorphic to a
neighborhood of $\CP^1 \times \{0\}$ in $\CP^1 \times \CC$ (see
\cite{McDuffRuled}).
In this neighborhood we find a family of holomorphic spheres $\CP^1
\times \{z\}$ parametrized by the points~$z$.
We can explicitly write down the holomorphic spheres that lie
completely inside $U$, but Gromov compactness tells us that as the
holomorphic curves approach the boundary of $U$, they cannot just
cease to exist but instead there is a well understood way in which
they can degenerate, which is called \emph{bubbling}.
Bubbling means that a family of holomorphic curves decomposes in the
limit into several smaller ones.
Sometimes bubbling can be controlled or even excluded by imposing
technical conditions, and in this case, the limit curve will just be a
regular holomorphic curve.
In the example we were sketching above, this means that if no bubbling
can happen, there will be regular holomorphic spheres (partially)
outside $U$ that are obtained by pushing the given ones towards the
boundary of $U$.
This limit curve is also part of the $2$-parameter space of spheres,
and thus it will be surrounded by other holomorphic spheres of the
same family.
As long as we do not have any bubbling, we can thus extend the family
by pushing the spheres to the limit and then obtain a new regular
sphere, which again is surrounded by other holomorphic spheres.
This way, we can eventually show that the whole symplectic manifold is
filled up by a $2$-dimensional family of holomorphic spheres.
Furthermore the holomorphic spheres do not intersect each other (in
dimension~$4$), and this way we obtain a $2$-sphere fibration of the
symplectic manifold.
In conclusion, we obtain in this example just from the existence of
the chart~$U$, and the conditions that exclude bubbling that the
symplectic manifold needs to be a $2$-sphere bundle over a compact
surface (the space of spheres).
Note that many arguments in the example above (in particular the idea
that the moduli spaces foliate the ambient manifold) do not hold in
general, that means for generic almost complex structures in manifolds
of dimension more than $4$.
Either one needs to weaken the desired statements or find suitable
work arounds.
The principle that is universal is the use of a well understood local
model in which we can detect a family of holomorphic curves.
If bubbling can be excluded, this family extends into the unknown
parts of the symplectic manifold, and can be used to understand
certain topological properties of this manifold.

\vspace{0.4cm}

These notes are based on a course that took place at the Université de
Nantes in June 2011 during the \emph{Trimester on Contact and
  Symplectic Topology}.
We will explain how holomorphic curves can be used to study symplectic
fillings of a given contact manifold.
Our main goal consists in showing that certain contact manifolds do
not admit any symplectic filling at all.
Since closed symplectic manifolds are usually studied using closed
holomorphic curves, it is natural to study symplectic fillings by
using holomorphic curves with boundary.
We will explain how the existence of so called \emph{Legendrian open
  books} (\LOB{}s) and \emph{bordered Legendrian open books}
(\BLOB{}s) controls the behavior of holomorphic disks, and what
properties we can deduce from families of such disks.
The notions are direct generalizations of the overtwisted disk
\cite{Gromov_HolCurves, Eliashberg_filling} and standardly embedded
$2$-spheres in a contact $3$-manifold \cite{BedfordGaveau,
  Gromov_HolCurves, HoferWeinstein}.
For completeness, we would like to mention that symplectic fillings
have also been studied successfully via punctured holomorphic curves
whose behavior is linked to Reeb orbit dynamics, and via closed
holomorphic curves by first capping off the symplectic filling to
create a closed symplectic manifold.

\subsection{Outline of the notes}

In the first part of these notes we will talk about Legendrian
foliations, and in particular about \LOB{}s and \BLOB{}s.
We will not consider any holomorphic curves here, but the main aim
will be instead to illustrate examples where these objects can be
localized.
In the second part, we study the properties of holomorphic disks
imposed by Legendrian foliations and convex boundaries.
In the last chapter, we use this information to understand moduli
spaces of holomorphic disks obtained from a \LOB{} or a \BLOB{}, and
we prove some basic results about symplectic fillings.
The content of these notes are based on an unfinished manuscript of
\cite{NiederkrugerHabilitation}.

\subsection{Notation}

We assume throughout a certain working knowledge on contact topology
(for a reference see for example
\cite[Chapter~3.4]{McDuffSalamonIntro} and \cite{Geiges_book}) and on
holomorphic curves \cite{AudinLafontaine, McDuffSalamonJHolo}.
The contact structures we consider in this text are always
\emph{cooriented}.
Remember that by choice of a coorientation, $(M, \xi)$ always obtains
a natural orientation and its contact structure~$\xi$ carries a
natural \emph{conformal} symplectic structure.
For both, it suffices to choose a \defin{positive} contact
form~$\alpha$, that means, a $1$-form with $\xi = \ker \alpha$ that
evaluates positively on vectors that are positively transverse to the
contact structure.
The orientation on $M$ is then given by the volume form
\begin{equation*}
  \alpha \wedge d\alpha^n \;,
\end{equation*}
where $\dim M = 2n + 1$, while the conformal symplectic structure is
represented by $\restricted{d\alpha}{\xi}$.
One can easily check that these notions are well-defined by choosing
any other positive contact form~$\alpha'$ so that there exists a
smooth function~$f\colon M \to \RR$ such that $\alpha' = e^f\,
\alpha$.

\subsection*{Further conventions}

Note that $\DD^2$ denotes in this text the \emph{closed} unit disk.
I owe it to Patrick Massot to have been converted to the following
\emph{jargon}.

\begin{definition}
  The term \defin{regular equation} \index{equation!regular} can refer
  in this text to any of the following objects:
  \begin{enumerate}
  \item When $\Sigma$ is a cooriented hypersurface in a manifold~$M$,
    then we call a smooth function $h\colon M \to \RR$ a
    \defin{regular equation for $\Sigma$}, if $0$ is a regular value
    of $h$ and $h^{-1}(0) = \Sigma$.
  \item When $\mathcal{D}\le TM$ is a singular codimension~$1$
    distribution, then we say that a $1$-form $\beta$ is a
    \defin{regular equation for $\mathcal{D}$}, if $\mathcal{D} = \ker
    \beta$ and if $d\beta \ne 0$ at singular points of $\mathcal{D}$.
  \end{enumerate}
\end{definition}
According to this definition, an equation of a contact structure is
just a contact form.

\begin{removed}
Let $E$ be a vector bundle over a manifold~$M$, and assume that $E$ is
equipped with a metric~$g$.
We use the notation
\begin{equation*}
  E_{<R} = \bigl\{v\in E\bigm|\, g(v,v) < R^2 \bigr\}\;.
\end{equation*}
for the \defin{open disk bundle of size $R$ in $E$}, and use similar
subscripts to refer to other subsets of the bundle~$E$.
\end{removed}

\begin{removed}

\subsection*{Commonly used results}

We will often use in this text the following standard result (see for
example \cite[Lemma~2.1]{Milnor_MorseTheory}).

\begin{proposition}\label{prop: Milnor Taylor formula}
  Let $U$ be an open convex neighborhood of $0$ in $\RR^n$.
  Every smooth function $f\colon U \to \RR$ may be written in the form
  \begin{equation*}
    f(x_1,\dotsc,x_n) = f(0,\dotsc,0)
    + \sum_{j=1}^n x_j\, g_j(x_1,\dotsc,x_n) \;,
  \end{equation*}
  where the $g_1,\dotsc, g_n\colon U \to \RR$ are smooth functions
  that satisfy
  \begin{equation*}
    g_j(0,\dotsc,0) = \frac{\partial f}{\partial x_j} (0,\dotsc,0) \;.
  \end{equation*}
\end{proposition}
\end{removed}

\chapter[(bordered) Legendrian open books]{\LOB{}s \& \BLOB{}s:
  Legendrian open books and bordered Legendrian open books}

\section{Legendrian foliations}

\subsection{General facts about Legendrian foliations}
\label{sec: introduction legendrian foliations}

Let $(M,\xi)$ be a contact manifold that contains a submanifold~$N$.
Generically, if we look at any point $p\in N$ the intersection between
$\xi_p$ and the tangent space $T_pN$ will be a codimension~$1$
hyperplane.
Globally though, the distribution $\mathcal{D} = \xi\cap TN$ may be
singular, because there can be points~$p\in N$ where $T_pN \subset
\xi_p$, and equally important the distribution~$\mathcal{D}$ will only
be in very rare cases a foliation.
In fact, if we choose a contact form~$\alpha$ for $\xi$, then we
obtain by the Frobenius theorem that $\mathcal{D}$ will only be a
(singular) foliation if
\begin{equation*}
  \restricted{\bigl(\alpha\wedge d\alpha\bigr)}{TN} \equiv 0 \;.
\end{equation*}
Another way to state this condition is to say that we have
$\restricted{d\alpha}{\mathcal{D}_p} = 0$ at every regular point~$p\in
N$ of $\mathcal{D}$, so that $\mathcal{D}_p$ has to be an isotropic
subspace of $(\xi_p, d\alpha_p)$.
In particular, this shows that the induced distribution~$\mathcal{D}$
can never be integrable if $\dim \mathcal{D} > \frac{1}{2}\, \dim
\xi$.
We will usually denote the distribution~$\xi\cap TN$ by $\fF$ whenever
it is a singular foliation.
Furthermore, we will call such an $\fF$ a \defin{Legendrian foliation}
if $\dim \fF = \frac{1}{2}\, \dim \xi$, which implies that $N$ has to
be a submanifold of dimension $n+1$ if the dimension of the ambient
contact manifold is $2n+1$. \index{Legendrian foliation}
For reasons that we will briefly sketch below, but that will be
treated extensively from Chapter~\ref{chapter: J-holomorphic curves}
on, we will be mostly interested in submanifolds carrying such a
Legendrian foliation.
Note in particular that in a contact $3$-manifold every
hypersurface~$N$ carries automatically a Legendrian foliation.
Denote the set of points $p\in N$ where $\fF$ is singular by
$\sing(\fF)$.
One of the basic properties of a Legendrian foliation is that for any
contact form~$\alpha$, the restriction $\restricted{d\alpha}{TN}$ does
not vanish on $\sing(\fF)$, because otherwise $T_pN \subset \xi_p$
would be an isotropic subspace of $(\xi_p, d\alpha_p)$ which is
impossible for dimensional reasons.
Since $\restricted{d\alpha}{TN}$ does not vanish on $\sing(\fF)$, we
deduce in particular that $N\setminus \sing(\fF)$ is a dense and open
subset of~$N$.
\textbf{The main reason, why we are interested in submanifolds that
  have a Legendrian foliation is that they often allow us to
  successfully use $J$-holomorphic curve techniques.}
On one side, such submanifolds will be automatically totally real for
any suitable almost complex structure on a symplectic filling, thus
posing a good boundary condition for the Cauchy-Riemann equation:
The solution space of a Cauchy-Riemann equation with totally real
boundary condition is often a finite dimensional smooth manifold, so
that it follows that the moduli spaces of $J$-holomorphic curves whose
boundaries lie in a submanifold with a Legendrian foliation will have
a nice local structure.
A second important property is that the topology of the Legendrian
foliation controls the behavior of $J$-holomorphic curves, and will
allow us to obtain many results in contact and symplectic topology.
Elliptic codimension~$2$ singularities of the Legendrian foliation
``emit'' families of holomorphic disks; suitable codimension~$1$
singularities form ``walls'' that cannot be crossed by holomorphic
disks.
In the rest of this section, we will state some general properties of
Legendrian foliations.
Theorem~\ref{theorem: Legendrian foliation determines contact germ}
shows that a manifold with a Legendrian foliation determines the germ
of the contact structure on its neighborhood.
This allows us to describe small deformations of the Legendrian
foliation, and study almost complex structures more explicitly (see
Section~\ref{sec: local_model} in Chapter~\ref{chapter: J-holomorphic
  curves}).
Theorem~\ref{theorem: realization of Legendrian foliations} gives a
precise characterization of the foliations that can be realized as
Legendrian ones.

\subsection{Singular codimension~$1$ foliations}\label{sec: sing
  codim1 foliation}

The principal aim of this section will be to explain the following
result due to Kupka \cite{KupkaSingFoliation} that tells us that the
behavior of a Legendrian foliation close to a singular point can
always be reduced to the $2$-dimensional situation (see
Fig.~\ref{fig:kupka chart}). \index{Kupka chart}

\begin{theorem}\label{theorem: Kupka_s_theorem}
  Let $N$ be a manifold with a singular foliation~$\fF$ that admits a
  regular equation~$\beta$.
  Then we find around any $p\in\sing(\fF)$ a chart with coordinates
  $(s,t,x_1, \dotsc, x_{n-1})$, such that $\beta$ is represented by
  the $1$-form
  \begin{equation*}
    a(s,t)\, ds  +  b(s,t)\, dt
  \end{equation*}
  for smooth functions~$a$ and $b$.
\end{theorem}

We will call any chart of $N$ of the form described in the theorem a
\defin{Kupka chart}.
Note that the foliation in a Kupka chart restricts on every
$2$-dimensional slice $\{(x_1, \dotsc, x_{n-1})= \mathrm{const}\}$ to
one that does not have any isochore singularities (a term introduced
in \cite{Giroux_91}).

\begin{proof}
  From the Frobenius condition $\beta\wedge d\beta \equiv 0$, it
  follows that $d\beta^2 = 0$, so that if $\dim N > 2$, there is a
  non-vanishing vector field~$X$ on a neighborhood of $p$ with
  $d\beta(X, \cdot) = 0$.
  We can also easily see that $X \in \ker \beta$ and $\lie{X} \beta =
  0$, because
  \begin{equation*}
    0 = \iota_X \bigl(\beta \wedge d\beta\bigr) =
    \beta(X)\,  d\beta - \beta \wedge \bigl(\iota_X
    d\beta\bigr) = \beta(X)\,  d\beta\;,
  \end{equation*}
  and $d\beta$ does not vanish on a neighborhood of~$p$.
  Let $\Phi^X_t$ be the flow of $X$, and choose a small
  hypersurface~$\Sigma$ transverse to~$X$.
  Using the diffeomorphism
  \begin{equation*}
    \Psi\colon \Sigma\times (-\epsilon, \epsilon) \to N, \,
    (p,t) \mapsto \Phi_t^X(p)
  \end{equation*}
  we can pull back the $1$-form~$\beta$ to $\Sigma\times
  (-\epsilon,\epsilon)$ and we see it reduces to
  $\restricted{\beta}{T\Sigma}$.
  By repeating this construction the necessary number of times we
  obtain the desired statement.
\end{proof}

\begin{figure}[htbp]
  \centering
  \includegraphics[height=4.5cm,keepaspectratio]{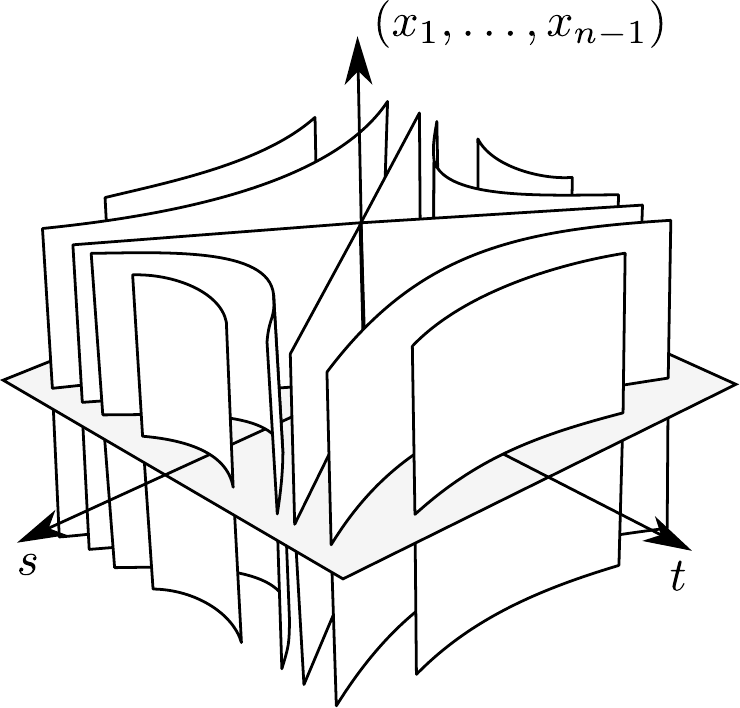}
  \caption{The singularities of a Legendrian foliation look locally
    like the product of $\RR^{n-1}$ with a foliation in the
    plane.}\label{fig:kupka chart}
\end{figure}

\subsection{Local behavior of Legendrian foliations}

We state the following two theorems without proof, and point the
interested reader to \cite{NiederkrugerHabilitation} for more details.
The situation in Section~\ref{sec: codim 1 sing} is treated in these
notes in full completeness to illustrate the flavor of the necessary
methods.
The first result tells us that a Legendrian foliation determines the
germ of the contact structure in its neighborhood.

\begin{theorem}\label{theorem: Legendrian foliation determines contact
    germ}
  Let $N$ be a compact manifold (possibly with boundary) and let
  $(M_1,\xi_1)$ and $(M_2,\xi_2)$ be contact manifolds.
  Assume that two embeddings $\iota_1\colon N \hookrightarrow M_1$ and
  $\iota_2\colon N \hookrightarrow M_2$ are given such that $\xi_1$
  and $\xi_2$ induce on $N$ the same cooriented Legendrian
  foliation~$\fF$.
  Then we find neighborhoods $U_1\subset M_1$ of $\iota_1(N)$ and
  $U_2\subset M_2$ of $\iota_2(N)$ together with a contactomorphism
  \begin{equation*}
    \Phi\colon (U_1,\xi_1) \to (U_2,\xi_2)
  \end{equation*}
  that preserves $N$, that means, $\Phi\circ \iota_1 = \iota_2$.
\end{theorem}

Another useful fact is the following theorem that tells us that the
singular foliations that can be realized as Legendrian ones are
exactly those that admit a regular equation (using the convention from
the introduction).
This result generalizes the $3$-dimensional situation
\cite{Giroux_91}, where this property was called a foliation without
``\textit{isochore singularities}''.

\begin{theorem}\label{theorem: realization of Legendrian foliations}
  Let $N$ be a manifold with a singular codimension-$1$
  foliation~$\fF$ given by a regular equation~$\beta$.
  Then we can find an (open) cooriented contact manifold~$(M,\xi)$
  that contains $N$ as a submanifold such that $\xi$ induces $\fF$ as
  Legendrian foliation on~$N$.
\end{theorem}

\section{Singularities of the Legendrian foliation}
\label{sec: singularities}

The singular set of a Legendrian foliation~$\fF$ can be extremely
complicated.
We will only discuss briefly a few general properties of such points,
before we specialize all considerations to two simple situations.
Let $N$ have a singular foliation~$\fF$ given by a regular
equation~$\beta$, and let $p\in \sing(\fF)$ be a singular point of
$\fF$.
Choose a Kupka chart~$U$ with coordinates $(s,t,x_1,\dotsc,x_{n-1})$
centered at $p$.
In this chart $\beta$ is represented by
\begin{equation*}
  a(s,t)\, ds + b(s,t)\, dt
\end{equation*}
with two smooth functions $a, b\colon U \to \RR$ that only depend on
the $s$- and $t$-coordinates, and that vanish at the origin.
To understand the shape of the foliation depending on the functions
$a$ and $b$, we might study trajectories of the vector field
\begin{equation*}
  X = b(s,t)\, \frac{\partial}{\partial s}
  - a(s,t)\, \frac{\partial}{\partial t}
\end{equation*}
that spans the intersection of the foliation with the $(s,t)$-slices.
Its divergence $\diverg X = \partial b/\partial s - \partial a
/ \partial t$ does not vanish, since $d\beta \ne 0$.
Up to a genericity condition, we know by the Grobman-Hartman theorem
that the flow of $X$ is $C^0$-equivalent to the flow of its
linearization (see \cite{Palis}).
In dimension~$2$, the Grobman-Hartman theorem even yields a
$C^1$-equivalence, but this does not suffice for our purposes.
For one, we would like to sick to a smooth model for all
singularities, but in fact it even suffices for our goals to only look
at singularities whose leaves are all radial, so we will use below a
more hands-on approach.

\subsection{Elliptic singularities}\label{sec: codim 2 sing}

The first type of singularities we allow for the foliation~$\fF$ on
$N$ are called \defin{elliptic}:
In this case, the point~$p\in \sing(\fF)$ admits a Kupka chart
diffeomorphic to $\RR^2 \times \RR^{n}$ with coordinates
$\{(s,t,x_1,\dotsc,x_n)\}$ in which the foliation is given as the
kernel of the $1$-form
\begin{equation*}
  s\,dt - t\,ds \,
\end{equation*}
that means, the leaves are just the radial rays in each $(s,t)$-slice.
We will always assume that the elliptic singularities of a
foliation~$\fF$ are closed isolated codimension~$2$ submanifolds~$S$
in the interior of $N$ with trivial normal bundle, so that the tubular
neighborhood of $S$ is diffeomorphic to $\DD_\epsilon^2 \times S$.
We assume additionally that the foliation~$\fF$ in this model
neighborhood are the points with constant angular coordinate in the
disk.

\subsection{Singularities of codimension~$1$}
\label{sec: codim 1 sing}

Singular sets of codimension~$1$ are extremely ungeneric, but can be
often found through explicit constructions (as in Example~\ref{expl:
  Legendrian sphere bundle}).
We will show in this section that by slightly deforming the foliated
submanifold one can sometimes modify the foliation in a controlled way
so that the singular set turns into a regular compact leaf (see
Fig.~\ref{fig: deforming ot disk at bdry}).
We will treat this situation in full detail to illustrate what type of
methods are needed for the proofs in this first chapter.

\begin{figure}[htbp]
  \centering
  \includegraphics[height=3.5cm,keepaspectratio]{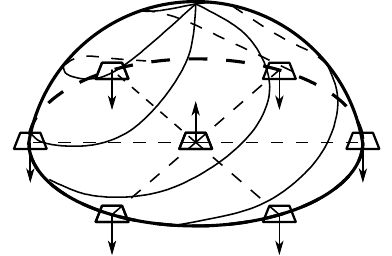}
  \caption{In dimension~$3$ it is well-known that we can get rid of
    $1$-dimensional singular sets of a Legendrian foliation by  slightly
    tilting the surface along the singular set.
    The picture represents how to produce an overtwisted disk whose
    boundary is a compact leaf of the foliation.}\label{fig: deforming
    ot disk at bdry}
\end{figure}

\begin{lemma}\label{lemma: normal form codimension 1 singularity}
  Let $N$ be a compact manifold with a singular codimension~$1$
  foliation~$\fF$ given by a regular equation~$\beta$.
  Assume that the singular set~$\sing(\fF)$ of the foliation contains
  a closed codimension~$1$ submanifold $S\hookrightarrow N$ that is
  cooriented.
  Then we can find a tubular neighborhood of $S$ diffeomorphic to
  $(-\epsilon, \epsilon) \times S$ such that $\beta$ pulls back to
  \begin{equation*}
    s\cdot \widetilde \beta \;,
  \end{equation*}
  where $s$ denotes the coordinate on $(-\epsilon, \epsilon)$, and
  $\widetilde \beta$ is a non-vanishing $1$-form on $S$ that defines a
  regular codimension~$1$ foliation on $S$.
\end{lemma}

\begin{proof}
  Choose a coorientation for $S$.
  We first find a vector field $X$ on a neighborhood of $S$ that is
  transverse to $S$ and lies in the kernel of $\beta$.
  Study the local situation in a Kupka chart~$U$ around a point $p\in
  S$ with coordinates $(s,t,x_1,\dotsc,x_{n-1})$.
  Assume that $\beta$ restricts to
  \begin{equation*}
    a(s,t)\, ds + b(s,t)\, dt \;,
  \end{equation*}
  such that $S\cap U$ corresponds to the subset~$\{s=0\}$, and such
  that $s$ increases in direction of the chosen coorientation.
  Since $a$ and $b$ vanish along $S\cap U$, we may write this form
  also as
  \begin{equation*}
    s\,a_s(s,t)\, ds + s\, b_s(s,t)\, dt =
    s\,\bigl(a_s(s,t)\, ds + b_s(s,t)\, dt\bigr)
  \end{equation*}
  with smooth functions $a_s$ and $b_s$ that satisfy the conditions
  \begin{equation*}
    a_s(0,t) = \frac{\partial a}{\partial s}(0,t)\quad
    \text{ and } \quad
    b_s(0,t) = \frac{\partial b}{\partial s}(0,t) \;.
  \end{equation*}
  The function $b_s$ does not vanish in a small neighborhood of $S\cap
  U$, because $0 \ne d\beta = \partial_s b\, ds\wedge dt$.
  Choose then on the Kupka chart~$U$ the smooth vector field
  \begin{equation*}
    X_U\bigl(s,t,x_1,\dotsc, x_{n-1}\bigr) =
    \partial_s - \frac{a(s,t)}{b(s,t)}\, \partial_t
    = \partial_s - \frac{a_s(s,t)}{b_s(s,t)}\, \partial_t \;.
  \end{equation*}
  This field lies in $\fF$, and is positively transverse to $S\cap U$
  Cover the singular set~$S$ with a finite number of Kupka charts
  $U_1, \dotsc, U_N$, construct vector fields $X_{U_j}$ according to
  the method described above, and glue them together to obtain the
  desired vector field~$X$ by using a partition of unity subordinate
  to the cover.
  We can use the flow of $X$ to obtain a tubular neighborhood of $S$
  that is diffeomorphic to $(-\epsilon,\epsilon) \times S$, where
  $\{0\}\times S$ corresponds to the submanifold~$S$, and $X$
  corresponds to the field~$\partial_s$, where $s$ is the coordinate
  on the interval~$(-\epsilon, \epsilon)$, and since $\beta(X) \equiv
  0$, it follows that $\beta$ does not contain any $ds$-terms.
  Let $\gamma$ be the $1$-form given by $\iota_X d\beta$.
  This form does not vanish on a neighborhood of the singular set~$S$,
  because $d\beta \ne 0$ while $\restricted{\beta}{TS} \equiv 0$, and
  so we can write
  \begin{equation*}
    0 \equiv \iota_X \bigl(\beta\wedge d\beta\bigr)
    = \beta(X)\, d\beta - \beta \wedge \bigl(\iota_X d\beta\bigr) 
    = - \beta\wedge \gamma \;.
  \end{equation*}
  This means that there is a smooth function $F\colon
  (-\epsilon,\epsilon) \times S \to \RR$ with $\restricted{F}{S} = 0$
  such that $\beta = F \gamma$.
  Furthermore, we get that
  \begin{equation*}
    \gamma = \iota_X d\beta =
    dF(X) \, \gamma + F\, \iota_X d\gamma
  \end{equation*}
  does not vanish along~$S$, but $F$ does, so we obtain on $S$ that
  $dF(X) = 1$, and it follows that $S$ is a regular zero level set of
  the function~$F$.
  In fact, we can also easily see from
  \begin{equation*}
    0 \equiv \beta \wedge d\beta = F^2\, \gamma\wedge d\gamma
  \end{equation*}
  that $\gamma\wedge d\gamma$ vanishes everywhere so that $\ker\gamma$
  defines a regular foliation~$\widetilde\fF$ that agrees with the
  initial foliation outside $\sing(\fF)$.
  Finally, we have $\iota_X \gamma \equiv 0$, and using a similar
  argument as before, we see
  \begin{equation*}
    0 \equiv \iota_X \bigl(\gamma\wedge d\gamma\bigr) =
    - \gamma\wedge \iota_X d\gamma
  \end{equation*}
  so that there is a smooth function $f\colon (-\epsilon, \epsilon)
  \times S \to \RR$ such that $\lie{X} \gamma = \iota_X d\gamma = f\,
  \gamma$.
  The flow in $s$-direction possibly rescales the $1$-form~$\gamma$,
  but it leaves its kernel invariant, thus the
  foliation~$\widetilde\fF$ is tangent to the $s$-direction and
  $s$-invariant.
  We can hence represent $\widetilde\fF$ on $(-\epsilon, \epsilon)
  \times S$ as the kernel of the $1$-form $\widetilde\beta =
  \restricted{\gamma}{TS}$ that does not depend on the $s$-coordinate,
  and does not have any $ds$-terms.
  It follows that $\gamma$ is equal to $\widetilde F\,
  \restricted{\gamma}{TS}$ for a function $\widetilde F$ that
  restricts on $S$ to $1$.
  For the initial $1$-form~$\beta$ this means that $\beta = (F
  \widetilde F)\, \widetilde \beta$, and $F \widetilde F$ is a smooth
  function and $\{0\}\times S$ is the (regular) level set of~$0$.
  We can redefine the model $(-\epsilon, \epsilon) \times S$ by using
  the flow of a vector field $G^{-1}\, \partial_s$ with $G=
  \partial_s (F\widetilde F)$ to achieve that $\beta$ reduces on this
  new model to $s\,\widetilde \beta$.
\end{proof}

Suppose from now on that the singular foliation is of the form
described in Lemma~\ref{lemma: normal form codimension 1 singularity},
that means, we have a closed manifold~$S$ with a regular
codimension~$1$ foliation $\fF_S$ given as the kernel of a
$1$-form~$\widetilde \beta$, and $N$ is diffeomorphic to $(-\epsilon,
\epsilon) \times S$ with a singular foliation~$\fF$ given as the
kernel of the $1$-form~$s\,\widetilde\beta$.
Remember that a $1$-form $\sigma$ on $S$ defines a section in $T^*S$
with the property that $\sigma^*\lcan = \sigma$.
We may realize $\fF$ as a Legendrian foliation, by embedding
$(-\epsilon,\epsilon) \times S$ into the $1$-jet space
$\bigl(\RR\times T^*S, dz + \lcan \bigr)$ via the map
\begin{equation*}
  (s,p) \mapsto (0, s\,\widetilde \beta)\;.
\end{equation*}
The foliations agree, and according to Theorem~\ref{theorem:
  Legendrian foliation determines contact germ} this model describes a
small neighborhood of $(N, s\widetilde\beta)$ embedded into an
arbitrary contact manifold.
Assume from now on additionally that $\widetilde \beta$ is a
\emph{closed} $1$-form on $S$ (by a result of Tischler, $S$ fibers
over the circle \cite{TischlerSuspensions}).
Choose a smooth odd function $f\colon (-\epsilon, \epsilon) \to \RR$
with compact support such that the derivative $f'(0) = -1$.
The section
\begin{equation*}
  (-\epsilon,\epsilon) \times S \hookrightarrow  \RR\times T^*S,
  \quad (s,p) \mapsto \bigl(\delta f(s), s\widetilde \beta \bigr)
\end{equation*}
describes for small $\delta > 0$ a $C^\infty$-small deformation of $N$
that agrees away from $S$ with $N$.
The perturbed submanifold~$N'$ also carries a Legendrian foliation
induced by $\ker\bigl(ds + \lcan\bigr)$, because the pull-back form
$\beta' = f' \, ds + s \widetilde \beta$ gives
\begin{equation*}
  \beta' \wedge d\beta' = \bigl(f' \, ds + s \widetilde \beta\bigr)
  \wedge \bigl(ds\wedge \widetilde \beta + s\,d\widetilde \beta\bigr) =
  s\, \bigl( f' \, ds + s\, \widetilde\beta\bigr) \wedge  d \widetilde\beta
  =  s\, f' \, ds \wedge  d\widetilde \beta \;,
\end{equation*}
which vanishes, so that $\beta'$ satisfies the Frobenius condition.
Furthermore, since $\beta'$ itself does not vanish anywhere, it is
easy to check that $\ker \beta'$ defines a regular foliation~$\fF'$,
and that $\{0\}\times S$ is a closed leaf of $\fF'$.
As a conclusion, we obtain

\begin{corollary}
  Let $(M, \xi)$ be a contact manifold containing a submanifold~$N$
  with an induced Legendrian foliation~$\fF$.
  Assume that the singular set of $\fF$ contains a cooriented closed
  codimension~$1$-submanifold $S\subset N$, and that there is a
  regular foliation~$\fF$ that agrees outside $N$ with $\fF$, and that
  corresponds on $S$ with a fibration over the circle.
  Using an arbitrary small $C^\infty$ perturbation of $N$ close to
  $S$, we obtain a new Legendrian foliation for which $S$ has become a
  regular closed leaf.
\end{corollary}

\section{Examples of Legendrian foliations}

The following example relates Legendrian foliations to Lagrangian
submanifolds.
It is not important by itself, but it may help understanding the
construction of the \BLOB{}s in blown down Giroux domains given in
\cite{WeafFillabilityHigherDimension}, and I believe that it might
pave the way to other applications.

\begin{example}
  Let $P$ be a principal circle bundle over a base manifold~$B$, and
  suppose that $\xi$ is a contact structure on $P$ that is transverse
  to the $\SSS^1$-fibers and invariant under the action.
  It is well-known that by averaging, we can choose an
  $\SSS^1$-invariant contact form~$\alpha$ for $\xi$ and that there
  exists a symplectic form~$\omega$ on $B$ such that $\pi^*\omega =
  d\alpha$, where $\pi$ is the bundle projection $\pi\colon P \to B$.
  The symplectic form~$\omega$ represents the image of the Euler
  class~$e(P)$ in $H^2(B,\RR)$, and hence $P$ cannot be a trivial
  bundle (see \cite{BoothbyWang}).
  The manifold~$(P_L, \alpha)$ is usually called the
  \defin{pre-quantization of the symplectic manifold $(B,\omega)$} (or
  the \defin{Boothby-Wang manifold}).
  Let $L$ be a Lagrangian submanifold in $(B,\omega)$, and let $P_L :=
  \pi^{-1}(L)$ be the fibration over $L$.
  Note first that in this situation, we have $\restricted{\omega}{TL}
  = 0$, so that $e(P_L) = \restricted{e(P)}{L}$ will automatically
  either vanish or be a torsion class.
  We assume that $e(P_L) = 0$, so that the fibration~$P_L$ will be
  trivial, and we can find a section $\sigma\colon L \to P_L$.
  We have $\restricted{\bigl(\alpha\wedge d\alpha\bigr)}{T P_L} =
  \restricted{\bigl(\alpha\wedge \pi^*\omega\bigr)}{T P_L} \equiv 0$,
  so that $\xi$ induces a Legendrian foliation~$\fF$ on $P_L$.
  Furthermore, since the infinitesimal generator~$X_\varphi$ of the
  circle action satisfies $\alpha(X_\varphi) \equiv 1$, it follows
  that $\fF$ is everywhere regular.
  Using the section~$\sigma$, we can identify $P_L$ with $\SSS^1 \times
  L$, and write $\restricted{\alpha}{T P_L}$ as
  \begin{equation*}
    d\varphi + \beta\;,
  \end{equation*}
  where $\varphi$ is the coordinate on the circle and $\beta$ is a
  closed $1$-form on $L$.
  The leaves of the foliation~$\fF$ are local sections, but they need
  not be global ones, and usually these leaves will not even be
  compact.
  Instead the proper way to think of them is as the horizontal lift of
  the flat connection $1$-form~$\restricted{\alpha}{T P_L}$.
  Choose any loop $\gamma \subset L$ based at a point~$p_0\in L$.
  We want to lift $\gamma(t)$ to a path $\widetilde \gamma(t) =
  \bigl(e^{i\varphi(t)}, \gamma(t)\bigr)$ in $P_L \cong \SSS^1 \times
  L$ that is always tangent to a leaf of $\fF$, so that
  \begin{equation*}
    \widetilde\gamma'(t) = \bigl(-\beta(\gamma'(t)), \gamma'(t)\bigr) \;.
  \end{equation*}
  In particular start and end point of $\widetilde \gamma$ are related
  by the monodromy
  \begin{equation*}
    C_\gamma := -\int_\gamma \beta\;,
  \end{equation*}
  that means, if $\widetilde \gamma$ starts at $\bigl(e^{i\varphi_0},
  p_0\bigr) \in \SSS^1 \times L$, then its end point will be
  $\bigl(e^{i(\varphi_0 + C_\gamma)}, p_0\bigr)$.
  Note that since the connection is flat, that means, $\beta$ is
  closed, two homologous paths from $p_0$ to $p_1$ will lift the end
  point in the same way.
  Thus we have a well-defined map
  \begin{equation*}
    H_1(L, \ZZ) \to \SSS^1 \;.
  \end{equation*}
  The leaves of the Legendrian foliation will only be compact, if the
  image of this map is discrete.
  Note that the embedding of $H^1(L,\QQ) \to H^1(L,\RR)$ is dense, and
  so we find a $1$-form~$\beta'$ arbitrarily close to $\beta$ such
  that the monodromy for every loop in $L$ will be a rational number.
  Clearly, we can extend $\delta = \beta'-\beta$ to a $1$-form defined
  on the whole bundle~$P$, and suppose that $\delta$ is sufficiently
  small so that $\alpha' = \alpha + \delta$ determines a contact
  structure that is isotopic to the initial one.
  We may hence suppose that after a small perturbation of $\alpha$
  that the Legendrian foliation on $P_L$ is given by $d\phi + \beta'$.
  In fact, since $H_1(L, \ZZ)$ is finitely generated, we find a
  number~$c\in \QQ$ such that all possible values of the monodromy are
  a multiple of $c$, and by slightly perturbing $\alpha$ we obtain a
  regular Legendrian foliation on $P_L$, with compact leaves.
\end{example}

The second example gives a Legendrian foliation with a codimension~$1$
singular set.

\begin{example}\label{expl: Legendrian sphere bundle}
  Let $L$ be any smooth $(n+1)$-dimensional manifold with a Riemannian
  metric~$g$.
  It is well-known that the unit cotangent
  bundle~$\SSS\bigl(T^*L\bigr)$ carries a contact structure given as
  the kernel of the canonical $1$-form~$\lcan$.
  The fibers of this bundle are Legendrian spheres, hence if we choose
  any smooth regular loop~$\gamma\colon \SSS^1 \to L$, and if we study
  the fibers lying over this path, we obtain the submanifold $N_\gamma
  := \pi^{-1}(\gamma)$ that has a singular Legendrian foliation.
  In fact, we can naturally decompose $\restricted{T^*L}{\gamma}$ into
  the two subsets~$U_+$ and $U_-$ defined as
  \begin{equation*}
    U_\pm = \bigl\{\nu\in N_\gamma\bigm|\,
    \pm \nu (\gamma') \ge 0 \bigr\} \;.
  \end{equation*}
  These sets correspond in each fiber of $N_\gamma$ to opposite
  hemispheres.
  The singular set of the Legendrian foliation on $N_\gamma$ is
  $U_+\cap U_-$, and that the regular leaves correspond to the
  intersection of each fiber of $N_\gamma$ with the interior of $U_+$
  and $U_-$.
  In particular, if $N_\gamma$ is orientable, we obtain that it can be
  written as
  \begin{equation*}
    \bigl( \SSS^1 \times \SSS^n, x_0\,d\varphi\bigr)\;,
  \end{equation*}
  where $\varphi$ is the coordinate on $\SSS^1$, and $(x_0, \dotsc,
  x_n)$ are the coordinates on $\SSS^n$.
  Using the results of Section~\ref{sec: codim 1 sing}, we can perturb
  $N_\gamma$ to a submanifold with a regular Legendrian foliation
  composed of two Reeb components.
\end{example}

\section{Legendrian open books}\label{sec: Legendrian open books}

Even though we discussed Legendrian foliations quite generally, we
will only be interested in two special types: \emph{Legendrian open
  books} introduced in \cite{NiederkrugerRechtman} and \emph{bordered
  Legendrian open books} introduced in
\cite{WeafFillabilityHigherDimension}.
Both objects were defined with the aim of generalizings results from
$3$-dimensional contact topology that hold for the $2$-sphere with
standard foliation and the overtwisted disk respectively
\cite{BedfordGaveau, Gromov_HolCurves, Eliashberg_filling,
  HoferWeinstein}.

\begin{definition}
  Let $N$ be a closed manifold.  An \defin{open book} on $N$ is a
  pair~$(B, \vartheta)$ where:
  \begin{itemize}
  \item The \defin{binding}~$B$ is a nonempty codimension~$2$
    submanifold in the interior of $N$ with trivial normal bundle.
  \item $\vartheta \colon N \setminus B \to \SSS^1$ is a fibration,
    which coincides in a neighborhood $B \times \DD^2$ of $B = B
    \times \{0\}$ with the normal angular coordinate.
  \end{itemize}
\end{definition}

\begin{definition}
  If $N$ is a compact manifold with nonempty boundary, then a
  \defin{relative open book} on $N$ is a pair~$(B, \vartheta)$ where:
  \begin{itemize}
  \item The \defin{binding}~$B$ is a nonempty codimension~$2$
    submanifold in the interior of $N$ with trivial normal bundle.
  \item $\vartheta \colon N \setminus B \to \SSS^1$ is a fibration
    whose fibers are transverse to $\p N$, and which coincides in a
    neighborhood $B \times \DD^2$ of $B = B \times \{0\}$ with the
    normal angular coordinate.
\end{itemize}
\end{definition}

We are interested in studying contact manifolds with submanifolds with
a Legendrian foliation that either define an open book or a relative
open book.

\begin{definition}
  A closed submanifold~$N$ carrying a Legendrian foliation~$\fF$ in a
  contact manifold~$(M,\xi)$ is a \defin{Legendrian open book}
  (abbreviated \defin{\LOB}), if $N$ admits an open book $(B,
  \vartheta)$, whose fibers are the regular leaves of the Legendrian
  foliation (the binding is the singular set of $\fF$).
\end{definition}

\begin{definition}
  A compact submanifold~$N$ with boundary in a contact
  manifold~$(M,\xi)$ is called a \defin{bordered Legendrian open book}
  (abbreviated \defin{\BLOB}), if $N$ carries a Legendrian
  foliation~$\fF$ and if it has a relative open book $(B, \vartheta)$
  such that:
  \begin{itemize}
  \item [(i)] the regular leaves of $\fF$ lie in the fibers of
    $\theta$,
  \item [(ii)] $\sing(\fF) = \p N \cup B$.
  \end{itemize}
  A contact manifold that contains a \BLOB is called
  \defin{$PS$-overtwisted}. \index{$PS$-overtwisted}
\end{definition}

\begin{example}
  \begin{itemize}
  \item [(i)] Every \LOB in a contact $3$-manifold is diffeomorphic to
    a $2$-sphere with the binding consisting of the north and south
    poles, and the fibers being the longitudes.
    This special type of \LOB{} has been studied extensively and has
    given several important applications, see for example
    \cite{BedfordGaveau, Gromov_HolCurves, Eliashberg_filling,
      HoferWeinstein}.
    It is easy to find such \LOB{}s locally, for example, the unit
    sphere in $\RR^3$ with the standard contact structure $\xi = \ker
    \bigl(dz + x\,dy - y\,dx\bigr)$.
  \item [(ii)] A \BLOB in a $3$-dimensional contact manifold is an
    overtwisted disk (with singular boundary).
  \item [(iii)] In higher dimensions, the plastikstufe had been
    introduced as a filling obstruction
    \cite{NiederkrugerPlastikstufe}, but note that a plastikstufe is
    just a specific \BLOB that is diffeomorphic to $\DD^2\times B$,
    where the fibration is the one of an overtwisted disk (with
    singular boundary) on the $\DD^2$-factor, extended by a product
    with a closed manifold~$B$.
    Topologically a \BLOB might be \emph{much} more general than the
    initial definition of the plastikstufe.
    For example, a plastikstufe in dimension~$5$ is always
    diffeomorphic to a solid torus~$\DD^2 \times \SSS^1$ while a
    $3$-manifold admits a relative open book if and only if its
    boundary is a nonempty union of tori.
  \end{itemize}
\end{example}

The importance of the previous definitions lie in the following two
theorems, which will be proved in Chapter~\ref{chapter: J-holomorphic
  curves}.

\begin{main_theorem}[\cite{NiederkrugerPlastikstufe, WeafFillabilityHigherDimension}]\label{thm: non-fillability}
  Let $(M, \xi)$ be a contact manifold that contains a \BLOB~$N$, then
  $M$ does not admit any semi-positive weak symplectic filling $(W,
  \omega)$ for which $\restricted{\omega}{TN}$ is exact.
\end{main_theorem}

The statement above is a generalization of the analogous statement
found first for the overtwisted disk in \cite{Gromov_HolCurves,
  Eliashberg_filling}.

\begin{remark}
  A \BLOB obstructs always (semi-positive) \emph{strong} symplectic
  filling, because in that case the restriction of $\omega$ to $N$ is
  exact.
\end{remark}

\begin{remark}
  In dimension $4$ and $6$, every symplectic manifold is automatically
  semi-positive.
\end{remark}

\begin{main_theorem}[\cite{NiederkrugerRechtman}]\label{thm: homology
    LOB}
  Let $(M, \xi)$ be a contact manifold of dimension $(2n+1)$ that
  contains a \LOB~$N$.
  If $M$ has a weak symplectic filling $(W,\omega)$ that is
  symplectically aspherical, and for which $\restricted{\omega}{TN}$
  is exact, then it follows that $N$ represents a trivial class in
  $H_{n+1}(W, \ZZ_2)$.
  If the first and second Stiefel-Whitney classes~$w_1(N)$ and
  $w_2(N)$ vanish, then we obtain that $N$ must be a trivial class in
  $H_{n+1}(W, \ZZ)$.
\end{main_theorem}

\begin{remark}
  The methods from \cite{HoferWeinstein} can be generalized for
  Theorem~\ref{thm: non-fillability}, see
  \cite{AlbersHoferWeinsteinPlastikstufe}, and for Theorem~\ref{thm:
    homology LOB}, see \cite{NiederkrugerRechtman}, to find closed
  contractible Reeb orbits.
\end{remark}

\section{Examples of \BLOB{}s}

The most important result of these notes is the construction of
non-fillable manifolds in higher dimensions.
The first such manifolds were obtained by Presas in
\cite{PresasExamplesPlastikstufes}, and modifying his examples it was
soon possible to show that every contact structure can be converted
into one that is $PS$-overtwisted
\cite{vanKoertPSOvertwistedEverywhere}.
This result was reproved and generalized in
\cite{EtnyreGeneralizedLutzTwists}, where it was shown that we may
modify a contact structure into one that is $PS$-overtwisted without
changing the homotopy class of the underlying almost contact
structure.
A very nice explicit construction in dimension~$5$ that is similar to
the $3$-dimensional Lutz twist was given in \cite{Mori_Lutz}.
In \cite{WeafFillabilityHigherDimension} the construction was extended
and produced examples that are not $PS$-overtwisted but share many
properties with $3$-manifold that have positive Giroux torsion.
The following unpublished construction is due to Francisco Presas who
explained it to me during a stay in Madrid.
It is probably the easiest way to produce a closed $PS$-overtwisted
manifolds of arbitrary dimensions.

\begin{theorem}[Fran Presas]\label{thm: fiber sum PS-ot}
  Let $(M_1,\xi_1)$ and $(M_2,\xi_2)$ be contact manifolds of
  dimension~$2n+1$ that both contain a $PS$-overtwisted
  submanifold~$(N, \xi_N)$ of codimension~$2$ with trivial normal
  bundle.
  The \defin{fiber sum} \index{fiber sum} of $M_1$ and $M_2$ along $N$
  is a $PS$-overtwisted $(2n+1)$-manifold.
\end{theorem}
\begin{proof}
  Let $\alpha_N$ be a contact form for $\xi_N$.
  The manifold~$N$ has neighborhoods~$U_1 \subset M_1$ and $U_2
  \subset M_2$ that are contactomorphic to
  \begin{equation*}
    \DD_{\sqrt{\epsilon}}^2 \times N
  \end{equation*}
  with contact structure given as the kernel of the $1$-form $\alpha_N
  + r^2\, d\varphi$ \cite[Theorem~2.5.15]{Geiges_book}.
  We can remove the submanifold~$\{0\}\times N$ in this model, and do
  a reparametrization of the $r$-coordinate by $s=r^2$ to bring the
  neighborhood into the form
  \begin{equation*}
    (0,\epsilon) \times \SSS^1 \times N
  \end{equation*}
  with contact form $\alpha_N + s\, d\varphi$.
  We extend $M_1 \setminus N$ and $M_2\setminus N$ by attaching the
  negative $s$-direction to the model collar, so that we obtain a
  neighborhood
  \begin{equation*}
    \bigl( (-\epsilon,\epsilon) \times \SSS^1 \times N, \quad
    \alpha_N + s\, d\varphi\bigr) \;.
  \end{equation*}
  Denote these extended manifolds by $(\widetilde M_1, \widetilde
  \xi_1)$ and $(\widetilde M_2, \widetilde \xi_2)$, and glue them
  together using the contactomorphism
  \begin{align*}
    (-\epsilon,\epsilon) \times \SSS^1 \times N & \to
    (-\epsilon,\epsilon) \times \SSS^1 \times N \\
    (s, \varphi, p) &\mapsto (-s, -\varphi, p) \;.
  \end{align*}
  We call the contact manifold $(M', \xi')$ that we have obtained this
  way the \defin{fiber sum} of $M_1$ and $M_2$ along $N$.
  If $S$ is a \BLOB in $N$, then it is easy to see that $\{0\} \times
  \SSS^1 \times S$ is a \BLOB in the model neighborhood $(-\epsilon,
  \epsilon) \times \SSS^1 \times N$.
\end{proof}

With this proposition, we can now construct non-fillable contact
manifolds of arbitrary dimension.
Every oriented $3$-manifold admits an overtwisted contact structure in
every homotopy class of almost contact structures.
Let $(M,\xi)$ be a compact manifold, let $\alpha_M$ be a contact form
for $\xi$.
A fundamental result due to Emmanuel Giroux gives the existence of a
compatible open book decomposition for $M$ \cite{Giroux_talk}.
Using this open book decomposition, it is easy to find functions
$f,g\colon M \to \RR$ such that
\begin{equation*}
  \bigl(M \times \TT^2, \ker (\alpha_M + f\,dx + g\,dy) \bigr)
\end{equation*}
is a contact structure, see \cite{BourgeoisTori}, where $(x,y)$
denotes the coordinates on the $2$-torus.
The fibers $M\times \{z\}$ are contact submanifold with trivial normal
bundle, so that in particular if $(M,\xi)$ is $PS$-overtwisted, we can
apply the construction above to glue two copies of $M\times \TT^2$
along a fiber $M\times \{z\}$.
This way, we obtain a $PS$-overtwisted contact structure on $M\times
\Sigma_2$, where $\Sigma_2$ is a genus~$2$ surface.
Using this process inductively, we find closed $PS$-overtwisted
contact manifolds of any dimension $\ge 3$.
Note that in dimension~$5$, we can find more easily examples to which
we can apply Theorem~\ref{thm: fiber sum PS-ot}, so that it is not
necessary to rely on \cite{BourgeoisTori}.
Let $(M,\xi)$ be an overtwisted $3$-manifold with contact
form~$\alpha$.
After normalizing $\alpha$ with respect to a Riemannian metric, it
describes a section
\begin{equation*}
  \sigma_\alpha\colon M \to \SSS(T^*M)
\end{equation*}
in the unit cotangent bundle.
It satisfies the fundamental relation~$\sigma_\alpha^* \lcan =
\alpha$, hence it gives a contact embedding of $(M, \xi)$ into
$\bigl(\SSS(T^*M), \ker \lcan\bigr)$.
For trivial normal bundle, this allows us to glue with
Theorem~\ref{thm: fiber sum PS-ot} two copies together and obtain a
$PS$-overtwisted $5$-manifold.

\chapter{Behavior of $J$-holomorphic disks imposed by
  convexity}\label{chapter: J-holomorphic curves}

The following section only fixes notation, and explains some
well-known facts about $J$-convexity.
With some basic knowledge on $J$-holomorphic curves, one can safely
skip it and continue directly with Section~\ref{sec: local_model},
which describes the local models around the binding and the boundary
of the \LOB{}s and \BLOB{}s and the behavior of holomorphic disks that
lie nearby.
The next two sections include a description about moduli spaces and
their basic properties, but most results are only explained in an
intuitive way without giving any proofs.
The fifth section deals with the Gromov compactness of the considered
moduli spaces, and the chapter finishes proving the two applications
that relate a \LOB or a \BLOB to the topology of a symplectic filling.

\section{Almost complex structures and maximally foliated
  submanifolds}
\label{section: J-convexity}

\subsection{Preliminaries: $J$-convexity}

\subsubsection{The maximum principle}

One of the basic ingredients in the theory of $J$-holomorphic curves
with boundary is the maximum principle, which we will now briefly
describe in the special case of Riemann surfaces.
We assume in this section that $(\Sigma, j)$ is a Riemann surface that
does not need to be compact and may or may not have boundary.
We define the differential operator $d^j$ that associates to every
smooth function $f\colon \Sigma \to \RR$ a $1$-form given by
\begin{equation*}
  \bigl(d^j f\bigr)(v) := - df (j\, v)
\end{equation*}
for $v\in T\Sigma$.

\begin{definition}
  We say that a function~$f\colon (\Sigma, j) \to \RR$ is
  \begin{itemize}
  \item [(a)] \defin{harmonic} if the $2$-form~$dd^jf$ vanishes
    everywhere, \index{function!harmonic}
  \item [(b)] it is \defin{subharmonic} if the $2$-form~$dd^jf$ is a
    positive volume form with respect to the orientation defined by
    $(v, j\, v)$ for any non-vanishing vector~$v \in T\Sigma$.
  \item [(c)] If $f$ only satisfies
    \begin{equation*}
      dd^jf\bigl(v, j\, v\bigr) \ge 0
    \end{equation*}
    then we call it \defin{weakly subharmonic}.
    \index{function!subharmonic}
  \end{itemize}
\end{definition}

In particular, if we choose a complex chart~$\bigl(U\subset \CC,
\phi\bigr)$ for $\Sigma$ with coordinate $z=x+iy$, we can represent
$f$ by $f_U := f\circ \phi^{-1} \colon U \to \RR$.
The $2$-form~$dd^j f$ simplifies on this chart to $dd^i f_U$, because
$\phi$ is holomorphic with respect to $j$ and $i$, and we can write
$dd^i f_U$ in the form $\bigl(\triangle f_U\bigr)\, dx\wedge dy$,
where the Laplacian is defined as
\begin{equation*}
  \triangle f_U = \frac{\partial^2 f_U}{\partial x^2}
  + \frac{\partial^2 f_U}{\partial y^2} \;.
\end{equation*}
Note that $f_U$ is subharmonic, if and only if $dd^i
f_U(\partial_x, \partial_y) > 0$, that means, $\triangle f_U > 0$.
For strictly subharmonic functions, it is obvious that they may not
have any interior maxima, because the Hessian needs to be negative
definite at any such point.
We really need to consider both weakly subharmonic functions and the
behavior at boundary points.
To prove the maximum principle in this more general setup, we use the
following technical result.

\begin{lemma}\label{lemma: subharmonic boundary derivative}
  Let $f\colon \DD^2 \subset \CC \to \RR$ be a function that is $C^1$
  on the closed unit disk, and both $C^2$ and weakly subharmonic on
  the interior of the disk.
  Assume that $f$ takes its maximum at a boundary point $z_0 \in \p
  \DD^2$ and is everywhere else strictly smaller than $f(z_0)$.
  Choose an arbitrary vector~$X\in T_{z_0}\CC$ at $z_0$ pointing
  transversely out of $\overline{\DD}^2$.
  Then the derivative~$\lie{X} f(z_0)$ in $X$-direction needs to be
  \emph{strictly} positive.
\end{lemma}
\begin{proof}
  We will perturb $f$ to a \emph{strictly} subharmonic function making
  use of the auxiliary function $g\colon \overline{\DD}^2 \to \RR$
  defined by
  \begin{equation*}
    g(r)  =  r^4 - \frac{9}{4}\,r^2  + \frac{5}{4} \;.
  \end{equation*}

  \begin{figure}[htbp]
    \centering
    \includegraphics[height=3.5cm,keepaspectratio]{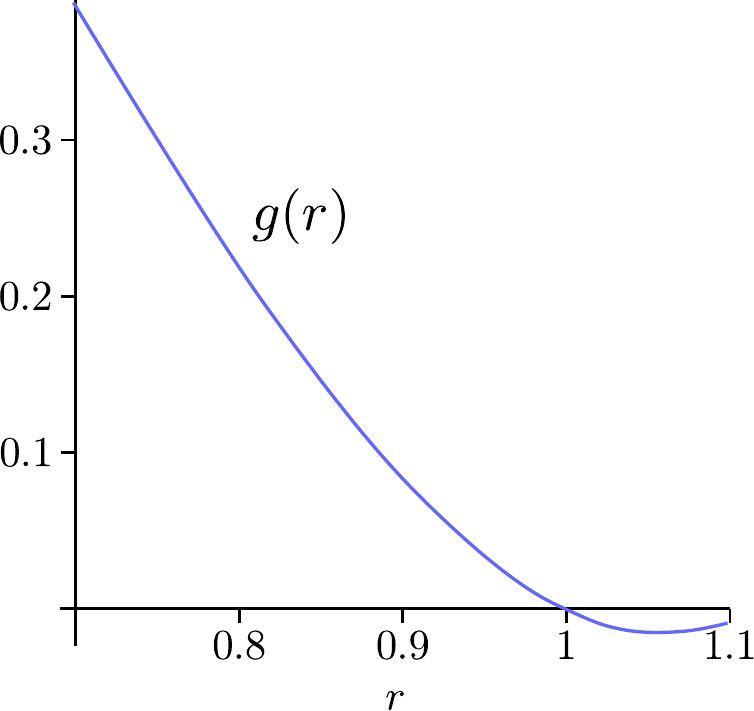}
    \caption{The function~$g(r)$ is subharmonic, vanishes on the
      boundary, and has negative radial derivative.}\label{fig: graph
      used for perturbing weak into strict subharmmonic}
  \end{figure}

  The function~$g$ vanishes along the boundary $\p \DD^2$, and its
  derivative in any direction~$v$ that is positively transverse to the
  boundary~$\p\DD^2$ is strictly negative, because $\partial_\varphi g
  = 0$ and because
  \begin{equation*}
    r\,\partial_r g = \frac{1}{2}\,r^2\, (8r^2 - 9) \;.
  \end{equation*}
  Finally, we also see that $g$ is strictly subharmonic on the open
  annulus~$\AA = \{z \in \CC \, | \, 3/4 < \abs{z} < 1 \}$ as
  \begin{equation*}
    \triangle g = \frac{\partial^2 g}{\partial x^2}
    + \frac{\partial^2 g}{\partial y^2} =   16 r^2 - 9 \;.
  \end{equation*}
  We slightly perturb $f$ by setting $f_\epsilon = f + \epsilon\, g$
  for small $\epsilon > 0$, and we additionally restrict $f_\epsilon$
  to the closure of the annulus~$\AA$.
  Note in particular that $f_\epsilon$ must take its maximum
  on~$\p\AA$, because $f_\epsilon$ is \emph{strictly} subharmonic on
  the interior of $\AA$ so that one of $\frac{\partial^2
    f_\epsilon}{\partial x^2}$ or $\frac{\partial^2
    f_\epsilon}{\partial y^2}$ must be strictly positive.
  This contradicts existence of possible interior maximum points.
  The functions~$f_\epsilon$ are equal to $f$ along the outer boundary
  of $\AA$ so that the maximum of $f_\epsilon$ will either lie in
  $z_0$ or on the inner boundary of~$\AA$.
  The initial function~$f$ is by assumption strictly smaller than
  $f(z_0)$ on the inner boundary of the annulus and by choosing
  $\epsilon$ sufficiently small, it follows that the perturbed
  function~$f_\epsilon$ will still be strictly smaller than
  $f_\epsilon(z_0) = f(z_0)$.
  Thus $z_0$ will also be the maximum of $f_\epsilon$.
  Let $X$ be a vector at $z_0$ that points transversely out of
  $\overline{\DD}^2$.
  The derivative $\lie{X} f_\epsilon$ at $z_0$ cannot be strictly
  negative, because $z_0$ is a maximum, and so since
  \begin{equation*}
    0 \le \lie{X} f_\epsilon = \lie{X} f + \epsilon\, \lie{X} g \;,
  \end{equation*}
  the derivative of $f$ in $X$-direction has to be \emph{strictly}
  positive, yielding the desired result.
\end{proof}

Now we are prepared to state and prove the maximum principle.

\begin{theorem}[Weak maximum principle]\label{thm: maximum principle}
  Let $(\Sigma, j)$ be a connected compact Riemann surface.
  A weakly subharmonic function $f\colon \Sigma \to \RR$ that attains
  its maximum at an interior point $z_0\in \Sigma \setminus \p \Sigma$
  must be constant.
\end{theorem}
\begin{proof}
  The proof is classical and holds in much greater generality (see for
  example~\cite{GilbargTrudinger}).
  Nonetheless we will explain it in the special case needed by us to
  show that it only uses elementary techniques.
  The strategy is simply to find a closed disk in the interior of the
  Riemann surface with the properties required by Lemma~\ref{lemma:
    subharmonic boundary derivative}.
  Then the function~$f$ increases in radial direction further, so that
  the maximum point was not really a maximum.

  \begin{figure}[htbp]
    \centering
    \includegraphics[height=4.5cm,keepaspectratio]{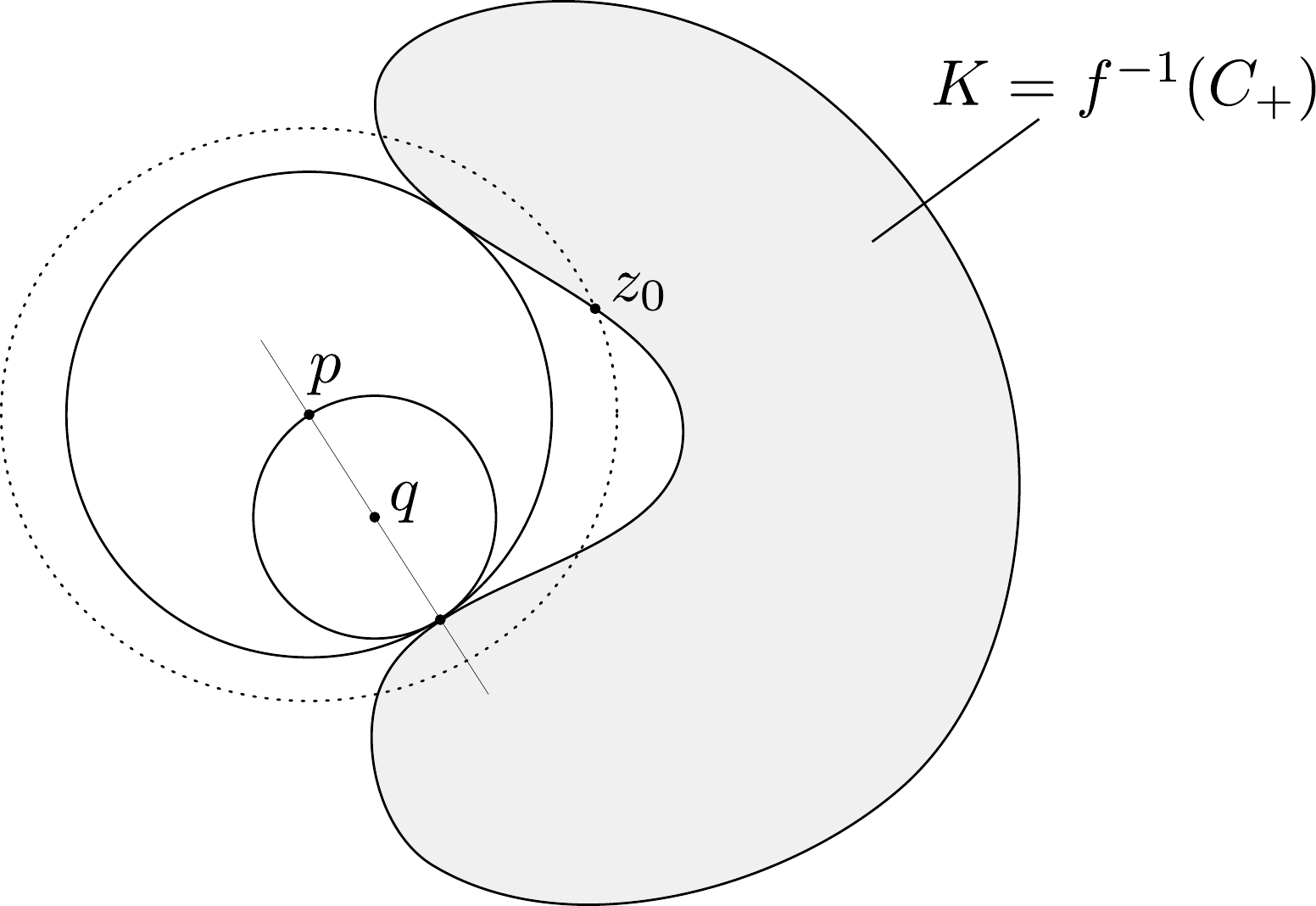}
    \caption{Constructing a disk that has a single maximum on its
      boundary.}\label{fig: disk for max principle}
  \end{figure}

  More precisely, assume $f$ not to be constant, and to have a maximum
  at an interior point $z_+ \in \Sigma \setminus \p\Sigma$ with $C_+
  := f(z_+)$.
  The subset $K := f^{-1}(C_+)\cap \mathring \Sigma$ is closed in
  $\mathring \Sigma$.
  For every point $z \in K$, we find an $R_z > 0$ such that the open
  disk $D_{R_z}(z)$ is contained in some complex chart.
  There must be a point $z_0 \in K$ for which the half sized disk
  $D_{R_{z_0}/2}(z_0)$ intersects $\mathring \Sigma \setminus K$, for
  otherwise $K$ would be open and hence as $\mathring\Sigma$ is
  connected, $K = \mathring\Sigma$.
  Let $p$ be a point in $D_{R_{z_0}/2}(z_0)\setminus K$ (see
  Fig.~\ref{fig: disk for max principle}).
  It lies so close to $z_0$ that the entire closed disk of radius
  $\abs{p-z_0}$ lies in the chart~$U$, and then we can choose first a
  disk~$\overline{\DD_R(p)}$ centered at $p$, where $R$ is the largest
  number for which the \emph{open} disk does not intersect
  $f^{-1}(C_+)$.
  We are interested in finding a closed disk that intersects
  $f^{-1}(C_+)$ at a \emph{single} boundary point: For this let $q$ be
  the mid point between $p$ and one of the boundary points in $\p
  \DD^2_R(p)\cap f^{-1}(C_+)$.
  The disk~$\DD^2_{R/2}(q)$ touches $f^{-1}(C_+)$ at exactly one
  point.
  This smaller disk satisfies the conditions of Lemma~\ref{lemma:
    subharmonic boundary derivative}, and so it follows that the
  derivative of~$f$ at the maximum is strictly positive in radial
  direction.
  But since this point lies in the interior of $\Sigma$, it follows
  that $f$ still increases in that direction and hence this point
  cannot be the maximum.
  Of course, the whole existence of the disk was based on the
  assumption that $f$ was not constant, so we obtain the statement of
  the theorem.
\end{proof}

If $\Sigma$ has boundary, we also get the following refinement.

\begin{theorem}[Boundary point lemma]\label{thm: boundary point lemma}
  Let $f\colon \Sigma \to \RR$ be a weakly subharmonic function on a
  connected compact Riemann surface~$(\Sigma, j)$ with boundary.
  Assume $f$ takes its maximum at a point $z_+ \in \p \Sigma$, then
  $f$ will either be constant or the derivative at $z_+$
  \begin{equation*}
    \lie{X} f(z_+) > 0
  \end{equation*}
  in any outward direction $X \in T_{z_+}\Sigma$ has to be strictly
  positive.
\end{theorem}
\begin{proof}
  Denote the maximum $f(z_+)$ by $C_+$.
  By the maximum principle, Theorem~\ref{thm: maximum principle}, we
  know that $f$ will be constant if there is a point $z\in
  \Sigma\setminus \p \Sigma$ for which $f(z) = C_+$.
  We can thus assume that for all $z\notin \p \Sigma$, we have $f <
  C_+$.
  Using a chart $U$ around the point $z_+$, that represents an open
  set in $\HH := \{z\in \CC|\, \ImaginaryPart z \ge 0\}$, such that
  $z_+$ corresponds to the origin, we can easily find a small disk in
  $\HH$ that touches $\p\HH$ only in $0$, and hence allows us to
  directly apply Lemma~\ref{lemma: subharmonic boundary derivative} to
  complete the proof.
\end{proof}

\subsubsection{Plurisubharmonic functions}

We will now explain the connection between the previous section and
contact topology.
Let $(W,J)$ be an almost complex manifold, that means that $J$ is a
section of the endomorphism bundle~$\End(TM)$ with $J^2 = -
\mathbf{1}$.
Define the differential $d^J f$ of a smooth function~$f\colon W\to
\RR$ as before by
\begin{equation*}
  \bigl(d^Jf\bigr)(v) :=  - df(J\cdot v)
\end{equation*}
for any vector $v\in TW$.

\begin{definition}
  We say that a function~$h\colon W \to \RR$ is
  \defin{$J$-plurisubharmonic}, if the $2$-form
  \begin{equation*}
    \omega_h :=  dd^J h
  \end{equation*}
  evaluates positively on $J$-complex lines, that means that
  $\omega_h(v, J v)$ is strictly positive for every non-vanishing
  vector $v \in TW$. \index{function!plurisubharmonic}
  If $\omega_h$ vanishes, then we say that $h$ is
  \defin{$J$-harmonic}. \index{function!harmonic}
\end{definition}

\begin{remark}
  \begin{enumerate}
  \item If $h$ is $J$-plurisubharmonic, then $\omega_h$ is an exact
    symplectic form that tames~$J$.
  \item If $\omega_h$ is only non-negative, then we say that $h$ is
    \defin{weakly $J$-plurisubharmonic}.
    This notion might be for example interesting in the context of
    confoliations.
  \end{enumerate}
\end{remark}

Let $(\Sigma, j)$ be a Riemann surface that does not need to be
compact, and may or may not have boundary.
We say that a smooth map~$u\colon \Sigma \to W$ is
\defin{$J$-holomorphic}, if its differential commutes with the pair
$(j,J)$, that means, at every $z\in \Sigma$ we have
\begin{equation*}
  J\cdot Du = Du\cdot j \;.
\end{equation*}
\index{holomorphic map}
Using the commutation relation, we easily check for every
$J$-holomorphic map~$u$ and every smooth function $f\colon U \to \RR$
the formula
\begin{equation}
  \label{eq: pullback dJ}
  u^* d^J f = - df\cdot J \cdot Du  = - df\cdot Du \cdot j
  = - d(f\circ u) \cdot j =  d^j (f\circ u) = d^ju^*f\;.
\end{equation}

\begin{corollary}
  If $u\colon (\Sigma,j) \to (W,J)$ is $J$-holomorphic and $h\colon W
  \to \RR$ is a $J$-plurisubharmonic function, then $h\circ u$ will be
  weakly subharmonic, because
  \begin{equation*}
    dd^j (h\circ u) = d\, u^* d^Jh = u^* d d^J h
  \end{equation*}
  and because the differential~$Du$ commutes with the complex
  structures, so that
  \begin{equation*}
    dd^j (h\circ u) \bigl(v, j v)
    =  d d^J h (Du\cdot v, J\cdot Du\cdot v) \ge 0
  \end{equation*}
  for every vector~$v \in T\Sigma$.
  The function is strictly positive precisely at points $z\in U$,
  where $Du_z$ does not vanish.
\end{corollary}

The maximum principle restricts severely the behavior of holomorphic
maps:

\begin{corollary}\label{cor: holomorphic maps and interior maxima}
  Let $u\colon (\Sigma,j) \to (W,J)$ be a $J$-holomorphic map and
  $h\colon W \to \RR$ be a $J$-plurisubharmonic function.
  If $u$ is not a constant map then $h\circ u\colon \Sigma \to \RR$
  will never take its maximum on the interior of $\Sigma$.
\end{corollary}
\begin{proof}
  Since $h\circ u$ is weakly subharmonic, it follows immediately from
  the maximum principle (Theorem~\ref{thm: maximum principle}) that
  $h\circ u$ must be constant if it takes its maximum in the interior
  of $\Sigma$, and hence $d(h\circ u) = 0$.
  On the other hand, we know that if there were a point $z\in \Sigma$
  with $D_z u \ne 0$, then $\omega_h(Du\cdot v, Du\cdot j v)$ would
  need to be strictly positive for non-vanishing vectors.
  This is not possible though, because $u^*\omega_h = dd^j(h\circ u) =
  0$.
\end{proof}

\begin{corollary}\label{cor: boundary point lemma holomorphic maps}
  Let $(\Sigma,j)$ be a Riemann surface with boundary, $u\colon
  (\Sigma,j) \to (W,J)$ a $J$-holomorphic map and $h\colon W \to \RR$
  be a $J$-plurisubharmonic function.
  If $h\circ u\colon \Sigma \to \RR$ takes its maximum at $z_0\in \p
  \Sigma$ then it follows either that $d(h\circ u)(v) > 0$ for every
  vector~$v\in T_{z_0}\Sigma$ pointing transversely out of the
  surface, or $u$ will be constant.
\end{corollary}
\begin{proof}
  The proof is analogous to the previous one, but uses the boundary
  point lemma (Theorem~\ref{thm: boundary point lemma}) instead of the
  simple maximum principle.
\end{proof}

\begin{remark}
  Note that if $h$ is only \emph{weakly} plurisubharmonic, then we can
  only deduce in the two corollaries above that $u$ has to lie in a
  level set of $h$, and not that $u$ itself must be constant.
\end{remark}

\subsubsection{Contact structures as convex boundaries}

Now we will finally explain the relation between plurisubharmonic
functions and contact manifolds.

\begin{definition}
  Let $(W,J)$ be an almost complex manifold with boundary.
  We say that $W$ has \defin{$J$-convex boundary}, if there exists a
  smooth function $h\colon W \to (-\infty,0]$ with the properties
  \begin{itemize}
  \item $h$ is $J$-plurisubharmonic on a \emph{neighborhood} of $\p
    W$,
  \item $h$ is a regular equation for $\p W$, that means, $0$ is a
    regular value of $h$ and $\p W = h^{-1}(0)$.
  \end{itemize}
\end{definition}

Note that the function~$h$ in the definition takes its maximum on $\p
W$, so that it must be strictly increasing in outward direction.
We will show that the boundary of an almost complex manifold is
$J$-convex if and only if it carries a natural cooriented contact
structure (whose conformal symplectic structure tames $J$).
Remember that we are always assuming our contact manifolds to be
cooriented.
Hence the manifold is oriented, and its contact structure will have a
natural conformal symplectic structure.

\begin{definition}
  Let $M$ be a codimension~$1$ submanifold in an almost complex
  manifold $(W,J)$.
  The \defin{subbundle of complex tangencies} of $M$ is the
  $J$-complex subbundle
  \begin{equation*}
    \xi := TM \cap (J\cdot TM) \;.    
  \end{equation*}
\end{definition}

\begin{proposition}
  Let $(W,J)$ be an almost complex manifold with boundary~$M := \p W$
  and let $\xi$ be the subbundle of complex tangencies of $M$.
  We have the following equivalence:
  \begin{enumerate}
  \item The boundary~$M$ is $J$-convex.
  \item The subbundle~$\xi$ is a cooriented contact structure whose
    natural orientation is compatible with the boundary orientation of
    $M$, and whose natural conformal symplectic structure tames
    $\restricted{J}{\xi}$.
  \end{enumerate}
\end{proposition}
\begin{proof}
  To prove the direction ``$(1)\Rightarrow(2)$'', let $h$ be the
  $J$-plurisubharmonic equation of $M$ that exists by assumption.
  A straight forward calculation shows that the kernel of the $1$-form
  $\alpha := \restricted{d^Jh}{TM}$ is precisely $\xi$, and in
  particular that $\alpha$ does not vanish.
  Furthermore $\restricted{d\alpha}{TM} = \restricted{\omega_h}{TM}$
  is a symplectic structure on $\xi$ that tames $\restricted{J}{\xi}$,
  so that $\alpha$ is a contact form.
  To check that $\alpha\wedge d\alpha^{n-1}$ is a positive volume form
  with respect to the boundary orientation induced on $M$ by $(W,J)$,
  let $R_\alpha$ be the Reeb field of $\alpha$, and define a vector
  field~$Y = - J\, R_\alpha$.
  The field~$Y$ is positively transverse to $\p W$, because $\lie{Y} h
  = dh(Y) = d^Jh(R_\alpha) = \alpha(R_\alpha) = 1$ is positive.
  Choosing a basis $(v_1,\dotsc, v_{2n-2})$ for $\xi$ at a point $p\in
  M$, we compute
  \begin{equation*}
    \alpha \wedge d\alpha^{n-1} \bigl(R_\alpha, v_1,\dotsc, v_{2n-2}\bigr)
    = d\alpha^{n-1} \bigl(v_1,\dotsc, v_{2n-2}\bigr) 
    = \omega_h^{n-1} \bigl(v_1,\dotsc, v_{2n-2}\bigr)\;.
  \end{equation*}
  Similarly, we obtain
  \begin{equation*}
    \begin{split}
      \omega_h^n \bigl(Y, R_\alpha, v_1,\dotsc, v_{2n-2}\bigr) & = n\,
      \omega_h(Y, R_\alpha) \cdot \omega_h^{n-1}
      \bigl(v_1,\dotsc, v_{2n-2}\bigr) \\
      &= n\, \omega_h(R_\alpha, J R_\alpha) \cdot \omega_h^{n-1}
      \bigl(v_1,\dotsc, v_{2n-2}\bigr) \;,
    \end{split}
  \end{equation*}
  where we have used that $\omega_h(R_\alpha, v_j) = d\alpha
  (R_\alpha, v_j) = 0$ for all $j\in \{1,\dotsc,n-1\}$.
  The first term~$\omega_h(R_\alpha, J R_\alpha)$ is positive, and
  hence $\alpha\wedge d\alpha^{n-1}$ and $\iota_Y \omega_h^n$ induce
  identical orientations on $M$.
  To prove the direction ``$(2)\Rightarrow(1)$'', choose any collar
  neighborhood~$(-\epsilon,0] \times M$ for the boundary, and let $t$
  be the coordinate on $(-\epsilon,0]$.
  First note that $\alpha = \restricted{d^Jt}{TM}$ is a non-vanishing
  $1$-form with kernel~$\xi$, so in particular it will be contact.
  Let $R_\alpha$ be the Reeb field of $\alpha$, and set $Y := - J\,
  R_\alpha$.
  As before, the field~$Y$ is positively transverse to $M$, because of
  $\lie{Y} t = - dt(J\,R_\alpha) = \alpha(R_\alpha) = 1$.
  Let $C$ be a large constant, whose size will be determined below,
  and set $h(t,p) := e^{Ct} - 1$.
  Clearly, $h$ is a regular equation for $M$, and we claim that for
  sufficiently large $C$, $h$ will be a $J$-plurisubharmonic function.
  Let $v \in T_pW$ be any non-vanishing vector at $p\in M$ and
  represent it as
  \begin{equation*}
    v = a Y + b R_\alpha + c Z\;,
  \end{equation*}
  where $Y$ and $R_\alpha$ were defined above, and $Z \in \xi$ is a
  vector in the contact structure that has been normalized such that
  $d\alpha(Z, JZ) = \omega_t(Z, JZ) = 1$.
  Note that the $1$-form $\alpha_C = \restricted{d^Jh}{TM} = C
  e^{Ct}\, \alpha$ is a contact form that represents the same
  coorientation as $\alpha$.
  We compute $\omega_h = d d^J h = C e^{Ct}\,\bigl(\omega_t +
  C\,dt\wedge d^Jt\bigr)$, which simplifies for $t=0$ further to
  $\omega_h = C\, \bigl(\omega_t + C\,dt\wedge d^Jt\bigr)$ and so we
  have
  \begin{equation*}
    \omega_h(R_\alpha, \cdot) = C \, \bigl(\omega_t(R_\alpha,\cdot)
    - C\,dt \bigr)  \quad\text{ and }\quad
    \omega_h(Y, \cdot) = C \,\bigl(\omega_t(\cdot, J\, R_\alpha)
    + C\, d^Jt \bigr)
  \end{equation*}
  This implies $\omega_h(R_\alpha, Z) = \omega_h(R_\alpha, JZ) = 0$
  for all $Z \in \xi$, and $\omega_h(Y, R_\alpha) = C^2 +
  C\,\omega_t(R_\alpha, J R_\alpha)$ can be made arbitrarily large by
  increasing the size of $C$.
  With these relations we obtain
  \begin{align*}
    \omega_h(v, Jv) &= \omega_h( a Y + b R_\alpha + c Z, a R_\alpha
    - bY + c JZ) \\
    &= \bigl(a^2 + b^2\bigr) \, \omega_h(Y, R_\alpha) + c^2\,
    \omega_h(Z, JZ)  + ac\, \omega_h(Y, JZ) + bc\, \omega_h(Y, Z) \\
    &= \bigl(a^2 + b^2\bigr) \, \bigl(C^2 + O(C)\bigr) + C
    \,\bigl(c^2\,\omega_t(Z, JZ) +
    ac\, \omega_t(Y, JZ) + bc\, \omega_t(Y, Z)\bigr) \\
    \intertext{and setting $A_a = \omega_t(Y, JZ)$ and $A_b =
      \omega_t(Y, Z)$ and using that $\omega_t(Z, JZ) = 1$}
    &= \bigl(a^2 + b^2\bigr) \, \bigl(C^2 + O(C)\bigr) + C\,
    \bigl(c^2 + A_aac + A_bbc\bigr)\\
    &= \bigl(a^2 + b^2\bigr) \, \bigl(C^2 + O(C)\bigr) +
    \frac{C}{2}\,\, \Bigl(\bigl(c + a A_a\bigr)^2 - a^2A_a^2 +
    \bigl(c + b A_b\bigr)^2 - b^2A_b^2\Bigr) \\
    &= a^2 \, \bigl(C^2 + O(C)\bigr) + b^2 \, \bigl(C^2 + O(C)\bigr) +
    \frac{C}{2}\, \bigl((c + aA_a)^2 + (c + b A_b)^2\bigr) \;.
  \end{align*}
  By choosing $C$ large enough, we can ensure that the $a^2$- and
  $b^2$-coefficients are both positive.
  Then it is obvious from the computation above that $\omega_h$ tames
  $J$, and hence $h$ is $J$-plurisubharmonic.
\end{proof}

\subsubsection{Legendrian foliations in convex boundaries}

\begin{definition}
  A \defin{totally real submanifold~$N$} of an almost complex
  manifold~$(W,J)$ is a submanifold of dimension $\dim N =
  \frac{1}{2}\, \dim W$ that is not tangent to any $J$-complex line,
  that means, $TN \cap (J\,TN) = \{0\}$, which is equivalent to
  requiring
  \begin{equation*}
    \restricted{TW}{N} = TN \oplus (J\,TN) \;.
  \end{equation*}
  \index{submanifold!totally real}
\end{definition}

\begin{proposition}
  Let $(W,J)$ be an almost complex manifold with $J$-convex boundary
  $(M, \xi)$.
  Assume $N$ is a submanifold of $M$ for which the complex
  tangencies~$\xi$ induce the Legendrian foliation $\fF = TN \cap
  \xi$.
  Then it is easy to check that $N\setminus \sing(\fF)$ is totally
  real.
\end{proposition}
\begin{proof}
  If $X \in TN$ is a non-vanishing vector with $JX$ also in $TN$, then
  in particular
  \begin{equation*}
    X \in TN \cap (JTN) \subset TM \cap (JTM) = \xi \;,
  \end{equation*}
  so that $X$ and $JX$ have to lie in $\fF$.
  The $2$-form $d\alpha$ tames $\restricted{J}{\xi}$ so that
  $d\alpha(X, JX) > 0$, but $\restricted{d\alpha}{\fF}$ vanishes at
  regular points of the foliation, and hence $X$ must be $0$.
\end{proof}

We will next study the restrictions imposed by a Legendrian foliation
on $J$-holomorphic curves.
Let $(\Sigma, j)$ be a compact Riemann surface with boundary, and let
$A$ be a subset of an almost complex manifold~$(W,J)$.
We introduce for $J$-holomorphic maps $u\colon \Sigma \to W$ with
$u(\p\Sigma) \subset A$ the notation
\begin{equation*}
  u\colon (\Sigma, \p\Sigma, j) \to (W, A, J) \;.
\end{equation*}
Note that we are always supposing that $u$ is at least $C^1$ along the
boundary.

\begin{corollary}\label{cor: boundary of curves transverse to
    Legendrian foliation}
  Let $(W,J)$ be an almost complex manifold with convex boundary~$(M,
  \xi)$.
  Let $N \hookrightarrow M$ be a submanifold with an induced
  Legendrian foliation~$\fF$, and let $u$ be a $J$-holomorphic map
  \begin{equation*}
    u\colon (\Sigma, \p\Sigma, j) \to
    (W, N\setminus \sing(\fF), J) \;.
  \end{equation*}
  If there is an interior point $z_0\in \Sigma \setminus \p \Sigma$ at
  which $u$ touches $M$, or if $\p u$ is not positively transverse to
  $\fF$, then $u$ is a constant map.
\end{corollary}
\begin{proof}
  Choose a $J$-plurisubharmonic function $h\colon W\to \RR$ that is a
  regular equation for $M$.
  The first implication follows directly from Corollary~\ref{cor:
    holomorphic maps and interior maxima}, because $z_0$ would be an
  interior maximum for $h\circ u$.
  For the second implication note first that $h\circ u$ takes its
  maximum on $\p \Sigma$ so that if $u$ is not constant, we have by
  Corollary~\ref{cor: boundary point lemma holomorphic maps} that the
  derivative $\lie{v} (h\circ u)$ is strictly positive for every point
  $z_1\in \partial \Sigma$ and every vector $v\in T_{z_1}\Sigma$
  pointing out of $\Sigma$.
  Now if $w \in T\Sigma$ is a vector that is tangent to $\p \Sigma$
  such that $j w$ points inward (so that $w$ corresponds to the
  boundary orientation of $\p \Sigma$, because $(-jw, w)$ is a
  positive basis of $T\Sigma$), we obtain
  \begin{equation*}
    \alpha(Du\cdot w) = - dh(J Du\cdot w) = - dh(Du\cdot jw) =
    - d(h\circ u) (jw) > 0 \;.
  \end{equation*}
  The boundary of $\p u$ has thus to be positively transverse to
  $\xi$, and so it is in particular positively transverse to the
  Legendrian foliation~$\fF$.
\end{proof}

Note that the result above applies only for holomorphic maps that are
$C^1$ along the boundary.

\subsection{Preliminaries: $\omega$-convexity}\label{sec:
  omega-convexity}

Above we have explained the notion of $J$-convexity, and the relevant
relationship between contact and almost complex structures.
In this section, we want to discuss the notion of $\omega$-convexity,
that means the relationship between an (almost) symplectic and a
contact structure.
In fact, we are not interested in studying almost complex manifolds
for their own sake, but we would like to use the almost complex
structure to understand instead a symplectic manifold $(W,\omega)$.
As initiated by Gromov, we introduce an auxiliary almost complex
structure to be able to study $J$-holomorphic curves in the hope that
even though the $J$-holomorphic curves depend very strongly on the
almost complex structure chosen, we'll be able to extract interesting
information about the initial symplectic structure.
For this strategy to work, we need the almost complex structure to be
\defin{tamed} by $\omega$, that means, we want
\begin{equation*}
  \omega(X, JX) > 0
\end{equation*}
for every non-vanishing vector $X\in TW$. \index{almost complex
  structure!tamed}
This tameness condition is important, because it allows us to control
the limit behavior of sequences of holomorphic curves (see
Section~\ref{sec: compactness}).
As explained in the previous section, $J$-convexity is a property that
greatly helps us in understanding holomorphic curves in ambient
manifolds that have boundary.
When $(W,\omega)$ is a symplectic manifold with boundary~$M = \p W$,
we would thus like to chose an almost complex structure~$J$ that is
\begin{itemize}
\item tamed by $\omega$, and
\item that makes the boundary $J$-convex.
\end{itemize}
In particular, if such a $J$ exists, we know that the boundary admits
an induced contact structure
\begin{equation*}
  \xi = TM \cap \bigl(J\cdot TM\bigr) \;.
\end{equation*}
From the symplectic or contact topological view point, the opposite
setup would be more natural though: given a symplectic manifold~$(W,
\omega)$ with contact boundary $(M,\xi)$, can we choose an almost
complex structure~$J$ that is tamed by $\omega$, and that makes the
boundary $J$-convex such that $\xi$ is the bundle of $J$-complex
tangencies?
The general answer to that question was given in
\cite{WeafFillabilityHigherDimension}.

\begin{definition}
  Let $(M,\xi)$ be a cooriented contact manifold of dimension~$2n-1$,
  and let $(W,\omega)$ be a symplectic manifold whose boundary is $M$.
  Let $\alpha$ be a positive contact form for $\xi$, and assume that
  the orientation induced by $\alpha \wedge d\alpha^{n-1}$ on $M$
  agrees with the boundary orientation of $(W,\omega)$.
  We call $(W, \omega)$ a \defin{weak symplectic filling} of
  $(M,\xi)$, if
  \begin{equation*}
    \alpha\wedge \bigl(T\,d\alpha + \omega\bigr)^{n-1} > 0
  \end{equation*}
  for every $T\in [0,\infty)$.
  \index{filling!weak symplectic}
\end{definition}

The proofs of the following statements are very lengthy, hence we will
omit the proofs referring instead to the Appendix of
\cite{WeafFillabilityHigherDimension} for more details.

\begin{theorem}
  Let $(M,\xi)$ be a cooriented contact manifold, and let $(W,\omega)$
  be a symplectic manifold with boundary $M = \p W$.
  The following two statements are equivalent
  \begin{itemize}
  \item $(W,\omega)$ is a weak symplectic filling of $(M,\xi)$.
  \item There exists an almost complex structure~$J$ on $W$ that is
    tamed by $\omega$ and that makes $M$ a $J$-convex boundary whose
    $J$-complex tangencies are $\xi$.
  \end{itemize}
  Furthermore the space of all almost complex structures that satisfy
  these conditions is contractible (if non-empty).
\end{theorem}

A weak filling is a notion that is relatively recent in higher
dimensions; traditionally it is the concept of a strong symplectic
filling that has been studied for a much longer time.
Let $(W,\omega)$ be a symplectic manifold.
A vector field~$X_L$ is called a \defin{Liouville vector field}, if it
satisfies the equation
\begin{equation*}
  \lie{X_L}\omega = \omega\;.
\end{equation*}
\index{vector field!Liouville}

\begin{definition}
  Let $(M,\xi)$ be a cooriented contact manifold, and let $(W,\omega)$
  be a symplectic manifold whose boundary is $M$.
  We call $(W, \omega)$ a \defin{strong symplectic filling} of
  $(M,\xi)$, if there exists a Liouville vector field~$X_L$ on a
  neighborhood of $M$ such that $\lambda :=
  \restricted{\bigl(\iota_{X_L}\omega\bigr)}{TM}$ is a positive
  contact form for $\xi$.
  \index{filling!strong symplectic}
\end{definition}

It is easy to see that a strong filling is in particular a weak
filling.
Note that the symplectic form of a strong filling becomes always exact
when restricted to the boundary, but that this needs not be true for a
weak filling; if it is then it will usually still not be a strong
symplectic filling, but by Corollary~\ref{cor: exact cobordism} it can
deformed into one.

\begin{lemma}
  Let $(W,\omega)$ be a symplectic manifold and let $M$ be a
  hypersurface (possibly a boundary component of $W$) together with a
  non-vanishing $1$-form~$\lambda$.
  Assume that the restriction of $\omega$ to $\ker \lambda$ is
  symplectic.
  Then there is a tubular neighborhood of $M$ in $W$ that is
  symplectomorphic to the model
  \begin{equation*}
    \bigl((-\epsilon, \epsilon) \times M, \,
    d(t\,\lambda) + \restricted{\omega}{TM}\bigr) \;,
  \end{equation*}
  where $t$ is the coordinate on the interval $(-\epsilon, \epsilon)$.
  The $0$-slice $\{0\}\times M$ corresponds in this identification to
  the hypersurface~$M$.
  If $M$ is a boundary component of $W$ then of course we need to
  replace the model by $(-\epsilon, 0] \times M$ or by $[0, \epsilon)
  \times M$ depending on whether $\lambda\wedge \omega^{n-1}$ is
  oriented as the boundary of $(W,\omega)$ or not.
\end{lemma}
For the proof see \cite[Lemma~2.6]{WeafFillabilityHigherDimension}.

\begin{proposition}\label{prop: cohomolog collar}
  Let $(W,\omega)$ be a weak filling of a contact manifold $(M,\xi)$,
  and let $\Omega$ be a $2$-form on $M$ that is cohomologous to
  $\restricted{\omega}{TM}$.
  Choose a positive contact form~$\alpha$ for $(M,\xi)$.
  Then if we allow $C >0$ to be sufficiently large, we can attach a
  collar $[0,C]\times M$ to $W$ with a symplectic form $\omega_C$ that
  agrees close to $\{C\}\times M$ with $d\bigl(t \alpha\bigr) +
  \Omega$, and such that the new manifold is a weak filling of
  $\bigl(\{t_0\} \times M, \xi\bigr)$ for every $t_0 \in [0, C]$.
\end{proposition}

The proof can be found in
\cite[Lemma~2.10]{WeafFillabilityHigherDimension}.

\begin{corollary}\label{cor: exact cobordism}
  Let $(W,\omega)$ be a weak symplectic filling of $(M,\xi)$ and
  assume that $\omega$ restricted to a neighborhood of $M$ is an exact
  symplectic form.
  Then we may deform $\omega$ on a small neighborhood of $M$ such it
  becomes a strong symplectic filling.
\end{corollary}
\begin{proof}
  Since $\restricted{\omega}{TM}$ is exact, we can apply the
  proposition above with $\Omega = 0$.
  Afterwards we can isotope the collar back into the neighborhood of the
  boundary of $W$.
\end{proof}

Note that two contact structures that are strongly filled by the same
symplectic manifold are isotopic, while a symplectic manifold may be a
weak filling of two different contact manifolds.
This is true even when the restriction of the symplectic structure to
the boundary is exact, see
\cite[Remark~2.11]{WeafFillabilityHigherDimension}.

\section{Holomorphic curves and Legendrian foliations}\label{sec:
  local_model}

Let $(W,J)$ be an almost complex manifold with $J$-convex
boundary~$(M,\xi)$, and let $N \subset M$ be a submanifold carrying a
Legendrian foliation~$\fF$.
The aim of this section will be to better understand the behavior of
$J$-holomorphic maps
\begin{equation*}
  u\colon (\Sigma, \p \Sigma, j) \to (W, N, J) \;,
\end{equation*}
that lie close to a singular point~$p\in \sing(\fF)$ of the Legendrian
foliation.
For this we will assume that $J$ is of a very specific form in a
neighborhood of the point $p$.

\subsection{Existence of $J$-convex functions close to totally real
  submanifolds}

As a preliminary tool, we will need the following result.

\begin{proposition}\label{prop: totally real and psh functions}
  Let $(W,J)$ be an almost complex structure that contains a closed
  totally real submanifold~$L$.
  Then there exists a smooth function $f\colon W \to [0,\infty)$ with
  $L = f^{-1}(0)$ that is $J$-plurisubharmonic on a neighborhood of
  $L$.
  In particular, it follows that $df_p = 0$ at every point $p\in L$.
\end{proposition}
\begin{proof}
  We will first show that we find around every point $p\in L$ a chart
  $U$ with coordinates $\{(x_1,\dotsc,x_n; y_1,\dotsc,y_n)\} \subset
  \RR^{2n}$ such that $L\cap U = \{y_1 = \dotsb = y_n = 0\}$ and
  \begin{equation*}
    \restricted{J\,\frac{\partial}{\partial x_j}}{L\cap U}
    = \restricted{\frac{\partial}{\partial y_j}}{L\cap U} \;.
  \end{equation*}
  For this, start by choosing coordinates $\{(x_1,\dotsc,x_n)\}\subset
  \RR^n$ for the submanifold~$L$ around the point~$p$, and consider
  the associated vector fields
  \begin{equation*}
    Y_1 = J\, \frac{\partial}{\partial x_1}, \dotsc,
    Y_n = J\, \frac{\partial}{\partial x_n}
  \end{equation*}
  along $L$.
  These vector fields are everywhere linearly independent and
  transverse to $L$, hence, we can define a smooth map from from a
  small ball around $0$ in $\RR^{2n} = \{(x_1,\dotsc,x_n;
  y_1,\dotsc,y_n)\}$ to $W$ by
  \begin{equation*}
    y_1\,Y_1(x_1,\dotsc,x_n) + \dotsm +  y_n\,Y_1(x_1,\dotsc,x_n)
    \mapsto \exp \bigl(y_1\,Y_1 + \dotsm +  y_n\,Y_1\bigr) \;,
  \end{equation*}
  where $\exp$ is the exponential map for an arbitrary Riemannian
  metric on $W$.
  If the ball is chosen sufficiently small, the map will be a chart
  with the desired properties.
  For such a chart~$U$, we will choose a function
  \begin{equation*}
    f_U\colon U \to [0,\infty), \, (x_1,\dotsc,x_n;y_1,\dotsc,y_n) \mapsto
    \frac{1}{2}\,\bigl(y_1^2 + \dotsm +y_n^2\bigr) \;.
  \end{equation*}
  It is obvious that both the function itself, and its differential
  vanish along $L\cap U$.
  Furthermore $f$ is plurisubharmonic close to $L\cap U$, because
  \begin{align*}
    dd^J f_U &= d\bigl(y_1\,d^J y_1 + \dotsm + y_n\,d^J y_n\bigr) \\
    &= dy_1\wedge d^J y_1 + \dotsm + dy_n\wedge d^J y_n + y_1\,dd^J
    y_1 + \dotsm + y_n\,dd^J y_n
    \intertext{simplifies at $L\cap U$ to}
    \restricted{dd^J f_U}{L\cap U} &= dx_1\wedge dy_1 + \dotsm +
    dx_n\wedge dy_n \;,
  \end{align*}
  where we have used that all $y_j$ vanish, and that
  $J\,\frac{\partial}{\partial x_j} = \frac{\partial}{\partial y_j}$
  and $J\,\frac{\partial}{\partial y_j} = J^2\,
  \frac{\partial}{\partial x_j} = - \frac{\partial}{\partial x_j}$.
  It is easy to check that this $2$-form evaluates positively on
  complex lines along $L\cap U$, and hence also in a small
  neighborhood of $p$.
  Now to obtain a global plurisubharmonic function as stated in the
  proposition, cover $L$ with finitely many charts $U_1, \dotsc, U_N$,
  each with a function $f_1,\dotsc,f_N$ according to the construction
  given above.
  Choose a subordinate partition of unity $\rho_1,\dotsc,\rho_N$, and
  define
  \begin{equation*}
    f = \sum_{j=1}^N \rho_j\cdot f_j \;.
  \end{equation*}
  The function $f$ and its differential $df = \sum_{j=1}^N
  \bigl(\rho_j\, df_j + f_j\, d\rho_j\bigr)$ vanish along $L$ so that
  the only term in
  \begin{equation*}
    \begin{split}
      dd^J f &= d\sum_{j=1}^N \bigl(\rho_j\, d^Jf_j
      + f_j\, d^J\rho_j\bigr) \\
      &= \sum_{j=1}^N \bigl(\rho_j\, dd^Jf_j + d\rho_j\wedge d^Jf_j +
      f_j\, dd^J\rho_j + df_j\wedge d^J\rho_j\bigr)
    \end{split}
  \end{equation*}
  that survives along $L$ is the first one, giving us along $L$
  \begin{equation*}
    dd^J f = \sum_{j=1}^N \rho_j\, dd^Jf_j \;.
  \end{equation*}
  This $2$-form is positive on $J$-complex lines, and hence there is a
  small neighborhood of $L$ on which $f$ is plurisubharmonic.
  Finally, we modify $f$ to be positive outside this small
  neighborhood so that we have $L = f^{-1}(0)$ as required.
\end{proof}

\begin{corollary}\label{coro: holomorphic curves in cotangent}
  Let $(W,J)$ be an almost complex structure that contains a closed
  totally real submanifold~$L$.
  Then we find a small neighborhood~$U$ of $L$ for which every
  $J$-holomorphic map
  \begin{equation*}
    u\colon (\Sigma, \p\Sigma, j) \to \bigl(W, L, J\bigr)
  \end{equation*}
  from a compact Riemann surface needs to be constant if
  $u(\Sigma)\subset U$.
\end{corollary}
\begin{proof}
  Let $f\colon W \to [0,\infty)$ be the function constructed in
  Proposition~\ref{prop: totally real and psh functions}, and let
  $U\subset (W, J)$ be the neighborhood of $L$, where $f$ is
  $J$-plurisubharmonic.
  Because $u(\Sigma) \subset U$, we obtain from Corollary~\ref{cor:
    holomorphic maps and interior maxima} that $f\circ u$ must take
  its maximum on the boundary of $\Sigma$, but because $f\circ u$ is
  zero on all of $\partial \Sigma$, it follows that $f\circ u$ will
  vanish on the whole surface~$\Sigma$.
  The image $u(\Sigma)$ lies then in the totally real submanifold~$L$,
  and this implies that the differential of $u$ vanishes everywhere.
  Hence there is a $\bfQ_0 \in L$ with $u(z) = \bfQ_0$ for all $z\in
  \Sigma$.
\end{proof}

\subsection{$J$-holomorphic curves close to elliptic singularities of
  a Legendrian foliation}

The aim of this section will be to show that for a suitable choice of
an almost complex structure, elliptic singularities give birth to a
family of holomorphic disks, and that apart from these disks and their
branched covers, no other holomorphic disks may get close to the
elliptic singularities.
Before studying the higher dimensional case, we will construct a model
situation for a $4$-dimensional almost complex manifold with convex
boundary.

\subsubsection{Dimension~$4$}

Consider $\CC^2$ with its standard complex structure~$i$.
Then it is easy to check that $h(z_1,z_2) = \frac{1}{2}\,
\bigl(\abs{z_1}^2 + \abs{z_2}^2\bigr)$ is a plurisubharmonic function
whose regular level sets are the concentric spheres around the origin.
We choose the level set $M = h^{-1}(1/2)$, that is, the boundary of
the closed unit ball $W : = h^{-1}\bigl((-\infty,1/2]\bigr)$ that is
$i$-convex and has the induced contact form
\begin{equation*}
  \alpha_0 = \restricted{d^ih}{TM} = x_1\,dy_1 -  y_1\, dx_1
  + x_2\, dy_2 - y_2\, dx_2 \;.
\end{equation*}
We only want to study a neighborhood~$U$ of $(0,1)$ in $W$.
Embed a small disk by the map
\begin{equation*}
  \Phi\colon z \mapsto \bigl(z, \sqrt{1- \abs{z}^2}\bigr)
\end{equation*}
into $M\cap U$, and denote the image of $\Phi$ by $N_0$.
This submanifold is the intersection of $M = \SSS^3$ with a hyperplane
whose $z_2$-coordinate is purely real.
The restriction of $\alpha_0$ to $N_0$ reduces to
\begin{equation}
  \restricted{\alpha_0}{TN_0} = \Phi^*\alpha_0 = x\,dy - y\, dx \;,
  \label{eq: foliation on disk}
\end{equation}
so that the Legendrian foliation has at the origin an elliptic
singularity (of the type described in Section~\ref{sec: codim 2
  sing}).
Let $U$ be the subset
\begin{equation*}
  U = \bigl\{(z_1,z_2)\in \CC^2\bigm|\, \RealPart (z_2) > 1 - \delta\bigr\}
  \cap h^{-1}\bigl((-\infty, 1/2]\bigr)
\end{equation*}
for small $\delta > 0$, that means, we take the unit ball and cut off
all points under a certain $x_2$-height.
The following propositions explain that there is essentially a unique
holomorphic disk with boundary in $N_0$ passing through a given point
$(z_1,z_2)\in N_0 \cap U$.
All other holomorphic curves with the same boundary condition will
either be constant or will be (branched) covers of that disk.

\begin{proposition}\label{prop: representative for Bishop disk}
  Denote the intersection of $U$ with the complex plane $\CC\times
  \{x\}$ for $x \in (1 - \delta,1)$ by $L_x$.
  For every $x_2 \in (1 - \delta,1)$, there exists a unique
  \emph{injective} holomorphic map
  \begin{equation*}
    u_{x_2}\colon (\DD^2, \p\DD^2) \to \bigl(L_{x_2}, \p L_{x_2}\bigr)
  \end{equation*}
  that satisfies $u_{x_2}(0) = (0,x_2)$ and $u_{x_2}(1) \in
  \{(x_1,x_2) \in U |\, x_1 > 0\}$.
\end{proposition}

The last two conditions only serve to fix a parametrization of a given
geometric disk.

\begin{proof}
  The desired map~$u_{x_2}$ can be explicitly written down as
  \begin{equation*}
    u_{x_2}(z) = \bigl(C z, x_2\bigr)
  \end{equation*}
  with $C = \sqrt{1-x_2^2}$.
  To prove uniqueness assume that there were a second holomorphic map
  \begin{equation*}
    \tilde u_{x_2}\colon (\DD^2, \p\DD^2) \to
    \bigl(L_{x_2}, \p L_{x_2}\bigr)
  \end{equation*}
  with the required properties.
  It is clear that $L_{x_2} = \{(x+iy,x_2)\in \CC^2 \bigm|\, x^2 + y^2
  \le 1- x_2^2\}$ is a round disk.
  By Corollary~\ref{cor: boundary of curves transverse to Legendrian
    foliation}, the restriction $\restricted{u_{x_2}}{\p \DD^2}$ of
  the map to the boundary has non-vanishing derivative, and it is by
  assumption injective, hence it is a diffeomorphism onto $\p
  L_{x_2}$.
  This proves that $u_{x_2}$ has to be for topological reasons
  surjective on $L_{x_2}$ (otherwise we could construct a retract of
  the disk onto its boundary).
  Note also that the germ of a holomorphic map around the origin in
  $\CC$ is always biholomorphic to $z\mapsto z^k$ for some integer
  $k\in \NN_0$, so that the differential of $u_{x_2}$ may not vanish
  anywhere, because otherwise $u_{x_2}$ could not be injective.
  Together this allows us to define a biholomorphism
  \begin{equation*}
    \varphi := u_{x_2}^{-1} \circ \tilde u_{x_2}\colon
    (\DD^2, \p \DD^2) \to (\DD^2, \p \DD^2)
  \end{equation*}
  with $\varphi(0) = 0$ and $\varphi(1) = 1$, but the only
  automorphism of the disk with these properties is the identity, thus
  showing that $u_{x_2} = \tilde u_{x_2}$.
\end{proof}

\begin{proposition}\label{prop: hol curves in neighborhood elliptic
    sing dim 4}
  Let
  \begin{equation*}
    u\colon (\Sigma, \p\Sigma; j) \to (U, N_0; i)
  \end{equation*}
  be any holomorphic map from a connected compact Riemann
  surface~$(\Sigma,j)$ to $U$ with $u(\p\Sigma) \subset N_0$.
  Either $u$ is constant or its image is one of the slices~$L_{x_2} =
  U\cap \bigl(\CC\times \{x_2\}\bigr)$.
  If $u$ is injective at one of its boundary points, then $\Sigma$
  will be a disk, and after a reparametrization by a Möbius
  transformation, $u$ will be equal to the map~$u_{x_2}$ given in
  Proposition~\ref{prop: representative for Bishop disk}.
\end{proposition}
\begin{proof}
  Note that we are supposing that $u$ is at least $C^1$ on the
  boundary so that by Corollary~\ref{cor: boundary of curves
    transverse to Legendrian foliation} the map~$u$ will be constant
  if it touches the elliptic singularity in $N$.
  The proof of the proposition will be based on the harmonicity of the
  coordinate functions~$x_1$, $y_1$, $x_2$, and $y_2$.
  Let $f\colon U \to \RR$ be the function $(z_1,z_2) \mapsto y_2 =
  \ImaginaryPart (z_2)$.
  Since $\Sigma$ is a compact domain, the function $f\circ u$ attains
  somewhere on $\Sigma$ its maximum and its minimum, and applying the
  maximum principle, Corollary~\ref{cor: holomorphic maps and interior
    maxima}, to $f\circ u$ itself and also to $-f\circ u$, we obtain
  that both the maximum and the minimum have to lie on $\p \Sigma$.
  But since $u(\p\Sigma) \subset N_0$ has vanishing imaginary
  $z_2$-part, it follows that $f\circ u \equiv 0$ on the whole
  surface.
  Using now the Cauchy-Riemann equations, it immediately follows that
  the real part of the $z_2$-coordinate of $u$ has to be constant
  everywhere.
  We can deduce that the image of $u$ has to lie in one of the slices
  $L_{x_2} = \CC\times \{x_2\}$, and in particular the boundary
  $u(\p\Sigma)$ lies in the circle $\p L_{x_2} = \bigl\{(x+iy, x_2)\in
  \CC^2\bigm|\, x^2 + y^2 = 1- x_2^2 \bigr\}$.
  Assume that $u$ is not constant.
  Since $u$ lies in $L_{x_2}$, we can use the map $u_{x_2}$ from
  Proposition~\ref{prop: representative for Bishop disk}, to define a
  holomorphic map
  \begin{equation*}
    \varphi := u_{x_2}^{-1} \circ u \colon (\Sigma, \p\Sigma) \to
    (\DD^2, \p \DD^2) \;.
  \end{equation*}
  If $u$ were not surjective on $L_{x_2}$, we could suppose (after a
  Möbius transformation on the target space) that the image of
  $\varphi$ does not contain $0$.
  The function $h(z) = - \ln \abs{z}$ on $\DD^2 \setminus \{0\}$ is
  harmonic, because it is locally the real part of a holomorphic
  function, and because $h\circ \varphi$ would have its maximum on the
  interior of $\Sigma$, we obtain that $h\circ \varphi$ is constant,
  so that the image of $\varphi$ lies in $\p \DD^2$.
  The image of a non-constant holomorphic map is open, and hence $u$
  must be constant.
  Assume now that $u$ is injective at one of its boundary points.
  As we have shown in Proposition~\ref{prop: representative for Bishop
    disk} the restriction $\restricted{u}{\p \Sigma}\colon \p\Sigma
  \to \p L_{x_2}$ will be a diffeomorphism for each component of $\p
  \Sigma$ so that $\p\Sigma$ must be connected.
  Furthermore, it follows that $u$ will also be injective on a small
  neighborhood of $\p L_{x_2}$, because if we find two sequences
  $(z_k)_k$ and $(\tilde z_k)_k$ coming arbitrarily close to
  $\p\Sigma$ with $u(z_k) = u(\tilde z_k)$ for every $k$, then after
  assuming that they both converge (reducing if necessary to
  subsequences), we see by continuity that $\lim u(z_k) = \lim
  u(\tilde z_k)$ and $\lim z_k, \lim \tilde z_k \in \p \Sigma$, so
  that we can conclude that $\lim z_k = \lim \tilde z_k$.
  Using that the differential of $u$ in $\lim z_k$ is not singular, we
  obtain that for $k$ sufficiently large, we will always have $z_k =
  \tilde z_k$ showing that $u$ is indeed injective on a small
  neighborhood of $\p\Sigma$.
  Assume $z_0\in \Sigma$ is a point at which the
  differential~$D\varphi$ vanishes.
  Then we know that $\varphi$ can be represented in suitable charts as
  $z\mapsto z^k$ for some $k\in \NN$, but if $k > 1$ this yields a
  contradiction, because we know that $\varphi$ is a biholomorphism on
  a neighborhood of $\p \Sigma$, and hence its degree must be $1$.
  Since $\varphi$ is holomorphic, it preserves orientations, so that on
  the other hand, we would have that the degree would need to be
  \emph{at least} $k$, if there were such a critical point.
  We obtain that $\varphi$ has nowhere vanishing differential, and
  hence it must be a regular cover, but since it is of degree $1$, it
  is in fact a biholomorphism, and $\Sigma$ must be a disk.
\end{proof}

\subsubsection{The higher dimensional situation}\label{sec: elliptic
  singularity high dimension local model}

In this section, $L$ will always be a closed manifold, and we will
choose for $T^*L$ an almost complex structure~$J_L$ for which the
$0$-section~$L$ is totally real, so that there is by
Proposition~\ref{prop: totally real and psh functions} a
function~$f_L\colon T^*L \to [0,\infty)$ that vanishes on $L$ (and
only on $L$) and that is plurisubharmonic on a small neighborhood of
$L$.
As before, we will first describe a very explicit manifold that will
serve as a model for the neighborhood of an elliptic singularity.
Let $\CC^2 \times T^*L$ be the almost complex manifold with almost
complex structure $J = i\oplus J_L$, where $i$ is the standard complex
structure on $\CC^2$.
We define a function~$f\colon \CC^2 \times T^*L \to [0,\infty)$ by
\begin{equation*}
  f(z_1,z_2, \bfQ, \bfP) = \frac{1}{2}\,\bigl(\abs{z_1}^2 + \abs{z_2}^2\bigr)
  + f_L(\bfQ, \bfP)
\end{equation*}
If we stay in a sufficiently small neighborhood of the $0$-section of
$T^*L$, this function is clearly $J$-plurisubharmonic and we denote
its regular level set $f^{-1}\bigl(1/2\bigr)$ by $M$; its contact form
is given by
\begin{equation*}
  \alpha := \restricted{d^J f}{TM} = \restricted{\bigl(x_1\,dy_1 -
    y_1\,dx_1 + x_2\,dy_2 - y_2\,dx_2 + d^{J_L} f_L\bigr)}{TM} \;.
\end{equation*}
Now we define a submanifold~$N$ in $M$ as the image of the map
\begin{equation*}
  \Phi\colon \DD^2 \times L \hookrightarrow M
  \subset \CC^2 \times T^*L
\end{equation*}
given by $\Phi\bigl(z;\bfQ\bigr) = \Bigl(z, \sqrt{1-\abs{z}^2}; \bfQ,
\0\Bigr)$, that means, the image of $\Phi$ is the product of the
$0$-section in $T^*L$ and the submanifold~$N_0$ given in the previous
section.
The submanifold has a Legendrian foliation~$\fF$ induced by
\begin{equation*}
  \restricted{\alpha}{TN} = \Phi^* d^J f = x\,dy - y\,dx \;.
\end{equation*}
In particular, the leaves of the foliation are parallel to the
$L$-factor in $\DD^2 \times L$ and $\fF$ has an elliptic singularity
in $\{0\}\times L$.
Note that both the almost complex structure as well as the
submanifold~$N$ split as a product, thus if we consider a holomorphic
map
\begin{equation*}
  u\colon (\Sigma, \p\Sigma; j) \to
  \bigl(\CC^2 \times T^*L, N; J\bigr) \;,
\end{equation*}
we can decompose it into $u=(u_1, u_2)$ with
\begin{align*}
  u_1\colon& (\Sigma, \p\Sigma; j) \to (\CC^2, N_0; i) \\
  u_2\colon& (\Sigma, \p\Sigma; j) \to (T^*L, L; J_L) \;.
\end{align*}
This allows us to treat each factor independently from the other one,
and we will easily be able to obtain similar results as in the
previous section.
Since we are interested in finding a local model, we will first
restrict our situation to the following subset
\begin{equation} \label{eq: model neighborhood}
  U := \bigl\{ (z_1,z_2; \bfQ, \bfP)\bigm|
  \RealPart(z_2) \ge 1 - \delta\bigr\} \cap
  f^{-1}\bigl([0,1/2]\bigr)
\end{equation}
that is, for $\delta$ sufficiently small, a compact neighborhood of
$N$ in $f^{-1}\bigl([0,1/2]\bigr)$, because the points $(z_1,z_2;
\bfQ, \bfP)$ in $U$ satisfy
\begin{equation*}
  0 \le \frac{1}{2}\,\abs{z_1}^2 
  + f_L(\bfQ, \bfP) \le \frac{1}{2}\, \bigl(1 -  \abs{z_2}^2\bigr) \le
  \delta - \frac{1}{2}\, \delta^2 \le \delta
\end{equation*}
so that all coordinates are bounded.
Note in particular, that this bound on the $\bfP$-coordinates
guarantees that $f$ will be $J$-plurisubharmonic on $U$.
The submanifold $N\cap U$ can also be written in the following easy
form
\begin{equation*}
  \bigl\{ (z, x_2; \bfQ, \0)\bigm| \text{ $x_2 \ge 1 - \delta$
    and $\abs{z}^2 = 1 - x_2^2$}\bigr\} \times L \;.
\end{equation*}

\begin{corollary}\label{cor: hol curves in neighborhood elliptic sing
    high dim}
  Let
  \begin{equation*}
    u\colon (\Sigma, \p\Sigma, j) \to \bigl(U, N\cap U; J\bigr)
  \end{equation*}
  be any holomorphic map from a connected compact Riemann
  surface~$(\Sigma,j)$ to $U$ with $u(\p\Sigma) \subset N$.
  Either $u$ is constant or its image is one of the slices~$L_{x_2,
    \bfQ_0} = \Bigl( \CC\times \{x_2\} \times \{\bfQ_0\}\Bigr) \cap U$
  with $x_2\in [1-\delta, 1)$ and $\bfQ_0$ a point on the $0$-section
  of $T^*L$.
  If $u$ is injective at one of its boundary points, then $\Sigma$
  will be a disk, and $u$ is equal to
  \begin{equation*}
    u(z) = \bigl(u_{x_2}\circ \varphi(z); \bfQ_0, \0\bigr) \;,
  \end{equation*}
  where
  ~$u_{x_2}$ is the map given in Proposition~\ref{prop: representative
    for Bishop disk}, and $\varphi$ is a Möbius transformation of the
  unit disk.
\end{corollary}
\begin{proof}
  Let $u$ be a $J$-holomorphic map as in the statement.
  We will study $u$ by decomposing it into $u = \bigl(u_{\CC^2},
  u_{T^*L}\bigr)$ with
  \begin{align*}
    u_{\CC^2}\colon& (\Sigma, \p\Sigma, j) \to (\CC^2, N, i) \\
    u_{T^*L}\colon& (\Sigma, \p\Sigma, j) \to (T^*L, L, J_L) \;.
  \end{align*}
  Using that $f_L$ is $J_L$-plurisubharmonic on the considered
  neighborhood of the $0$-section contained in $U$, it follows from
  Corollary~\ref{coro: holomorphic curves in cotangent} that
  $u_{T^*L}$ is constant.
  Once we know that $u_{T^*L}$ is constant, the situation for
  $u_{\CC^2}$ is identical to the one in Proposition~\ref{prop: hol
    curves in neighborhood elliptic sing dim 4}, so that we obtain the
  desired result.
\end{proof}

The results obtained so far only explain the behavior of holomorphic
curves that are completely contained in the model neighborhood~$U$.
Next we will extend this result to show that a holomorphic curve is
either disjoint from the subset~$U$ or is lies completely inside $U$.
Assume $(W,J)$ is a compact almost complex manifold with convex
boundary $M = \p W$.
Let $N$ be a submanifold of $M$, and assume that there is a compact
subset~$U$ in $W$ such that $U$ is diffeomorphic to the model above,
with $M\cap U$, $N\cap U$ and $\restricted{J}{U}$ all being equal to
the corresponding objects in our model neighborhood.

\begin{proposition}\label{prop: no hol curve may enter model
    neighborhood from outside}
  Let
  \begin{equation*}
    u\colon (\Sigma, \p\Sigma; j) \to (W, N; J)
  \end{equation*}
  be a holomorphic map, and let $U$ be a compact subset of $W$ that
  agrees with the model described above.
  If $u(\Sigma)$ intersects $U$, then it has to lie entirely in $U$,
  and it will be consequently of the form given by Corollary~\ref{cor:
    hol curves in neighborhood elliptic sing high dim}.
\end{proposition}
\begin{proof}
  Assume $u$ to be a holomorphic map whose image lies partially in
  $U$.
  The set~$U$ is a compact manifold with corners, and we write $\p U =
  \p_M U \cup \p_W U$, where
  \begin{align*}
    \p_M U &= U \cap M \\
    \p_W U &= \bigl\{ (z_1,z_2; \bfQ, \bfP)\bigm| \RealPart(z_2) \ge 1
    - \delta\bigr\} \cap f^{-1}\bigl([0,1/2]\bigr)
  \end{align*}
  We will show that the real part of the $z_2$-coordinate of $u$ needs
  to be constant.
  This then proves the proposition, because it prevents $u$ from
  leaving $U$.

\begin{figure}[htbp]
  \centering
  \includegraphics[height=4.5cm,keepaspectratio]{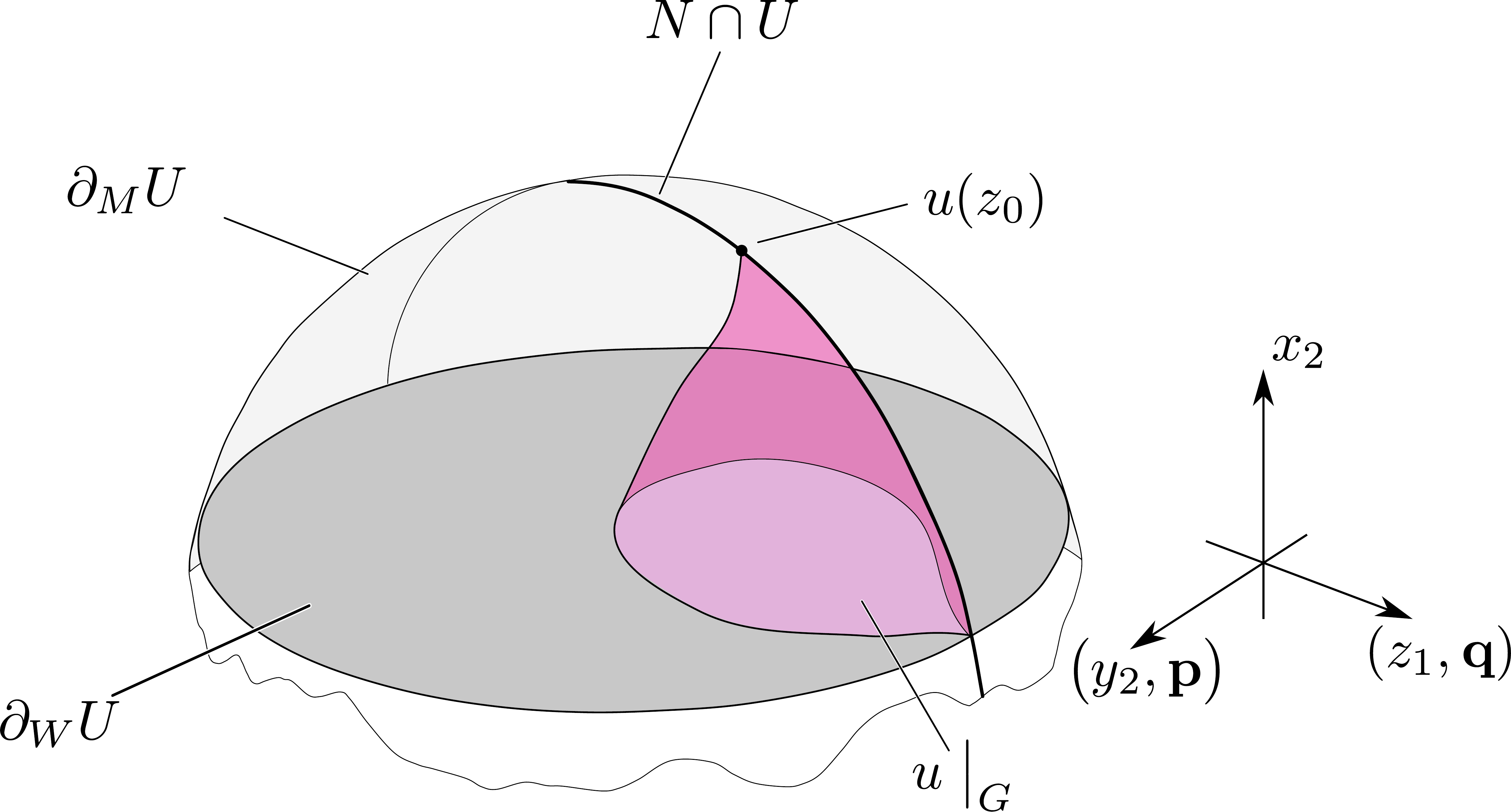}
  \caption{}\label{fig: model_nbhd_sing_set}
\end{figure}

  Thus assume instead that the real part of $z_2$ does vary on $u$.
  Slightly decreasing the cut-off level~$\delta$ in \eqref{eq: model
    neighborhood} using Sard's theorem, the holomorphic map $u$ will
  intersect $\p_W U$ transversely, so that $u^{-1}\bigl(\p_W U\bigr)$
  will be a properly embedded submanifold of $\Sigma$.
  We will restrict $u$ to the compact subset $G =
  u^{-1}\bigl(U\bigr)$, and denote the boundary components of this
  domain by $\p_M G = u^{-1}\bigl(N\cap U\bigr)$ and $\p_W G =
  u^{-1}\bigl(\p_W U\bigr)$.
  We thus have a holomorphic map
  \begin{equation*}
    \restricted{u}{G}\colon (G, \p G; j) \to
    \bigl(U, \p U; J\bigr)
  \end{equation*}
  with $u\bigl(\p_M G\bigr) \subset N\cap U$ and $u\bigl(\p_W
  G\bigr)\subset \p_W U$.
  The coordinate maps $h_x\colon (z_1,z_2; \bfQ, \bfP) \mapsto
  \RealPart(z_2)$ and $h_y\colon (z_1,z_2; \bfQ, \bfP) \mapsto
  \ImaginaryPart(z_2)$ are harmonic, and it follows by the maximum
  principle that the maximum of $h_x\circ \restricted{u}{G}$ will lie
  for each component of $G$ on the boundary of that component.
  Furthermore the maximum of $h_x\circ \restricted{u}{G}$ cannot lie
  on $\p_W G$, because by our assumption $\restricted{u}{G}$ is
  transverse to $\p_W U$.
  It follows that the maximum of $h_x\circ \restricted{u}{G}$ will be
  a point $z_0 \in \p_M G$; in particular $z_0$ does not lie on one of
  the edges of $G$.
  By the boundary point lemma, either $h_x\circ \restricted{u}{G}$ is
  constant or the outward derivative of this function at $z_0$ must be
  strictly positive.
  On the other hand, the function $h_y\circ \restricted{u}{G}$ is
  equal to $0$ all along the boundary~$\p_M G$ so that the derivatives
  of $h_x\circ \restricted{u}{G}$ and $h_y\circ \restricted{u}{G}$
  vanish at $z_0$ in directions that are tangent to the boundary.
  Using the Cauchy-Riemann equation we see that this implies that the
  derivatives of these two functions at $z_0$ vanish in \emph{every
    direction}, in particular this implies
  that the function $h_x\circ \restricted{u}{G}$ needs to be constant.
  In either case, we have proved that the image of $u$ lies completely
  inside $U$.
\end{proof}

The conclusion of the results in this section is that every curve that
intersects a certain neighborhood of the elliptic singularities lies
completely in this neighborhood and can be explicitly determined.

\subsection{$J$-holomorphic curves close to codimension~$1$
  singularities}

Let $(N, \fF)$ be a submanifold with Legendrian foliation and with
non-empty boundary.
We will show in this section that a boundary component of $N$ lying in
in the singular set of $\fF$ can sometimes exclude that any
holomorphic curve gets close to this component.
This way, the boundary may block any holomorphic disks from escaping
the submanifold~$N$.
The argument is similar to that of the previous section, where we
constructed an almost complex manifold that served as a model for the
neighborhood of the singular set.

\begin{remark}
  We will only be dealing here with the easiest type of singular sets:
  Products of a closed manifold with $\SSS^1$.
  A more general situation has been considered in
  \cite{WeafFillabilityHigherDimension}, where the singular set is
  allowed to be a fiber bundle over the circle.
\end{remark}

Let $T^*F$ be the cotangent bundle of a closed manifold~$F$, choose an
almost complex structure $J_F$ on $T^*F$ for which $F$ is a totally
real submanifold, and let $f_F\colon T^*F \to [0,\infty)$ be the
function constructed in Proposition~\ref{prop: totally real and psh
  functions} that only vanishes along the $0$-section of $T^*F$ and
that is $J_F$-plurisubharmonic close to the $0$-section~$F$.
Define $(W,J)$ as
\begin{equation*}
  W := \CC \times T^*\SSS^1 \times T^*F =
  \bigl\{(x+iy;\varphi, r; \bfQ,\bfP)\bigr\} \;,
\end{equation*}
and let $J$ be the almost complex structure $i \oplus i \oplus J_F$,
where the complex structure on $T^*\SSS^1$ is the one induced from the
identification of $T^*\SSS^1$ and $\CC / (2\pi\ZZ)$ with $\varphi +ir
\sim \varphi + 2\pi +ir$.
The function
\begin{equation*}
  f\colon W \to [0,\infty), \, (x+iy; \varphi, r;\bfQ,\bfP) \mapsto
  \frac{1}{2}\,\bigl(x^2 + y^2\bigr) + \frac{1}{2}\, r^2
  + f_F (\bfQ,\bfP)
\end{equation*}
is $J$-plurisubharmonic on a neighborhood where the values of $\bfP$
are small.
We denote the level set $f^{-1}(1/2)$ by $M$, and note that for small
values of $\bfP$, it is a smooth contact manifold with contact form
\begin{equation*}
  \alpha_M := \restricted{\bigl( x\,dy - y\,dx -  r\,d\varphi
    + d^{J_F}f_F\bigr)}{TM} \;.
\end{equation*}
Let $N$ be the submanifold of $M$ given as the image of the map
\begin{equation*}
  \Phi\colon \SSS^1\times [0,\epsilon) \times F, \,
  (\varphi, r; \bfQ) \mapsto \bigl(\sqrt{1-r^2}; \varphi, r; \bfQ,\0\bigr) \;.
\end{equation*}
It has a Legendrian foliation~$\fF$, because $\Phi^*\alpha_M = -r\,
d\varphi$ that becomes singular exactly at the boundary $\p N =
\bigl\{1\} \times \SSS^1 \times F$.
Our local model will be the subset
\begin{equation*}
  U = \bigl\{(x+iy;\varphi, r;\bfQ,\bfP)\bigm|\,
  x\ge 1-\delta \bigr\} \cap f^{-1}\bigl([0,1/2]\bigr)
\end{equation*}
for sufficiently small~$\delta > 0$.
Clearly $U$ contains $\p N = \sing \bigl(\ker(-r\,d\varphi)\bigr)$.
Furthermore $U$ is compact, because all coordinates are bounded:
Points $(x+iy; \varphi, r;\bfQ,\bfP)$ in $U$ satisfy
\begin{equation*}
  0 \le \frac{1}{2}\, y^2 + \frac{1}{2}\, r^2 + f_F (\bfQ,\bfP)
  = f(x+iy; \varphi, r;\bfQ,\bfP) - \frac{1}{2}\, x^2
  \le 1/2\,\bigl(1 - x^2\bigr) \le \delta\;.
\end{equation*}
We also obtain that if $\delta$ has been chosen small enough, $f$ is
everywhere $J$-plurisubharmonic on $U$.

\begin{remark}
  Note that the construction of the local model also applies in the
  case of contact $3$-manifolds, because $F$ may be just a point.
\end{remark}

We will first exclude existence of holomorphic curves that are
entirely contained in $U$.

\begin{proposition}\label{prop: no curves close to codim1 singularity}
  A $J$-holomorphic map
  \begin{equation*}
    u\colon (\Sigma,\p\Sigma, j) \to (U, N \cap U, J)
  \end{equation*}
  from a compact Riemann surface into $U$, whose boundary is mapped
  into $N\cap U$, must be constant.
\end{proposition}
\begin{proof}
  As in the previous section, we can decompose $u$ as
  $\bigl(u_{\CC\times T^*\SSS^1}, u_{T^*F}\bigr)$ with
  \begin{align*}
    u_{\CC\times T^*\SSS^1}\colon & (\Sigma,\p\Sigma, j) \to
    \Bigl(\CC\times T^*\SSS^1, \bigl\{ (\sqrt{1-r^2}; \varphi,
    r)\bigm|\, \varphi\in \SSS^1, r\in [0,\epsilon)\bigr\},
    i\oplus i\Bigr) \\
    u_{T^*F}\colon & (\Sigma,\p\Sigma, j) \to (T^*F, F, J_F) \;.
  \end{align*}
  Note in particular that the boundary conditions also split in this
  decomposition, so that we obtain two completely uncoupled problems.
  Furthermore, using Corollary~\ref{coro: holomorphic curves in
    cotangent}, it follows that the second map is constant, because
  $f_F$ is a $J_F$ plurisubharmonic function on the considered
  neighborhood.
  To show that $u_{\CC\times T^*\SSS^1}$ is constant, use the harmonic
  function $g\bigl(z;\varphi,r\bigr) = \ImaginaryPart(z)$.
  Since $g\circ u_{\CC\times T^*\SSS^1}$ vanishes along $\p\Sigma$, it
  follows that $g\circ u_{\CC\times T^*\SSS^1}$ has to be zero on the
  whole Riemann surface, and combining this with the Cauchy-Riemann
  equation, it follows that the real part of the $z$-coordinate of
  $u_{\CC\times T^*\SSS^1}$ is equal to a constant $C\in [1-\delta,
  1]$.
  Now that we know that the first coordinate of $u_{\CC\times
    T^*\SSS^1}$ is constant, we see that the boundary of $u_{\CC\times
    T^*\SSS^1}$ has to lie in the circle $\bigl\{ (C; \varphi, +
  \sqrt{1-C^2})\bigm|\, \varphi\in \SSS^1\bigr\} \subset \CC\times
  T^*\SSS^1$.
  This allows us to study only the second coordinate of $u_{\CC\times
    T^*\SSS^1}$ reducing our map to the form
  \begin{equation*}
    u_{T^*\SSS^1} \colon (\Sigma,\p\Sigma, j) \to
    \bigl(T^*\SSS^1, S, i\bigr) \;,
  \end{equation*}
  where $S = \bigl\{ r = + \sqrt{1-C^2} \bigr\}$.
  Using that the map $(r,\varphi) \mapsto r$ is harmonic, and that it
  is constant along the boundary of $\Sigma$, we obtain that the whole
  image of the surface has to lie in the corresponding circle,
  implying with the Cauchy-Riemann equation that $u_{T^*\SSS^1}$ needs
  to be constant.
\end{proof}

Next we will show that holomorphic curves may not enter the domain~$U$
even partially.
Let $(W,J)$ be now a compact almost complex manifold with convex
boundary $M = \p W$, and let $N$ be a submanifold of $M$ with $\p N\ne
\emptyset$.
Assume that $W$ contains a compact subset~$U$ that is identical to the
model neighborhood constructed above such that $M\cap U$, $N\cap U$
and $\restricted{J}{U}$ all agree with the corresponding objects in
the model.

\begin{proposition}\label{lemma: complexStructureBoundaryBLOB}
  If the image of a $J$-holomorphic map
  \begin{equation*}
    u\colon (\Sigma,\p\Sigma, j) \to (W, N, J)
  \end{equation*}
  intersects the neighborhood~$U$, then it will be constant.
\end{proposition}
\begin{proof}
  It suffices to show that the image of $u$ lies inside $U$, because
  we can then apply Proposition~\ref{prop: no curves close to codim1
    singularity}.
  Following the same line of arguments as in the proof of
  Proposition~\ref{prop: no hol curve may enter model neighborhood
    from outside}, one can show that the real part of the first
  coordinate of $u$ needs to be constant.
  We recommend the reader to work out the details as an exercise.
\end{proof}

\begin{remark}
  Note that when the codimension~$1$ singular set lies in the interior
  of the maximally foliated submanifold, one can find under additional
  conditions a family of holomorphic annuli with one boundary
  component on each side of the singular set (see
  \cite{NiederkrugerWendl}).
  The reason why these curves do not appear in the results of this
  section are that we are assuming that all boundary components of the
  holomorphic curves lie locally on one side of the singular set.
\end{remark}

\chapter{Moduli spaces of disks and filling obstructions}
\label{chap: moduli spaces}

\section{The moduli space of holomorphic disks}

Let us assume again that $(W,J)$ is an almost complex manifold, and
that $N\subset W$ is a totally real submanifold.
We want to study the space of maps
\begin{equation*}
  u\colon (\DD^2, \p \DD^2) \to (W, N; J)
\end{equation*}
that are $J$-holomorphic (strictly speaking they are
$(i,J)$-holomorphic), meaning that we want the differential of $u$ to
be complex linear, so that it satisfies at every $z\in \Sigma$ the
equation
\begin{equation*}
  Du_z\cdot i = J\bigl(u(z)\bigr)\cdot Du_z \;.
\end{equation*}
Note that $J$ depends on the point $u(z)$!
A different way to state this equation is by introducing the
Cauchy-Riemann operator
\begin{equation*}
  \bar\partial_J u = J(u)\cdot Du - Du\cdot i \;,
\end{equation*}
and writing $\bar \partial_J u = 0$, so that the space of
$J$-holomorphic maps, we are interested in then becomes
\begin{equation*}
  \widetilde \mM \bigl(\DD^2,  N; J\bigr) =
  \bigl\{u\colon  \DD^2 \to W \bigm|\,
  \text{$\bar\partial_J u = 0$ and $u(\p \DD^2) \subset N$} \bigr\} \;.
\end{equation*}

\begin{remark}
  The situation of holomorphic disks is a bit special compared to the
  one of general holomorphic maps, because all complex structures on
  the disk are equivalent.
  If $\Sigma$ were a smooth compact surface of higher genus, we would
  usually need to study the space of pairs $(u,j)$, where $j$ is a
  complex structure on $\Sigma$, and $u$ is a map $u\colon (\Sigma, \p
  \Sigma) \to (W,N)$ that should be $(j,J)$-holomorphic, that means,
  $J(u)\cdot Du - Du\cdot j = 0$.
  To be a bit more precise, we do not choose pairs $(u,j)$ with
  arbitrary complex structures $j$ on $\Sigma$, but we only allow for
  $j$ a single element in each equivalence class of complex
  structures:
  If $\varphi\colon \Sigma \to \Sigma$ is a diffeomorphism, and $j$ is
  some complex structure, then of course $\varphi^*j$ will generally
  be a complex structure different from $j$, but we usually identify
  all complex structures up to isotopy, and use that the space of
  equivalence classes of complex structures can be represented as a
  smooth finite dimensional manifold (see \cite{Hummel} for a nice
  introduction to this theory).
  Fortunately, these complications are not necessary for holomorphic
  disks (or spheres), and it is sufficient for us to work with the
  standard complex structure~$i$ on $\DD^2$.
\end{remark}

In this section, we want to explain the topological structure of the
space $\widetilde \mM \bigl(\DD^2, N; J\bigr)$ without entering into
too many technical details.
Instead of starting directly with our particular case, we will try to
argue on an intuitive level by considering a finite dimensional
situation that has strong analogies with the problem we are dealing
with.
Let us consider a vector bundle~$E$ of rank~$r$ over a smooth
$n$-manifold~$B$.
Choose a section~$\sigma\colon B\to E$, and let $M= \sigma^{-1}(0)$ be
the set of points at which $\sigma$ intersects the $0$-section.
We would ``expect'' $M$ to be a smooth submanifold of dimension~$\dim
M = n - r$ (if $n-r < 0$, we could hope not to have any intersections
at all); unfortunately, this intuitive expectation might very well be
false.
A sufficient condition under which it holds, is when $\sigma$ is
transverse to the $0$-section, that means, for every $x\in M$, the
tangent space to the $0$-section $T_xB$ in $T_xE$ spans together with
the image $D\sigma\cdot T_xB$ the whole tangent space $T_xE$.
It is well-known that when the transversality condition is initially
not true, it can be achieved by slightly perturbing the section
$\sigma$.
Let us now come again to the Cauchy-Riemann problem.
The role of $B$ will be taken by the space of all maps $u\colon
\bigl(\DD^2, \p \DD^2\bigr) \to (W,N)$, which we will denote by
$\bB\bigl(\DD^2; N\bigr)$.
We do not want to spend any time thinking about the regularity of the
maps and point instead to \cite{McDuffSalamonJHolo} as reference.
It is sufficient for us to observe that the space $\bB\bigl(\DD^2;
N\bigr)$ is a Banach manifold, that means, an infinite dimensional
manifold modeled on a Banach space.
The section~$\sigma$ will be replaced by the Cauchy-Riemann
operator~$\bar\partial_J$, and before pursuing this analogy further,
we want first to specify the target space of this operator.
In fact, $\bar \partial_J$ associates to every map $u\in
\bB\bigl(\DD^2; N\bigr)$ a $1$-form on $\Sigma$ with values in $TW$.
The formal way to state this is that we have for every map $u$ a
vector bundle $u^* TW$ over $\DD^2$, which allows us to construct
\begin{equation*}
  \Hom\bigl(T\DD^2, u^* TW\bigr) \;.
\end{equation*}
The sections in $\Hom\bigl(T\DD^2, u^* TW\bigr)$ form a vector space,
and if we look at all sections for \emph{all} maps~$u$, we obtain a
vector bundle over $\bB\bigl(\DD^2; N\bigr)$, whose fiber over a point
$u$ are all sections in $\Hom\bigl(T\DD^2, u^* TW\bigr)$.
We denote this bundle by $\eE\bigl(\DD^2; N\bigr)$.
The operator~$\bar \partial_J$ associates to every $u$, that means, to
every point of $\bB\bigl(\DD^2; N\bigr)$ an element in
$\eE\bigl(\DD^2; N\bigr)$ so that we can think of $\bar \partial_J$ as
a section in the bundle $\eE\bigl(\DD^2; N\bigr)$.
The $J$-holomorphic maps are the points of $\bB\bigl(\DD^2; N\bigr)$
where the section $\bar \partial_J$ intersects the $0$-section.
In fact, $\bar \partial_J u$ is always anti-holomorphic, because
\begin{equation*}
  J(u)\cdot \bar\partial_J u = - Du - J(u)\cdot Du\cdot i
  = \bigl(Du\cdot i - J(u)\cdot Du\bigr) \cdot i =
  - \bigl(\bar\partial_J u\bigr)\cdot i \;,
\end{equation*}
and for analytical reasons we will only consider sections in $\Hom
\bigl(T\Sigma, u^* TW\bigr)$ taking values in the subbundle
$\overline{\Hom_\CC} \bigl(T\Sigma, u^* TW\bigr)$ of anti-holomorphic
homomorphisms.
We denote the subbundle of sections taking values in
$\overline{\Hom_\CC} \bigl(T\Sigma, u^* TW\bigr)$ by $\bar
\eE_\CC\bigl(\DD^2; N\bigr)$.

\subsection{The expected dimension of $\widetilde\mM\bigl(\DD^2, N;
  J\bigr)$}\label{sec: expected dimension}

The rank of $\bar \eE_\CC\bigl(\DD^2; N\bigr)$ and the dimension of
$\bB\bigl(\DD^2; N\bigr)$ are both infinite, hence we cannot compute
the expected dimension of the solution space
$\widetilde\mM\bigl(\DD^2, N; J\bigr)$ as in the finite dimensional
case, where it was just the difference $\dim M - \rank E$.
Nonetheless we can associate a so called Fredholm index to a
Cauchy-Riemann problem.
We will later give some more details about how the index is actually
defined, for now we just note that it is an integer that determines
the expected dimension of the space $\widetilde\mM\bigl(\DD^2, N;
J\bigr)$.
For a Cauchy-Riemann problem with totally real boundary condition the
index has an easy explicit formula (see for example
\cite[Theorem~C.1.10]{McDuffSalamonJHolo}) that simplifies in our
specific case of holomorphic disks to
\begin{equation}
  \ind_u \bar \partial_J = \frac{1}{2}\, \dim W
  + \mu\bigl(u^* TW, u^* TN\bigr) \;,
  \label{eq: index formula}
\end{equation}
where we have used that the Euler characteristic of a disk is
$\chi\bigl(\DD^2\bigr) = 1$.

\begin{remark}
  We would like to warn the reader that the dimension of a moduli
  space of holomorphic disks or holomorphic spheres tends to increase,
  if we increase the dimension of the symplectic ambient manifold.
  Unfortunately, the opposite is true for a higher genus
  curve~$\Sigma$: The formula above becomes
  \begin{equation*}
    \ind_u \bar \partial_J = \frac{1}{2}\, \chi(\Sigma) \dim W
    + \mu\bigl(u^* TW, u^* TN\bigr) \;,
  \end{equation*}
  and since the Euler characteristic is negative, and it is harder to
  find curves with genus in high dimensional spaces than in lower
  dimensional ones.
\end{remark}

The Maslov index~$\mu$ is an integer that classifies loops of totally
real subspaces up to homotopy:

\begin{definition}
  Let $E_\CC$ be a complex vector bundle over the closed
  $2$-disk~$\DD^2$ and let $E_\RR$ be a totally real subbundle of
  $\restricted{E_\CC}{\p\DD^2}$ defined only over the boundary of the
  disk.
  The \defin{Maslov index $\mu(E_\CC,E_\RR)$} is an integer that is
  computed by trivializing $E_\CC$ over the disk, and choosing a
  continuous frame $A(e^{i\phi}) \in \GL(n,\CC)$ over the boundary
  $\p\DD^2$ representing $E_\RR$.
  We then set
  \begin{equation*}
    \mu(E_\CC,E_\RR) := \deg \frac{\det A^2}{\det (A^*A)} \;,
  \end{equation*}
  where $\deg(f)$ is the degree of a continuous map $f\colon \SSS^1 \to
  \SSS^1$.
\end{definition}

In these notes, we will compute the Maslov index only once, in
Section~\ref{sec: bishop family}, but note that the index $\ind_u
\bar \partial_J$ depends on the holomorphic disk~$u$ in
$\widetilde\mM\bigl(\DD^2, N; J\bigr)$, we are considering; this
should not confuse us however, because it only means that the space of
disks may have different components and the expected dimensions of the
different components do not need to agree.
We will now briefly explain how the index of $\bar \partial_J$ is
defined.
We have a map $\bar \partial_J\colon \bB\bigl(\DD^2; N\bigr) \to \bar
\eE_\CC\bigl(\DD^2; N\bigr)$, and we need to compute the linearization
of $\bar \partial_J$ at a point of $u \in \bB\bigl(\DD^2; N\bigr)$,
that means, we have to compute the differential
\begin{equation*}
  \bar D_J(u)\colon T_u \bB\bigl(\DD^2; N\bigr) \to T_{\bar \partial_J u} \bar
  \eE_\CC\bigl(\DD^2; N\bigr) \;.
\end{equation*}
To find $\bar D_J(u)$, choose a smooth path~$u_t$ of maps in
$\bB\bigl(\DD^2; N\bigr)$ with $u_0 = u$, then we can regard the image
$\bar \partial_J u_t$, and take its derivative with respect to $t$ in
$t = 0$.
If we set $\dot u_0 = \restricted{\frac{d}{dt}}{t = 0} u_t$, this
allows us to obtain a linear operator~$\bar D_J(u)$ by
\begin{equation*}
  \bar D_J(u)\cdot \dot u_0  = \restricted{\frac{d}{dt}}{t = 0}
  \bar \partial_J u_t \;.
\end{equation*}
It is a good exercise to determine the domain and target space of this
operator, and find a way to describe them.
The index of $\bar \partial_J$ at $u$ is defined as
\begin{equation*}
  \ind_u \bar \partial_J := \dim \ker \bar D_J(u) - \dim \coker \bar D_J(u) \;.
\end{equation*}
It is a remarkable fact that the index is finite and determined by
formula~\eqref{eq: index formula} above.
Also note that the index is constant on each connected component of
$\bB\bigl(\DD^2; N\bigr)$.

\subsection{Transversality of the Cauchy-Riemann problem}\label{sec:
  transversality of CR}

Just as in the finite dimensional analogue, it may happen that the
formal dimension we have computed does not correspond to the dimension
we are observing in an actual situation.
In fact, if the section~$\sigma$ (or in our infinite dimensional case,
$\bar \partial_J$) are not transverse to the $0$-section, there is no
reason why $M$ or $\widetilde\mM\bigl(\DD^2, N; J\bigr)$ would need to
be smooth manifolds at all.
On the other hand, if $\sigma$ is transverse to the $0$-section, then
$M = \sigma^{-1}(0)$ is a smooth submanifold of dimension~$\dim M -
\rank E$, and the analogue result is also true for the Cauchy-Riemann
problem:
If $\bar \partial_J$ is at every point of $\widetilde\mM\bigl(\DD^2,
N; J\bigr)$ transverse to $0$ (or said equivalently, if the cokernel
of the linearized operator is trivial for every holomorphic disk),
then $\widetilde\mM\bigl(\DD^2, N; J\bigr)$ will be a smooth manifold
whose dimension is given by the index of $\bar \partial_J$.
In the finite dimensional situation, we can often achieve
transversality by a small perturbation of $\sigma$, but of course,
this might require a subtle analysis of the situation, when we want to
perturb $\sigma$ only within a space of sections satisfying certain
prescribed properties.

\begin{definition}
  Let $u\colon \Sigma \to W$ be a holomorphic map from a Riemann
  surface with or without boundary.
  We call $u$ \defin{somewhere injective}, if there exists a point
  $z\in \Sigma$ with $du_z \ne 0$, and such that $z$ is the only point
  that is mapped by $u$ to $u(z)$, that means,
  \begin{equation*}
    u^{-1}\bigl(u(z)\bigr) = \{z\} \;.
  \end{equation*}
  \index{holomorphic map!somewhere injective}
\end{definition}

We call a holomorphic curve that is not the multiple cover of any
other holomorphic curve a \defin{simple holomorphic curve}.
Closed simple holomorphic curves are somewhere injective,
\cite[Proposition~2.5.1]{McDuffSalamonJHolo}.  \index{holomorphic
  curve!simple}
It is a non-trivial result that by perturbing the almost complex
structure~$J$, we can achieve transversality of the Cauchy-Riemann
operator for every \emph{somewhere injective} disk in $W$ with
boundary in a totally real submanifold~$N$.
We could hope that this theoretical result would be sufficient for us,
because the considered disks are injective along their boundaries, but
we have chosen a very specific almost complex structure in
Section~\ref{sec: local_model}, and perturbing this $J$ would destroy
the results obtained in that section.
Below, we will prove by hand that $\bar \partial_J$ is transverse to
$0$ for the holomorphic disks in our model neighborhood.

\begin{remark}
  Note that often it is not possible to work only with somewhere
  injective holomorphic curves, and perturbing $J$ will in that case
  not be sufficient to obtain transversality for holomorphic curves.
  Sometimes one can work around this problem by requiring that $W$ is
  semi-positive, see Section~\ref{sec: compactness}.
  Unfortunately, there are many situations where this approach won't
  work either, as is the case of SFT, where transversality has been
  one of the most important outstanding technical problems.
\end{remark}

\subsection{The Bishop family}\label{sec: bishop family}

In this section, we will show that the disks that we have found in
Section~\ref{sec: elliptic singularity high dimension local model}
lying in the model neighborhood are regular solutions of the
Cauchy-Riemann problem.
Before starting the actual proof of our claim, we will briefly
recapitulate the situation described in Section~\ref{sec: elliptic
  singularity high dimension local model}.
Let $(W, J)$ be an almost complex manifold of dimension~$2n$ with
boundary that contains a model neighborhood~$U$ of the desired form.
Remember that $U$ was a subset of $\CC^2 \times T^*L$ with almost
complex structure $i\oplus J_L$, that we had a function~$f\colon \CC^2
\times T^*L \to [0,\infty)$ given by
\begin{equation*}
  f(z_1,z_2, \bfQ, \bfP) = \frac{1}{2}\,\bigl(\abs{z_1}^2 + \abs{z_2}^2\bigr)
  + f_L(\bfQ, \bfP) \;,
\end{equation*}
and that the model neighborhood~$U$ was the subset
\begin{equation*}
  U := \bigl\{ (z_1,z_2; \bfQ, \bfP)\bigm|\,
  \RealPart(z_2) \ge 1 - \delta \bigr\} \cap
  f^{-1}\bigl([0,1/2]\bigr) \;.
\end{equation*}
The totally real manifold~$N$ is the image of the map
\begin{equation*}
  (z; \bfQ) \in \DD_\epsilon^2 \times L \mapsto
  \Bigl(z, \sqrt{1-\abs{z}^2}; \bfQ, \0\Bigr)\subset \p U \;.
\end{equation*}
For every pair~$(s,\bfQ) \in [1-\delta, 1)\times L$, we find a
holomorphic map of the form
\begin{equation*}
  \begin{split}
    u_{s,\bfQ}\colon & \bigl(\DD^2,\p \DD^2\bigr) \to U \\
    & z \mapsto \bigl(C_s z, s; \bfQ, \0)
  \end{split}
\end{equation*}
with $C_s = \sqrt{1 - s^2}$.
We call this map a \defin{(parametrized) Bishop disk}, and we call the
collection of these disks, the \defin{Bishop family}.
Sometimes we will not be precise about whether the disks are
parametrized or not, and whether we speak about disks with or without
a marked point (see Section~\ref{sec: moduli space marked point}), but
we hope that in each situation it will be clear what is meant.
To check that a given Bishop disk~$u_{s,\bfQ}$ is regular, we will
first compute the index of the linearized Cauchy-Riemann operator that
gives us the expected dimension for the space of holomorphic disks
containing the Bishop family.
Note that the observed dimension is $1 + \dim L + 3 = 1 + (n-2) + 3 =
n + 2$.
The first part, $1 + \dim L$ corresponds to the $s$- and
$\bfQ$-parameters of the family; the three corresponds to the
dimension of the group of Möbius transformations acting on the complex
unit disk:
If $u_{s,\bfQ}$ is a Bishop disk, and if $\varphi\colon \DD^2 \to
\DD^2$ is a Möbius transformation, then of course $u_{s,\bfQ}\circ
\varphi$ will also be a holomorphic map with admissible boundary
condition.
On the other hand we showed in Corollary~\ref{cor: hol curves in
  neighborhood elliptic sing high dim} that every holomorphic disk
that lies in $U$ is up to a Möbius transformation one of the Bishop
disks.
For the index computations, it suffices by Section~\ref{sec: expected
  dimension} to trivialize the bundle $E_\CC := u_{s,\bfQ}^* TW$ over
$\DD^2$, and study the topology of the totally real subbundle $E_\RR =
u_{s,\bfQ}^* TN$ over $\p\DD^2$.
Before starting any concrete computations, we will significantly
simplify the setup by choosing a particular chart:
Note that the $T^*L$-part of a Bishop disk~$u_{s,\bfQ}$ is constant,
we can hence choose a chart diffeomorphic to $\RR^{2n-4} =
\{(x_1,\dotsc,x_{n-2}; y_1,\dotsc,y_{n-2})\}$ for $T^*L$ with the
properties
\begin{itemize}
\item the point~$(\bfQ, \0)$ corresponds to the origin,
\item the almost complex structure~$J_L$ is represented at the origin
  by the standard~$i$,
\item the intersections of the $0$-section~$L$ with the chart
  corresponds to the subspace $(x_1,\dotsc,x_{n-2}; 0,\dotsc,0)$.
\end{itemize}
In the chosen chart, we write~$u_{s,\bfQ}$ as
\begin{equation*}
  u_{s,\bfQ}(z) = \bigl(C_s z, s; 0,\dotsc, 0\bigr)
  \in \CC^2 \times \RR^{2n-4}
\end{equation*}
with $C_s = \sqrt{1-s^2}$.
By our assumption, the complex structure on the second factor is at
the origin of $\RR^{2n-4}$ equal to $i$, and there is then a direct
identification of $u_{s,\bfQ}^* TW$ with $\CC^2 \times \CC^{n-2}$.
The submanifold~$N$ corresponds in the chart to
\begin{equation*}
  \bigl\{(z_1,z_2;x_1,\dotsc,x_{n-2}, 0,\dotsc,0) \in
  \CC^2\times \RR^{2n-4}
  \bigm|\, \ImaginaryPart{z_2} = 0,\,
  \abs{z_1}^2 + \abs{z_2}^2  = 1 \bigr\}\;.
\end{equation*}
The boundary of $u_{s,\bfQ}$ is given by $e^{i\varphi} \mapsto
\bigl(\sqrt{1-s^2}\,e^{i\varphi}, s; 0,\dotsc, 0\bigr)$, and the
tangent space of $TN$ over this loop is spanned over $\RR$ by the
vector fields
\begin{equation*}
  \Bigl(ie^{i\varphi}, 0; 0, \dotsc, 0\Bigr),\,
  \Bigl(-\frac{s}{\sqrt{1-s^2}} \, e^{i\varphi}, 1; 0, \dotsc, 0\Bigr),\, 
  \bigl(0,0; 1,0,\dotsc,0\bigr),\dotsc,
  \bigl(0,0; 0,\dotsc,0,1,0,\dotsc,0\bigr) \;.
\end{equation*}
We can now easily compute the Maslov index $\mu(E_\CC,E_\RR)$ as
\begin{equation*}
  \deg \frac{\det A^2}{\det (A^*A)} = \deg \frac{-e^{2i\varphi}}{1} = 2 \;,
\end{equation*}
where $A$ is the matrix composed by the vector fields given above.
Hence we obtain for the index
\begin{equation*}
  \ind_u \bar \partial_J = \frac{1}{2}\, \dim W
  + \mu\bigl(u_{s,\bfQ}^* TW, u_{s.\bfQ}^* TN\bigr) = n + 2 \;,
\end{equation*}
which corresponds to the observed dimension computed above.
We will now show that the linearized operator~$\bar D_J$ is
surjective.
We do not do this directly, but we compute instead the dimension of
its kernel, and show that it is equal (and not larger than) the
Fredholm index.
From the definition of the index
\begin{equation*}
  \ind_u \bar \partial_J := \ker \bar D_J(u) - \coker \bar D_J(u) \;,
\end{equation*}
we see that the cokernel needs to be trivial, and this way the
surjectivity result follows.
We now compute the linearized Cauchy-Riemann operator at a Bishop
disk~$u_{s,\bfQ}$.
Let $v_t$ be a smooth family of maps
\begin{equation*}
  v_t\colon \bigl(\DD^2, \p \DD^2\bigr) \to (U, N)
\end{equation*}
with $v_0 = u_{s,\bfQ}$ (think of each $v_t$ as a smooth map, but for
an analytically correct study, we would need to allow here for Sobolev
maps).
In this chart, we can write the family~$v_t$ as
\begin{equation*}
  v_t(z) = \bigl(z_1(z, t), z_2(z, t); \bfx(z,t), \bfy(z,t)\bigr)
  \in \CC^2 \times \RR^{2n-4}\;,
\end{equation*}
where we have set $\bfx(z,t) = \bigl(x_1(z, t), \dotsc, x_{n-2}(z,
t))$ and $\bfy(z,t) = \bigl(y_1(z, t), \dotsc, y_{n-2}(z, t))$, and we
require that the boundary of each of the $v_t$ has to lie in $N$.
When we now take the derivative of $v_t$ with respect to $t$ at $t=0$,
we obtain a vector in $T_{u_{s,\bfQ}} \bB$ that is represented by a
map
\begin{equation*}
  \dot v_0\colon \DD^2 \to \CC^2 \times \RR^{2\,(n-2)}, z\mapsto
  \bigl(\dot z_1(z), \dot z_2(z); \dot \bfx(z), \dot \bfy(z)\bigr)
\end{equation*}
with boundary conditions $\dot \bfy (z) = \0$ and $\ImaginaryPart \dot
z_2(z) = 0$ for every $z\in \p \DD^2$.
Furthermore taking the derivative of $\abs{z_1(z,t)}^2 +
\abs{z_2(z,t)}^2 = 1$ for every $z\in \p \DD^2$ with respect to $t$,
we obtain the condition
\begin{equation*}
  \bar z_1(z, 0) \cdot \dot z_1(z) + z_1(z, 0) \cdot \dot {\bar z}_1(z)
  + \bar z_2(z, 0) \cdot \dot z_2(z) + z_2(z, 0) \cdot \dot {\bar z}_2(z)
  = 0 \;,
\end{equation*}
which simplifies by using the explicit form of $\bigl(z_1(z, 0),
z_2(z, 0)\bigr)$ to
\begin{equation*}
  C_s \bar z \cdot \dot z_1(z) + C_s z \cdot \dot {\bar z}_1(z)
  + s  \dot z_2(z) + s  \dot {\bar z}_2(z)
  = 0
\end{equation*}
for every $z \in \p \DD^2$.
The linearization of the Cauchy-Riemann operator~$\bar\partial_J$ at
$u_{s,\bfQ}$  given by
\begin{equation*}
  \bar D_J \cdot \dot v_0 := \restricted{\frac{d}{dt}}{t=0} \bar \partial_J v_s
\end{equation*}
decomposes into the $\CC^2$-part
\begin{equation*}
   \bigl( i d\dot z_1 - d\dot z_1 i,\, i d\dot z_2 - d\dot z_2 i \bigr)
\end{equation*}
and the $\RR^{2\,(n-2)}$-part
\begin{equation*}
  \restricted{\frac{d}{dt}}{t=0} \Bigl( J_L\bigl(\bfx(z,t), \bfy(z,t)\bigr) \cdot
  \bigl(d \bfx(z,t), d \bfy(z,t)\bigr)
  - \bigl(d \bfx(z,t)\cdot i, d \bfy(z,t)\cdot i\bigr)  \Bigr) \;.
\end{equation*}
The second part can be significantly simplified by using first the
product rule, and applying then that $\bfx(z,0) = \0$ and
$\bfy(z,0)=\0$ are constant so that their differentials vanish.
We obtain then
\begin{equation*}
  J_L\bigl(\0, \0\bigr) \cdot \bigl(d \dot\bfx, d\dot\bfy\bigr)
  - \bigl(d \dot \bfx\cdot i, d \dot \bfy\cdot i\bigr) \;,
\end{equation*}
and using that $J_L(\0, \0) = i$, it finally reduces to
\begin{equation*}
  \bigl( d \dot \bfy - d\dot \bfx \cdot i,
  -  d\dot \bfx  - d\dot \bfy \cdot i \bigr) \;.
\end{equation*}
We have shown that linearized Cauchy-Riemann operator simplifies for
all coordinates to the standard Cauchy-Riemann operator, so that if
$\dot v_0(z) = \bigl(\dot z_1(z), \dot z_2(z);\dot \bfx(z), \dot
\bfy(z)\bigr)$ lies in the kernel of $\bar D_J$ then the coordinate
functions $\dot z_1(z), \dot z_2(z)$ and $\dot \bfx(z) + i \dot
\bfy(z)$ need all to be holomorphic in the classical sense.
Now using the boundary conditions, we easily deduce that $\dot\bfy(z)$
needs to vanish, because it is a harmonic function, and it takes both
maximum and minimum on $\p\DD^2$.
A direct consequence of $\dot\bfy \equiv \0$ and the Cauchy-Riemann
equation is that $\dot\bfx(z)$ will be everywhere constant.
We get the analogous result for the function $\dot z_2(z)$, so that we
can write
\begin{equation*}
  \dot v_0(z) = \bigl(\dot z_1(z), \dot s;\, \dot \bfQ_0, \0\bigr) \;,
\end{equation*}
where $\dot s$ is a real constant, and $\dot \bfQ_0$ is a fixed vector
in $\RR^{2\,(n-2)}$, and we only need to still understand the
holomorphic function $\dot z_1(z)$.
The boundary condition for $\dot z_1(z)$ is $\bar z \cdot \dot z_1(z)
+ z \cdot \dot {\bar z}_1(z) = - \frac{2\,s \dot s}{C_s}$ for every
$z\in \p\DD^2$.
Using that the function $\dot z_1(z)$ is holomorphic, we can write it
as power series in the form
\begin{equation*}
  \dot z_1(z) = \sum_{k=0}^\infty a_k\, z^k
\end{equation*}
and we get at  $e^{i\varphi}\in \p\DD^2$
\begin{equation*}
  \dot z_1\bigl(e^{i\varphi}\bigr) =
  \sum_{k=0}^\infty a_k\, e^{ik\varphi} \;.
\end{equation*}
Plugging these series into the equation of the boundary condition, we
find
\begin{equation*}
  e^{-i\varphi} \cdot \sum_{k=0}^\infty a_k\, e^{ik\varphi}
  + e^{i\varphi} \cdot \sum_{k=0}^\infty \bar a_k\, e^{-ik\varphi}
  = - \frac{2\,s \dot s}{C_s}
\end{equation*}
so that
\begin{equation*}
  \sum_{k=0}^\infty \bigl(a_k\, e^{(k-1)\,i\varphi}
  + \bar a_k\, e^{-(k-1)\, i \varphi} \bigr)
  = - \frac{2\,s \dot s}{C_s}
\end{equation*}
and by comparing coefficients we see that
\begin{equation*}
  a_1 + \bar a_1
  = - \frac{2\,s \dot s}{C_s}, \quad
  a_0
  + \bar a_2 = 0, \quad
  a_k = 0 \text{ for all $k\ge 3$}.
\end{equation*}
This means that the three (real) parameters we can choose freely are
$z_0$ and $\ImaginaryPart z_1$.
Concluding, we have found that the dimension of the kernel of $\bar
D_J$ is equal to $3 + 1 + n-2 = n + 2$ which corresponds to the
Fredholm index of our problem.
Thus there is no need to perturb $J$ on the neighborhood of the Bishop
family to obtain regularity.

\begin{corollary}\label{cor: existence adapted almost complex}
  Let $(W, \omega)$ be a compact symplectic manifold that is a weak
  symplectic filling of a contact manifold $(M, \xi)$.
  Suppose that $N$ is either a \LOB or a \BLOB in $M$, then we can
  choose close to the binding and to the boundary of $N$ the almost
  complex structure described in the previous sections, and extend it
  to an almost complex structure~$J$ that is tamed by $\omega$, whose
  bundle of complex tangencies along $M$ is $\xi$ and that makes $M$
  $J$-convex.
  By a generic perturbation away from the binding and the boundary of
  $N$, we can achieve that all somewhere injective holomorphic curves
  become regular.
  We call a $J$ with these properties an \defin{almost complex
    structure adapted to $N$}.  \index{almost complex
    structure!adapted to a \LOB/\BLOB}
\end{corollary}

The argument in the proof of the corollary above is that the Bishop
disks are already regular, and that all other simple holomorphic
curves have to lie outside the neighborhood where we require an
explicit form for $J$.
Thus it suffices to perturb outside these domains to obtain regularity
for every other simple curve.

\section{The moduli space of holomorphic disks with a marked
  point}\label{sec: moduli space marked point}

Until now, we only have studied the space of certain $J$-holomorphic
\emph{maps}
\begin{equation*}
  \widetilde \mM \bigl(\DD^2,  N; J\bigr) =
  \bigl\{u\colon  \DD^2 \to W \bigm|\,
  \text{$\bar\partial_J u = 0$ and $u(\p \DD^2) \subset N$} \bigr\} \;,
\end{equation*}
but many maps correspond to different parametrizations of the same
geometric disk.
To get rid of this ambiguity (and to obtain compactness), we quotient
the space of maps by the biholomorphic reparametrizations of the unit
disk, that means, by the Möbius transformations, but we will also add
a marked point $z_0\in \DD^2$ to preserve the structure of the
geometric disk.
To simplify the notation, we will also omit the almost complex
structure~$J$ in $\widetilde \mM \bigl(\DD^2, N\bigr)$.
From now on let 
\begin{equation*}
  \widetilde \mM \bigl(\DD^2,  N; z_0\bigr) =
  \bigl\{(u, z_0) \bigm|\,
  \text{$z_0\in \DD^2$, $\bar\partial_J u = 0$ and
    $u(\p \DD^2) \subset N$} \bigr\}
  = \widetilde \mM \bigl(\DD^2, N\bigr) \times \DD^2
\end{equation*}
be the space of holomorphic maps together with a special point $z_0\in
\DD^2$ that will be called the \defin{marked point}. \index{marked point}
The \defin{moduli space} we are interested in is the space of
equivalence classes
\begin{equation*}
  \mM \bigl(\DD^2,  N; z_0\bigr) =
  \widetilde\mM \bigl(\DD^2,  N; z_0\bigr) / \sim
\end{equation*}
where we identify two elements~$(u,z_0)$ and $(u',z_0')$, if and only
if there is a biholomorphism $\varphi\colon \DD^2 \to \DD^2$ such that
$u = u'\circ \varphi^{-1}$ and $z_0 = \varphi(z_0')$.
The map $(u,z) \mapsto u(z)$ descends to a well defined map
\begin{equation*}
  \begin{split}
    \ev\colon \mM \bigl(\DD^2, N; z_0\bigr) & \to W \\
    [u,z_0] & \mapsto u(z_0)
  \end{split}
\end{equation*}
on the moduli space, which we call the \defin{evaluation
  map}. \index{map!evaluation}
Let $N$ be a \LOB or a \BLOB, and assume that $B_0$ is one of the
components of the binding of $N$.
Since this is the only situation, we are really interested in in these
notes, we introduce the notation $\widetilde \mM_0 \bigl(\DD^2,
N\bigr)$ for the connected component in $\widetilde \mM \bigl(\DD^2,
N\bigr)$ that contains the Bishop family around $B_0$.
When adding a marked point, we write $\widetilde \mM_0 \bigl(\DD^2, N;
z_0\bigr)$ and $\mM_0 \bigl(\DD^2, N; z_0\bigr)$ for the corresponding
subspaces.
It is easy to see that $\mM_0 \bigl(\DD^2, N; z_0\bigr)$ is a smooth
(non-compact) manifold with boundary.
Note first that $\widetilde \mM_0 \bigl(\DD^2, N; z_0\bigr)$ is also a
smooth and non-compact manifold with boundary:
If $J$ is regular, we know that $\widetilde \mM_0 \bigl(\DD^2,
N\bigr)$ is a smooth manifold, and so the boundary of the product
manifold~$\widetilde \mM_0 \bigl(\DD^2, N; z_0\bigr)$ is
\begin{equation*}
  \p \widetilde \mM_0 \bigl(\DD^2,  N; z_0\bigr) =
  \widetilde \mM_0 \bigl(\DD^2,  N\bigr) \times \p \DD^2 \;.
\end{equation*}
Passing to the quotient preserves this structure, because the boundary
of the maps in $\widetilde \mM_0 \bigl(\DD^2, N\bigr)$ intersects each
of the pages of the open book exactly once (this is a consequence of
Corollary~\ref{cor: boundary of curves transverse to Legendrian
  foliation} and Section~\ref{sec: elliptic singularity high dimension
  local model} of Chapter~\ref{chapter: J-holomorphic curves}), and
hence each of the disks is injective along its boundary.
The only Möbius transformation that preserves the boundary pointwise
is the identity, hence it follows that the group of Möbius
transformations acts smoothly, freely and properly on $\widetilde
\mM_0 \bigl(\DD^2, N; z_0\bigr)$, and hence the quotient will be a
smooth manifold of dimension
\begin{equation*}
  \dim \mM_0 \bigl(\DD^2, N; z_0\bigr) =
  \dim \widetilde \mM_0 \bigl(\DD^2, N; z_0\bigr) - 3
  =  \ind_u \bar \partial_J + 2 - 3 = n + 1 \;.
\end{equation*}
As before the points on the boundary of $\mM_0 \bigl(\DD^2, N;
z_0\bigr)$ are the classes $[u,z]$ with $z\in \p\DD^2$.
It is also clear that the evaluation map $\ev_{z_0} \colon \mM_0
\bigl(\DD^2, N; z_0\bigr) \to W$ is smooth.
Remember that the Bishop disks contract to points as they approach the
binding~$B_0$.
We will show that we incorporate $B_0$ into the moduli space~$\mM_0
\bigl(\DD^2, N; z_0\bigr)$ and that the resulting space carries a
natural smooth structure that corresponds to the intuitive picture of
disks collapsing to one point.
The neighborhood of the binding~$B_0$ in $W$ is diffeomorphic to the
model
\begin{equation*}
  U = \bigl\{(z_1,z_2; \bfQ, \bfP)\in \CC^2 \times T^*B_0
  \bigm|\, \RealPart (z_2) > 1 - \delta\bigr\}
  \cap h^{-1}\bigl((-\infty, 1/2]\bigr)
\end{equation*}
for small $\delta > 0$ with the function
\begin{equation*}
  h(z_1,z_2) = \frac{1}{2}\,
  \bigl(\abs{z_1}^2 + \abs{z_2}^2\bigr) + f_{B_0}(\bfQ, \bfP) \;,
\end{equation*}
see Section~\ref{sec: elliptic singularity high dimension local
  model}.
The content of Proposition~\ref{prop: no hol curve may enter model
  neighborhood from outside} and of Corollary~\ref{cor: hol curves in
  neighborhood elliptic sing high dim} is that for every point
\begin{equation*}
  (z, s; \bfQ_0, \0) \in U
\end{equation*}
with $s \in (1-\delta, 1)$ and $\bfQ_0$ in the $0$-section of
$T^*B_0$,
\begin{itemize}
\item there is up to a Möbius transformation a unique holomorphic map
  $u\in \widetilde \mM_0 \bigl(\DD^2, N\bigr)$ containing that point
  in its image, and
\item $\widetilde \mM \bigl(\DD^2, N\bigr)$ does not contain any
  holomorphic maps whose image is not entirely contained in $U \cap
  (\CC\times \RR \times B_0)$.
\end{itemize}
As a result, it follows that $V = \ev_{z_0}^{-1} (U)$ is an open
subset of $\mM_0 \bigl(\DD^2, N; z_0\bigr)$, and that the restriction
of the evaluation map
\begin{equation*}
  \restricted{\ev_{z_0}}{V} \colon V \to U
\end{equation*}
is a diffeomorphism onto $U \cap \bigl(\CC \times (1-\delta, 1) \times
B_0\bigr)$.
The closure of this subset is the smooth submanifold
\begin{equation*}
  U \cap \bigl(\CC \times \RR \times B_0\bigr) \;,
\end{equation*}
which we obtain by including the binding $\{0\} \times \{1\} \times
B_0$ of $N$.
Using the evaluation map, we can identify $V$ with its image in $U$,
and this way glue $B_0$ to the moduli space~$\mM_0 \bigl(\DD^2, N;
z_0\bigr)$.
The new space is also a smooth manifold with boundary, and the
evaluation map extends to it, and is a diffeomorphism onto its image
in $U$ so that we can effectively identify $U$ with a subset of the
moduli space.
In particular, it follows that $B_0$ is a submanifold that is of
codimension~$2$ in the boundary of the moduli space.
The aim of the next section will consist in studying the Gromov
compactification of $\mM_0 \bigl(\DD^2, N; z_0\bigr)$.

\section{Compactness}\label{sec: compactness}

Gromov compactness is a result that describes the possible limits of a
sequence of holomorphic curves, and ensures under certain conditions
that every such sequence contains a converging subsequence.
In the limit, a given sequence of holomorphic curves may break into
several components, called \defin{bubbles}, each of which is again a
holomorphic curve.
We will not describe in detail what ``convergence'' in this sense
really means, but we only sketch the idea:
The holomorphic curves in a moduli space can be represented by
holomorphic maps, and in the optimal case, one could hope that by
choosing for each curve in the given sequence a suitable
representative, we might have uniform convergence of the maps, and
this way we would find the limit of the sequence as a proper
holomorphic curve.
Unfortunately, this is usually wrong, but it might be true that for
the correct choice of parametrization we have convergence on
subdomains.
Choosing different reparametrizations, we then obtain convergence on
different domains, and each such domain gives then rise to a bubble,
that means, a holomorphic curve that represents one component of the
Gromov limit.

\begin{theorem}[Gromov compactness]\label{thm: Gromov compactness}
  Let $(W, J)$ be a compact almost complex manifold (with or without
  boundary), and assume that $J$ is tamed by a symplectic form~$\omega$.
  Let $L$ be a compact totally real submanifold.
  Choose a sequence of $J$-holomorphic maps $u_k\colon (\DD^2, \p
  \DD^2) \to (W,L)$ whose \defin{$\omega$-energy}
  \begin{equation*}
    E(u_k) := \int_{\DD^2} u_k^* \omega
  \end{equation*}
  is bounded by a constant~$C > 0$.
  Then there is a subsequence of $\bigl(u_{k_l}\bigr)_l$ that
  converges in the Gromov sense to a bubble tree composed of a finite
  family of non-constant holomorphic disks $u_\infty^{(1)}, \dotsc,
  u_\infty^{(K)}$ whose boundary lies in $L$, and a finite family of
  non-constant holomorphic spheres $v_\infty^{(1)}, \dotsc,
  v_\infty^{(K')}$.
  The total energy is preserved so that
  \begin{equation*}
    \lim_{l\to \infty} E\bigl(u_{k_l}\bigr)
    = \sum_{j=1}^K E \bigl(u_\infty^{(j)}\bigr) +
    \sum_{j=1}^{K'} E \bigl(v_\infty^{(j)}\bigr) \;.
  \end{equation*}
  If each of the disks~$u_k$ is equipped with a marked point~$z_k \in
  \DD^2$, then after possibly reducing to a another subsequence, there
  is a marked point~$z_\infty$ on one of the components of the bubble
  tree such that $\lim_k z_k = z_\infty$ in a suitable sense.
\end{theorem}

The $\omega$-energy is fundamental in the proof of the compactness
theorem to limit the number of possible bubbles:
By \cite[Proposition~4.1.4]{McDuffSalamonJHolo}, there exists in the
situation of Theorem~\ref{thm: Gromov compactness} a constant~$\hbar >
0$ that bounds the energy of every holomorphic sphere or every
holomorphic disk $u_k\colon (\DD^2, \p \DD^2) \to (W,L)$ from below.
Since every bubble needs to have at least an $\hbar$-quantum of
energy, and since the total energy of the curves in the sequence is
bounded by $C$, the limit curve will never break into more than
$C/\hbar$ bubbles (the upper bound of the energy is also used to make
sure that each bubble is a compact surface).
We will show in the rest of this section that we can apply Gromov
compactness to sequences of holomorphic disks lying in the moduli
space $\mM_0 \bigl(\DD^2, N\bigr)$ studied in the previous section,
and how we can incorporate these limits into $\mM_0 \bigl(\DD^2, N;
z_0\bigr)$ to construct the compactification $\overline{\mM}_0
\bigl(\DD^2, N; z_0\bigr)$.

\begin{proposition}\label{prop: upper energy bound}
  Let $N$ be a \LOB or a \BLOB in the contact boundary~$(M, \xi)$ of a
  symplectic filling~$(W,\omega)$, and assume that we find a contact
  form~$\alpha$ for $\xi$ such that $\restricted{\omega}{TN} =
  \restricted{d\alpha}{TN}$.
  There is a global energy bound~$C > 0$ for all holomorphic disks in
  $\widetilde \mM_0 \bigl(\DD^2, N\bigr)$.
\end{proposition}
\begin{proof}
  There is a slight complication in our proof, because we may not
  assume that $\omega$ is globally exact, which would allow us to
  obtain the energy of a holomorphic disk by integrating over the
  boundary of the disk.
  To prove the desired statement, proceed as follows: Let $u\colon
  (\DD^2, \p \DD^2) \to (W, N)$ be any element in $\widetilde \mM_0
  \bigl(\DD^2, N\bigr)$.
  By our assumption, there exists a smooth path of maps~$u_t$ that
  starts at the constant map $u_0(z) \equiv b_0 \in B_0$ in the
  binding and ends at the chosen map $u_1 = u$.
  This family of disks may be interpreted as a map from the $3$-ball
  into $W$.
  The boundary consists of the image of $u_1$, and the union of the
  boundary of all disks $\restricted{u_t}{\p \DD^2}$.
  Using Stokes' theorem, we get
  \begin{equation*}
    0 = \int_{[0,1]\times \DD^2} u_t^*d\omega
    = \int_{\DD^2} u_1^*\omega +
    \int_{[0,1] \times \p\DD^2} u_t^*\omega
  \end{equation*}
  so that $E(u) = - \int_{[0,1] \times \p\DD^2} u_t^*\omega$.
  By our assumption, we have a contact form on the contact
  boundary~$M$ for which $\restricted{\omega}{TN} =
  \restricted{d\alpha}{TN}$, so that using Stokes' theorem a second
  time (and that $u_0(z) = b_0$) we get
  \begin{equation*}
    E(u) = \int_{\p\DD^2} u^*\alpha \;.
  \end{equation*}
  The Legendrian foliation on $N$ is an open book whose pages are
  fibers of a fibration $\vartheta\colon N\setminus B \to \SSS^1$.
  Hence the $1$-form~$d\vartheta$ and $\restricted{\alpha}{TN}$ have
  the same kernel, and it follows that there exists a smooth
  function~$f\colon N \to [0,\infty)$ such that
  \begin{equation*}
    \restricted{\alpha}{TN} = f\, d\vartheta \;.
  \end{equation*}
  The function~$f$ vanishes on the binding and on the boundary of a
  \BLOB, and $f$ is hence bounded on $N$ so that we define $C :=
  2\pi\, \max_{x\in N} \abs{f(x)}$.
  Using that the boundary of $u$ intersects every leaf of the open
  book exactly once, we obtain for the energy of $u$ the estimate
  \begin{equation*}
    E(u) = \int_{\p\DD^2} u^*\alpha  \le \max_{x\in N} \abs{f(x)}
    \,  \int_{\p\DD^2}  u^*d\vartheta \le 2\pi \, \max_{x\in N}
    \abs{f(x)} = C \;.  \qedhere
  \end{equation*}
\end{proof}

With the given energy bound, we obtain now Gromov compactness in form
of the following corollary.

\begin{corollary}\label{coro: bubbling for BLOBs}
  Let $N$ be a \LOB or a \BLOB in the contact boundary~$(M, \xi)$ of a
  symplectic filling~$(W,\omega)$, and assume that we find a contact
  form~$\alpha$ for $\xi$ such that $\restricted{\omega}{TN} =
  \restricted{d\alpha}{TN}$.
  Let $(u_k)_k$ be a sequence of holomorphic maps in $\widetilde \mM_0
  \bigl(\DD^2, N\bigr)$.
  There exists a subsequence $\bigl(u_{k_l}\bigr)_l$ that converges either
  \begin{itemize}
  \item uniformly up to reparametrizations of the domain to a
    $J$-holomorphic map $u_\infty \in \widetilde \mM_0 \bigl(\DD^2,
    N\bigr)$,
  \item to a constant disk $u_\infty(z) \equiv b_0$ lying in the
    binding of $N$,
  \item or to a bubble tree composed of a single holomorphic disk
    $u_\infty\colon (\DD^2, \p \DD^2) \to (W, N)$ and a finite family
    of non-constant holomorphic spheres $v_1, \dotsc, v_j$ with $j\ge
    1$.
  \end{itemize}
\end{corollary}
\begin{proof}
  We will apply Theorem~\ref{thm: Gromov compactness}.
  The submanifold~$N$ is not totally real along the binding~$B$ and
  $\p N$, but we simply remove a small open neighborhood of both sets.
  By Proposition~\ref{prop: no curves close to codim1 singularity},
  none of the holomorphic disks~$u_k$ may get close to $\p N$, and by
  Proposition~\ref{prop: no hol curve may enter model neighborhood
    from outside} we know precisely how the curves look like that
  intersect a neighborhood of $B$.
  If we find disks in $(u_k)_k$ that get arbitrarily close to the
  binding of $N$, then using that $B$ is compact, we may choose a
  subsequence that converges to a single point in the binding.
  If $(u_k)_k$ stays at finite distance from $B$, we may assume that
  the neighborhood, we have removed from $N$ is so small that the
  holomorphic disks we are studying all lie inside.
  If the sequence $(u_k)_k$ does not contain any subsequence that can
  be reparametrized in such a way that it converges to a single
  non-constant disk $u_\infty$, we use Gromov compactness to obtain a
  subsequence that splits into a finite collection of holomorphic
  spheres and disks.
  But as a consequence from Corollary~\ref{cor: boundary of curves
    transverse to Legendrian foliation}, we see that non-constant
  holomorphic disks attached to $N$ need to intersect the pages of the
  open book transversely in positive direction.
  A sequence of holomorphic disks that intersects every page of the
  open book exactly once, cannot split into several disks intersecting
  pages several times.
  In particular possible bubble trees contain by this argument a
  single disk in its limit.
\end{proof}

Above, we have obtained compactness for a sequence of disks, but we
would like to understand how these limits can be incorporated into the
moduli space.
Adding the bubble trees to the space of parametrized maps does not
give rise to a valid topology, because the bubbling phenomenon can
only be understood by using different reparametrizations of the disk
to recover all components of the bubble tree.
We will denote the compactification of $\mM_0 \bigl(\DD^2, N;
z_0\bigr)$ by $\overline{\mM}_0 \bigl(\DD^2, N; z_0\bigr)$.
For us, it is not necessary to understand the topology of
$\overline{\mM}_0 \bigl(\DD^2, N; z_0\bigr)$ in detail, but it will be
sufficient to see that bubbling is a ``codimension~$2$ phenomenon''.
In fact, it is not the topology of the moduli space itself we are
interested in, but our aim is to obtain information about the
symplectic manifold.
For this we want to make sure that the image under the evaluation map
of all bubble trees that appear in the limit, that means, of
$\overline{\mM}_0 \bigl(\DD^2, N; z_0\bigr) \setminus \mM_0
\bigl(\DD^2, N; z_0\bigr)$ is contained in the image of a smooth map
defined on a finite union of manifolds each of dimension at most
\begin{equation*}
  \dim \mM_0 \bigl(\DD^2, N; z_0\bigr)  - 2 \; .
\end{equation*}
For this to be true, we need to impose additional conditions for $(W,
\omega)$.

\begin{definition}
  A $2n$-dimensional symplectic manifold $(M,\omega)$ is called
  \begin{itemize}
  \item \defin{symplectically aspherical}, \index{symplectic
      structure!aspherical} if $\omega([A])$ vanishes for every
    $A\in\pi_2(M)$.
  \item It is called \defin{semipositive} \index{symplectic
      structure!semipositive} if every $A\in\pi_2(M)$ with
    $\omega([A]) > 0$ and $c_1(A)\ge 3 - n$ has non-negative Chern
    number.
  \end{itemize}

  Note that every symplectic $4$- or $6$-manifold is obviously
  semipositive.
\end{definition}

In a symplectically aspherical manifold no $J$-holomorphic spheres
exist, because their energy would be zero.
So in particular they may not appear in any bubble tree and
Corollary~\ref{coro: bubbling for BLOBs} implies in our situation that
every sequence of holomorphic disks contains a subsequence that either
collapses into the binding or that converges to a single disk in
$\mM_0 \bigl(\DD^2, N; z_0\bigr)$.
Using the results of Section~\ref{sec: moduli space marked point}, we
obtain the following corollary.

\begin{corollary}\label{cor: smooth compactification aspherical case}
  Let $(W,\omega)$ be a compact symplectically aspherical manifold
  that is a weak filling of a contact manifold~$(M, \xi)$.
  Let $N$ be a \LOB or a \BLOB in $M$, and assume that we find a
  contact form for $\xi$ such that $\restricted{\omega}{TN} =
  \restricted{d\alpha}{TN}$.
  Choose an almost complex structure~$J$ that is adapted to $N$ (as in
  Corollary~\ref{cor: existence adapted almost complex}).
  Then the compactification of the moduli space $\mM_0 \bigl(\DD^2, N;
  z_0\bigr)$ is a smooth compact manifold
  \begin{equation*}
    \overline{\mM}_0 \bigl(\DD^2, N;
    z_0\bigr) =  \mM_0 \bigl(\DD^2, N;
    z_0\bigr) \cup \bigl(\text{binding of $N$}\bigr)
  \end{equation*}
  with boundary.
  The binding of $N$ is a submanifold of codimension~$2$ in the
  boundary $\p \overline{\mM}_0 \bigl(\DD^2, N; z_0\bigr)$.
\end{corollary}

The condition of asphericity is very strong, and we will obtain more
general results by studying instead semipositive manifolds.
The important point here is that a generic almost complex structure
only ensure transversality for somewhere injective holomorphic curves,
see Section~\ref{sec: transversality of CR}.
Even though the holomorphic disks in $\mM_0 \bigl(\DD^2, N; z_0\bigr)$
are simple, it could happen that once the disks bubble, there appear
spheres that are multiple covers.
For these, we cannot guarantee transversality, and hence we cannot
directly predict if the compactification of $\mM_0 \bigl(\DD^2, N;
z_0\bigr)$ consists of adding ``codimension~$2$ strata'' or if we will
be forced to include too many bubble trees
Still, we know that every sphere that is not simple is the multiple
cover of a simple one (by the Riemann-Hurwitz formula a sphere can
only multiply cover a sphere), we can hence compute the dimension of
the moduli space of the underlying simple spheres, and use this
information as an upper bound for the dimension of the spheres that
appear in the bubble tree.
Let $v\colon \SSS^2 \to W$ be a holomorphic sphere that is a $k$-fold
cover of a sphere~$\widetilde v$ representing a homology class $[v]$
and $[\widetilde v] \in H_2(W, \ZZ)$ respectively with $[v] = k
[\widetilde v]$ and with $\omega\bigl([\widetilde v]\bigr) > 0$.
The expected dimension of the space of maps containing $v$ is by an
index formula
\begin{equation*}
  \ind_v \bar \partial_J = 2n + 2\,c_1\bigl([v]\bigr)
  = 2n + 2k\,c_1\bigl([\widetilde v]\bigr) \;.
\end{equation*}
The space of biholomorphisms of $\SSS^2$ has dimension~$6$, and hence
the expected dimension of the moduli space of unparametrized spheres
that contain $[v]$ is $\ind_v \bar \partial_J - 6 = 2\, (n-3) +
2k\,c_1([\widetilde v])$.
As we explained above and in Section~\ref{sec: transversality of CR},
this expected dimension does not correspond in general to the observed
dimension of the bubble trees, instead we study the expected dimension
of the underlying simple spheres.
The dimension of the space containing $\widetilde v$ is given by
$\ind_{\widetilde v} \bar \partial_J - 6 = 2\, (n-3) +
2\,c_1\bigl([\widetilde v]\bigr)$.
If $c_1\bigl([\widetilde v]\bigr) < n-3$, then the expected dimension
will be negative, and since we obtain regularity of all simple
holomorphic curves by choosing a generic almost complex structure, it
follows that the moduli space containing $\widetilde v$ is generically
empty.
As a consequence bubble trees appearing as limits do not contain any
component that is the $k$-fold cover of a simple sphere representing
the homology class~$[\widetilde v]$.
If $c_1\bigl([\widetilde v]\bigr) \ge n-3$, the definition of
semipositivity implies that $c_1\bigl([\widetilde v]\bigr) \ge 0$.
When we compare the expected dimension of the moduli space containing
$v$ with the one of the underlying disk~$\widetilde v$, we observe
that $\ind_v \bar \partial_J - 6 = 2\, (n-3) +
2k\,c_1\bigl([\widetilde v]\bigr) \ge 2\, (n-3) +
2\,c_1\bigl([\widetilde v]\bigr) = \ind_{\widetilde v} \bar \partial_J
- 6$.
Consider now the image in $W$ of all spheres in the moduli space of
$v$ that are $k$-fold multiple covers of some simple sphere.
Their image is contained in the image of the simple spheres lying in
the same moduli space as $\widetilde v$.
The dimension of this second moduli space is smaller or equal than the
expected dimension of the initial moduli space containing $v$, and
even though we cannot ensure regularity for $v$, we have an estimate
on the dimension of the subset containing all singular spheres.
The following result allows us to find the desired bound for the
dimension of the image of complete bubble trees.

\begin{proposition}\label{propo: dimension bubble trees}
  Assume that $(W, \omega)$ is semipositive.
  To compactify the moduli space $\mM_0(W, N, z_0)$, one has to add
  bubbled curves.
  We find a finite set of manifolds $X_1, \dotsc, X_N$ with $\dim X_j
  \le \dim \mM_0(W, N, z_0) - 2$ and smooth maps $f_j\colon X_j \to W$
  such that the image of the bubbled curves under the evaluation
  map~$\ev_{z_0}$ is contained in
  \begin{equation*}
    \cup f_j(X_j) \;.
  \end{equation*}
  When we consider instead the compactification of the boundary $\p
  \mM_0(W, N, z_0)$, that means the space of holomorphic disks with a
  marked point on the boundary of the disk only, then we obtain the
  analogue result, only that the manifolds $X_1, \dotsc, X_N$ have
  dimension $\dim X_j \le \dim \p \mM_0(W, N, z_0) - 2 = \dim \mM_0(W,
  N, z_0) - 3$.
\end{proposition}
\begin{proof}
  The standard way to treat bubbled curves consists in considering
  them as elements in a bubble tree: Here such a tree is composed by a
  simple holomorphic disk $u_0\colon\, (\DD^2, \SSS^1) \to (W, N)$ and
  holomorphic spheres $u_1,\dotsc,u_{k^\prime}\colon \,\SSS^2\to W$.
  These holomorphic curves are connected to each other in a certain
  way.  We formalize this relation by saying that the holomorphic
  curves are vertices in a tree, i.e.\ in a connected graph without
  cycles.
  We denote the edges of this graph by $\{u_i,u_j\}$, $0\le i<j\le
  k^\prime$.
  Now we assign to any edge two nodal points $z_{ij}$ and $z_{ji}$,
  the first one in the domain of the bubble $u_i$, the other one in
  the domain of $u_j$, and we require that $\ev_{z_{ij}}(u_i) =
  \ev_{z_{ji}}(u_j)$.
  For technical reasons, we also require nodal points on each
  holomorphic curve to be pairwise distinct.
  To include into the theory, trees with more than one bubble
  connected at the same point to a holomorphic curve, we add ``ghost
  bubbles''.
  These are constant holomorphic spheres inserted at the point where
  several bubbles are joined to a single curve.
  Now all the links at that point are opened and reattached at the
  ghost bubble.
  Ghost bubbles are the only constant holomorphic spheres we allow in
  a bubble tree.
  The aim is to give a manifold structure to these bubble trees, but
  unfortunately this is not always possible, when multiply covered
  spheres appear in the bubble tree.
  Instead, we note that the image of every bubble tree is equal to the
  image of a simple bubble tree, that means, to a tree, where every
  holomorphic sphere is simple and any two spheres have different
  image.
  Since we are only interested in the image of the evaluation map on
  the bubble trees, it is for our purposes equivalent to consider the
  simple bubble tree instead of the original one.
  The disk~$u_0$ is always simple, and does not need to be replaced by
  another simple curve.
  Let $u_0,u_1,\dotsc,u_{k^\prime}$ be the holomorphic curves
  composing the original bubble tree, and let $A_i\in H_2(W)$ be the
  homology class represented by the holomorphic sphere~$u_i$.
  The simple tree is composed by $u_0,v_1,\dotsc,v_{k}$ such that for
  every $u_j$ there is a bubble sphere~$v_{i_j}$ with equal image
  \begin{align*}
    u_j(\SSS^2) &= v_{i_j}(\SSS^2)
  \end{align*}
  and in particular $A_j = m_jB_{i_j}$, where $B_{i_j} = [v_{i_j}]\in
  H_2(W)$ and $m_j\ge 1$ is an integer.
  Write also $\mathbf{A}$ for the sum $\sum_{j=1}^{k^\prime}A_j$ and
  $\mathbf{B}$ for the sum $\sum_{i=1}^{k}B_i$.
  Below we will compute the dimension of this simple bubble tree.
  The initial bubble tree $u_0, u_1,\dotsc, u_{k^\prime}$ is the limit
  of a sequence in the moduli space $\mathcal{M}_0(W, N, z_0)$.
  Hence the connected sum $u_\infty := u_0\conSum \dotsm \conSum
  u_{k^\prime}$ is, as element of $\pi_2(W, N)$, homotopic to a disk
  $u$ in the bishop family, and the Maslov indeces
  \begin{equation*}
    \mu(u) :=
    \mu(u^*TW,u^*T N)
    \,\text{ and }\,
    \mu(u_\infty) :=
    \mu(u_\infty^*TW,u_\infty^*TN)
  \end{equation*}
  have to be equal.
  With the standard rules for the Maslov index (see for example
  \cite[Appendix~C.3]{McDuffSalamonJHolo}), we obtain
  \begin{align*}
    2 = \mu(u) = \mu(u_\infty) = \mu(u_0) + \sum_{j=1}^{k^\prime}
    2c_1([u_j]) = \mu(u_0) + 2c_1(\mathbf{A})\;.
  \end{align*}
  The dimension of the unconnected set of holomorphic curves
  $\widetilde \mM_{[u_0]}(W, N, z_0) \times \prod_{j=1}^k \widetilde
  \mM_{B_j}(W)$ for the simple bubble tree is
  \begin{align*}
    \bigl(n + \mu(u_0)\bigr) + \sum_{j=1}^k 2\,\bigl(n+c_1(B_j)\bigr)
    & = n + 2 -
    2c_1(A) + 2nk + \sum_{j=1}^k 2c_1(B_j) \\
    & = n + 2 + 2nk + 2\,
    \bigl(c_1(\mathbf{B})-c_1(\mathbf{A})\bigr)\;.
  \end{align*}
  In the next step, we want to consider the subset of connected
  bubbles, i.e.\ we choose a total of $k$ pairs of nodal points, which
  then have to be pairwise equal under the evaluation map.
  The nodal points span a manifold
  \begin{align*}
    Z(2k) &\subset \bigl\{(1,\dotsc,2k) \to \DD^2 \amalg \SSS^2 \amalg
    \dotsm \amalg \SSS^2\bigr\}
  \end{align*}
  of dimension~$4k$.
  The dimension reduction comes from requiring that the evaluation map
  \begin{align*}
    \ev\colon\, \widetilde \mM_{[u_0]}(W, N, z_0) \times \prod_{j=1}^k
    \widetilde \mM_{B_j}(W) \times Z(2k) \to W^{2k}
  \end{align*}
  sends pairs of nodal points to the same image in the symplectic
  manifold.
  By regularity and transversality of the evaluation map to the
  diagonal submanifold $\triangle(k)\hookrightarrow W^{2k}$, the
  dimension of the space of holomorphic curves is reduced by the
  codimension of $\triangle(k)$, which is $2nk$.
  As a last step, we have to add the marked point~$z_0$ used for the
  evaluation map~$\ev_{z_0}$, this way increasing the dimension by
  $2$, and then we take the quotient by the automorphism group to
  obtain the moduli space.
  The dimension of the automorphism group is $6k + 3$.
  Hence the dimension of the total moduli space is
  \begin{multline*}
    n + 2 + 2nk + 2\, \bigl(c_1(\mathbf{B})-c_1(\mathbf{A})\bigr) +
    4k - 2nk + 2 - (6k +3) \\
    = n +1 - 2k + 2\, \bigl(c_1(\mathbf{B})-c_1(\mathbf{A})\bigr) \le
    n + 1 - 2k \;.
  \end{multline*}
  The inequality holds because by the assumption of semipositivity,
  all the Chern classes are non-negative on holomorphic spheres, and
  all coefficients~$n_j$ in the difference $c_1(\mathbf{B}) -
  c_1(\mathbf{A}) = \sum_j c_1(B_j) - \sum_i c_1(A_i) = \sum_j
  c_1(B_j) - \sum_i m_i c_1(B_{j_i}) = \sum_j n_j c_1(B_j)$ are
  non-positive integers.
  The computations for the disks in $\p \mM_0\bigl(\DD^2, N;
  z_0\bigr)$ only differs by the requirement that the marked point
  needs to lie on the boundary of the disk~$u_0$ instead of moving
  freely on the bubble tree.
  Instead of having two degrees of freedom for this choice, we thus
  only add one extra dimension.
\end{proof}

\section{Proof of the non-fillability Theorem~\ref{thm:
    non-fillability}}

\begin{non_fillability_theorem}
  Let $(M, \xi)$ be a contact manifold that contains a \BLOB~$N$, then
  $M$ does not admit any semi-positive weak symplectic filling $(W,
  \omega)$ for which $\restricted{\omega}{TN}$ is exact.
\end{non_fillability_theorem}

Assume there were a semi-positive symplectic filling $(W, \omega)$ for
which $\restricted{\omega}{TN}$ is exact.
Let $\alpha$ be a positive contact form for $\xi$.
By Proposition~\ref{prop: cohomolog collar}, we can extend $(W,
\omega)$ with a collar in such a way that we have
$\restricted{\omega}{TN} = \restricted{d\alpha}{TN}$, which will allow
us to use the energy estimates of the previous section.
Now we choose an almost complex structure that is adapted to the
\BLOB~$N$ as in Corollary~\ref{cor: existence adapted almost complex},
and we will study the moduli space $\mM_0 \bigl(\DD^2, N; z_0\bigr)$
defined in Section~\ref{sec: moduli space marked point} of holomorphic
disks with one marked point lying in the same component as the Bishop
family around a chosen component~$B_0$ of the binding of $N$.
Trace a smooth path $\gamma\colon [0,1] \to N$ that starts at
$\gamma(0) \in B_0$ and ends on the boundary~$\p N$.
Assume further that $\gamma$ is a regular curve, and that it
intersects the binding and $\p N$ only on the endpoints of $[0,1]$.
We want to select a $1$-dimensional moduli space in $\mM_0
\bigl(\DD^2, N; z_0\bigr)$ by only considering
\begin{equation*}
  \mM^\gamma := \ev_{z_0}^{-1}\bigl(\gamma(I)\bigr) \;.
\end{equation*}
It will be important for us that $\gamma(I)$ does not intersect the
image of any bubble trees in $\overline{\mM}_0 \bigl(\DD^2, N;
z_0\bigr) \setminus \mM_0 \bigl(\DD^2, N; z_0\bigr)$.
By Proposition~\ref{propo: dimension bubble trees}, we have that the
bubble trees in $\overline{\p \mM}_0 \bigl(\DD^2, N; z_0\bigr)$ lie in
the image of a finite union of smooth maps defined on manifolds of
dimension $\dim \p \mM_0 \bigl(\DD^2, N; z_0\bigr) - 2 = \dim N - 2$.
The subset $N \setminus \ev_{z_0}(\text{bubble trees})$ is connected
and we can deform $\gamma$ keeping the endpoints fixed so that it does
not intersect any of the bubble trees.
For a small perturbation of $J$ (away from the binding and the
boundary of $N$), we can make sure that the evaluation map $\ev_{z_0}$
is transverse to the path~$\gamma(I)$.
If the perturbed~$J$ lies sufficiently close to the old one, then
$\gamma$ will also not intersect any bubble trees for this new $J$,
for otherwise we could choose a sequence of almost complex
structures~$J_k$ converging to the unperturbed $J$ such that for
everyone there existed a bubble tree~$v_k$ intersecting $\gamma$.
We would find a converging subsequence of $v_k$ yielding a bubble
tree~$v_\infty$ for the unperturbed almost complex structure
intersecting $\gamma$, which contradicts our assumption.
It follows that $\mM^\gamma$ is a collection of compact
$1$-dimensional submanifolds of $\p \mM_0 \bigl(\DD^2, N; z_0\bigr)$.
There is one component in $\mM^\gamma$, which we will denote by
$\mM_0^\gamma$ that contains the Bishop disks that intersect
$\gamma\bigl([0,\epsilon)\bigr)$.
We know that the Bishop disks are the only disks close to the binding,
and hence it follows that $\mM_0^\gamma$ cannot be a loop that closes,
but must be instead a closed interval.
The first endpoint of $\mM_0^\gamma$ is the constant disk with image
$\gamma(0) \in B_0$, and we will deduce a contradiction by showing
that no holomorphic disk can be the second endpoint of $\mM_0^\gamma$.
By Proposition~\ref{prop: no curves close to codim1 singularity},
there is a small neighborhood of $\p N$ that cannot be entered by any
holomorphic disk.
By our construction the endpoint of $\mM_0^\gamma$ cannot be any
bubble tree either.
It follows that the endpoint needs to be a regular disk $[u, z_0] \in
\p \mM_0 \bigl(\DD^2, N; z_0\bigr)$ for which the boundary of $u$ lies
in $N \setminus \bigl(\p N \cup B\bigr)$ and whose interior points
cannot touch $\p W$ either, because we are assuming that the boundary
of $W$ is convex.
It follows that this regular disk cannot really be the endpoint of
$\mM_0^\gamma$, because the evaluation map~$\ev_{z_0}$ will also be
transverse to $\gamma$ at $[u, z_0]$ so that we can extend
$\mM_0^\gamma$ further.
This leads to a contradiction that shows that the assumption that the
boundary of $W$ is everywhere convex cannot hold.

\section{Proof of Theorem~\ref{thm: homology LOB}}

For the proof, we first recall the definition of the degree of a map.

\begin{definition}
  Let $X$ and $Y$ be closed oriented $n$-manifolds.
  The \defin{degree} of a map $f\colon X \to Y$ is the integer $d =
  \deg(f)$ such that
  \begin{equation*}
    f_\# [X] = d\cdot [Y] \;,
  \end{equation*}
  where $[X] \in H_n(X,\ZZ)$ and $[Y] \in H_n(Y,\ZZ)$ are the
  fundamental classes of the corresponding manifolds.
  When the manifolds $X$ and $Y$ are not orientable, we define the
  degree to be an element of $\ZZ_2$ using the same formula, where the
  fundamental classes are elements in $H_n(X,\ZZ_2)$ and
  $H_n(Y,\ZZ_2)$.  \index{degree of a map}
\end{definition}
Note that we can easily compute the degree of a smooth map~$f$ between
smooth manifolds by considering a regular value~$y_0\in Y$ of $f$
(which by Sard's theorem exist in abundance), and adding
\begin{equation*}
  \deg f = \sum_{x\in f^{-1}(y_0)}  \sign Df_x\;,
\end{equation*}
where the point~$x$ contributes to the sum with $+1$, whenever $Df_x$
is orientation preserving, and contributes with $-1$ otherwise.
In case the manifolds are not orientable, we can always add $+1$ in
the above formula, but need to take sum over $\ZZ_2$.

\begin{homology_LOB_theorem}
  Let $(M, \xi)$ be a contact manifold of dimension $(2n+1)$ that
  contains a \LOB~$N$.
  If $M$ has a weak symplectic filling $(W,\omega)$ that is
  symplectically aspherical, and for which $\restricted{\omega}{TN}$
  is exact, then it follows that $N$ represents a trivial class in
  $H_{n+1}(W, \ZZ_2)$.
  If the first and second Stiefel-Whitney classes~$w_1(N)$ and
  $w_2(N)$ vanish, then we obtain that $[N]$ must be a trivial class
  in $H_{n+1}(W, \ZZ)$.
\end{homology_LOB_theorem}

Using Proposition~\ref{prop: cohomolog collar} we can assume that
$\restricted{\omega}{TN} = \restricted{d\alpha}{TN}$ for a chosen
contact form~$\alpha$.
Choose an almost complex structure~$J$ on $W$ that is adapted to the
\LOB~$N$, and let $\mM_0 \bigl(\DD^2, N; z_0\bigr)$ be the moduli
space of holomorphic disks with one marked point lying in the same
component as the Bishop family around a chosen component of the
binding of $N$.
Since $W$ is symplectically aspherical, we obtain by
Corollary~\ref{cor: smooth compactification aspherical case} that
$\overline{\mM}_0 \bigl(\DD^2, N; z_0\bigr)$ is a compact smooth
manifold with boundary.
It was shown in \cite{GeorgievaThesis} that $\mM_0 \bigl(\DD^2, N;
z_0\bigr)$ is orientable if the first and second Stiefel-Whitney
classes of $N\setminus B$ vanish.
With our assumptions this is the case, because $w_j\bigl(N\setminus
B\bigr) = \restricted{w_j(N)}{(N\setminus B)}$.
If $\mM_0 \bigl(\DD^2, N; z_0\bigr)$ is orientable then
$\overline{\mM}_0 \bigl(\DD^2, N; z_0\bigr)$ will also be orientable:
If there were an orientation reversing loop~$\gamma$ in the
compactified moduli space (which is obtained from $\mM_0 \bigl(\DD^2,
N; z_0\bigr)$ by gluing in $B$ as codimension~$3$ submanifold), then
due to the large codimension we could easily push $\gamma$ completely
into the regular part of the moduli space, where it would still need
to be orientation reversing.
It follows that the boundary $\p \overline{\mM}_0 \bigl(\DD^2, N;
z_0\bigr)$ is also homologically a boundary (either with $\ZZ$- or
$\ZZ_2$-coefficients depending on the orientability of the considered
spaces).
Denote the restriction of the evaluation map
\begin{equation*}
  \restricted{\ev_{z_0}}{\p \overline{\mM}_0 \bigl(\DD^2, N;
    z_0\bigr)}\colon \p \overline{\mM}_0 \bigl(\DD^2, N;
  z_0\bigr) \to N \;,
\end{equation*}
by $f$.
We know that close to the binding every point is covered by a unique
Bishop disk, this implies by the remarks made above that the
degree~$\deg(f)$ needs to be $\pm 1$.
We have the following obvious equation
\begin{equation*}
  \ev_{z_0} \circ \,
  \iota_{\p \overline{\mM}}
  = \iota_N \circ f \;,
\end{equation*}
where $\iota_{\p \overline{\mM}}$ denotes the embedding of $\p
\overline{\mM}_0\bigl(\DD^2, N; z_0\bigr)$ in
$\overline{\mM}_0\bigl(\DD^2, N; z_0\bigr)$ and $\iota_N$ the
embedding of $N$ in $W$.
The homomorphism induced by $\iota_{\p \overline{\mM}}$ is the trivial
map on the $(n+1)$-st homology group, so that the left side of the
equation gives rise to the $0$-map
\begin{equation*}
  H_{n+1}\Bigl(\p \overline{\mM}_0\bigl(\DD^2, N; z_0\bigr), R\Bigr)
  \to   H_{n+1}\bigl( W, R\bigr)
\end{equation*}
with $R$ being either $\ZZ$ or $\ZZ_2$.
Since $f_\#$ is $\pm$ identity, it follows that $\iota_N$ has to
induce the trivial map on homology, which implies that $N$ is
homologically trivial in $W$.

\bibliographystyle{amsalpha}

\providecommand{\bysame}{\leavevmode\hbox to3em{\hrulefill}\thinspace}
\providecommand{\MR}{\relax\ifhmode\unskip\space\fi MR }
\providecommand{\MRhref}[2]{%
  \href{http://www.ams.org/mathscinet-getitem?mr=#1}{#2}
}
\providecommand{\href}[2]{#2}

\printindex

\end{document}